\documentclass[paper=a4]{article}
\usepackage[english]{babel}
\usepackage{amsmath,amsfonts,amsthm, amssymb}
\usepackage[margin=1in]{geometry}
\usepackage{graphicx}
\usepackage{float}
\usepackage{graphicx}
\usepackage[]{xcolor}
\usepackage{bm}
\usepackage{tikz}
\usepackage{xifthen}
\usepackage{mathrsfs}
\usepackage{enumitem}
\usepackage{cite}
\usepackage{tikzsymbols}

\usetikzlibrary{arrows,decorations.pathmorphing,backgrounds,positioning,fit,petri, arrows.meta, matrix, calc, decorations.markings,shapes}
\tikzset{>=latex}

\DeclareMathOperator{\Int}{Int}
\DeclareMathOperator{\ext}{Ext}
\def\fv{\ensuremath\mathcal{F}}

\theoremstyle{plain}
\newtheorem{thm}{Theorem}[section]
\newtheorem{cor}[thm]{Corollary}
\newtheorem{lem}[thm]{Lemma}

\newtheorem{rem}[thm]{Remark}

\theoremstyle{definition}

\setlist[itemize]{noitemsep}
\setlist[enumerate]{noitemsep}

\title{\Huge \textbf{The degree-diameter problem for plane graphs with pentagonal faces}}
\author{Brandon Du Preez\\
Laboratory for Discrete Mathematics and Theoretical Computer Science\\
Department of Mathematics and Applied Mathematics\\
University of Cape Town\\
\texttt{brandon.dupreez@uct.ac.za}}
\date{January 2024}

\begin{document}
\maketitle

\begin{abstract}
	\noindent The degree-diameter problem consists of finding the maximum number of vertices $n$ of a graph with diameter $d$ and maximum degree $\Delta$.
	This problem is well studied, and has been solved for plane graphs of low diameter in which every face is bounded by a 3-cycle (triangulations), and plane graphs in which every face is bounded by a 4-cycle (quadrangulations). 
	In this paper, we solve the degree diameter problem for plane graphs of diameter 3 in which every face is bounded by a 5-cycle (pentagulations).
	We prove that if $\Delta \geq 8$, then $n \leq 3\Delta - 1$ for such graphs.
	This bound is sharp for $\Delta$ odd.
\end{abstract}

\section{Introduction}

The well known \textbf{degree-diameter problem} asks for the maximum order $n(\Delta, d)$ of a graph with maximum degree $\Delta$ and diameter $d$. 
By considering a $\Delta$-regular breadth-first tree, we easily obtain a trivial upper bound on $n(\Delta, d)$ known as the \textbf{Moore Bound}. 
The graphs attaining this bound for $\Delta > 2$ and $d > 1$ are called \textbf{Moore Graphs}, and there are only finitely many of them: the Peterson graph, the Hoffman-Singleton graph and --- conjecturally --- some `missing' Moore graph(s) of diameter 2 and maximum degree 57 \cite{mgd23_hoffman_60, fmg_bannai_73, mg_damerell_73}.
These Moore graphs are not planar, and the upper bounds attained on $n(\Delta, d)$ for planar graphs are substantially smaller than the Moore bound.

In \cite{lpg2_hell_93}, Hell and Seyffarth solve exactly the degree-diameter problem for planar graphs of diameter 2, showing that $n(\Delta, 2) = \frac{3}{2}\Delta + 1$ for such graphs. 
Further results for planar graphs are obtained in \cite{lpgmd_fellows_95} by Fellows, Hell and Seyffarth. 
They give bounds on $n(\Delta, 3)$ and show that for each fixed diameter $d$, there exists some constant $c$ such that $n(\Delta, d) \leq c\Delta^{\lfloor d/2 \rfloor}$.
For planar graphs with even diameter and large maximum degree, the degree-diameter problem was solved exactly by Tishchenko in \cite{mspg_tischenko_12}.

Further refining the problem, we consider plane graphs in which every face is bounded by a circuit or cycle of the same length $\rho$.
When $\rho = 3$, we obtain the well-studied maximal planar graphs / triangulations.
Seyffarth proved in \cite{mpgd2_seyffarth_89} that a triangulation of diameter 2 and $\Delta \geq 8$ has at most $\frac{3}{2}\Delta + 1$ vertices, and this bound is sharp. 
Interestingly, the bound is the same as the bound for general planar graphs obtained in \cite{lpg2_hell_93} --- and this fact is critical to the proof in \cite{lpg2_hell_93}.
Plane graphs with $\rho = 4$ are maximal planar bipartite graphs --- or quadrangulations.
For quadrangulations, Dalf\'{o}, Huemer and Salas prove the sharp bounds $n(\Delta, 2) = \Delta + 2$, $n(\Delta, 3) = 3\Delta - 1$ when $\Delta$ is odd and $n(\Delta, 3) = 3\Delta - 2$ when $\Delta$ is even \cite{ddpmpbg_dalfo_16}.
They also give approximate bounds on $n(\Delta, d)$ for quadrangulations with $d>3$ and $\Delta$ large.
The author considered plane graphs in which $\rho$ was (almost) as large as possible for fixed diameter $d$, obtaining the following sharp bounds: $n(\Delta,d) = 2d + 1$ when $\rho = 2d+1$ and $n(\Delta, d) = \Delta(d-1) + 2$ when $\rho = 2d$ \cite{pglf_dupreez_21}. 
The extremal graphs were also characterized.

The degree-diameter problem has been studied for graphs and triangulations on other surfaces -- see \cite{nmbg_siagiova_04, egdt_knor_97} --- as well as for highly structured graphs such as triangular and honeycomb networks \cite{ddphn_holub_14, ddptn_holub_15}.
In recent work, the problem was tackled for outerplanar graphs \cite{ddpop_dankelmann_17}, and a generalization of the degree-diameter problem is the subject of the 2022 paper \cite{nocmb_tuite_22}.
For a comprehensive overview of the degree-diameter problem, see Miller and \v{S}ir\'{a}\v{n}'s survey \cite{mgbsurvey_miller_13}. 
For the early version of this work, and related research, see \cite{dpg_dupreez_21}.

We call a plane graph in which every face is bounded by a cycle of length 5 a \textbf{pentagulation}. 
In this paper, we prove that $n(\Delta, 3) = 3\Delta - 1$ for pentagulations with $\Delta \geq 8$.
The paper begins with definitions and basic lemmas in Section \ref{sec:prelim}.
In Section \ref{sec:no_3_cycles}, we prove that a diameter 3 pentagulation is triangle-free.
The structure of 4-cycles and separating 5-cycles is explored in Section \ref{sec:sep_cycles}.
Section \ref{sec:disloc_4_cycle} introduces the notion of dislocated 4-cycles --- a concept central to the proof of the main theorem.
The proof that $n(\Delta, 3) \leq 3\Delta - 1$ for pentagulations is very involved, so we split it into three sections.
Section \ref{sec:bounding_1} considers pentagulations with a pair of dislocated 4-cycles, Section \ref{sec:bounding_2} proves the bound for pentagulations with a 4-cycle, but no dislocated pair, and Section \ref{sec:bounding_3} proves that a diameter 3 pentagulation with $\Delta \geq 8$ contains at least one 4-cycle, and gives examples to show the bound is sharp for $\Delta$ odd.
We conclude and give some further questions in Section \ref{sec:conclusion}.

\section{Preliminaries}
\label{sec:prelim}

For most definitions used, see \cite{gt_diestel_05}.
The distance between vertices $u$ and $v$ is denoted $\bm{d(u,v)}$, and for a set of vertices or subgraph $S$, let $\bm{d(u,S)} = \min\{d(u,w) : w\in S \}$.
Let $\bm{N_i(v)}$ be the set of vertices at distance $i$ from $v$.
A cycle $C$ in a plane graph $G$ partitions the plane into an \textbf{interior} bounded region denoted $\bm{\Int(C)}$, an \textbf{exterior} unbounded region $\bm{\ext(C)}$, and the cycle $C$ itself.
Denote $\bm{\Int[C]} = \Int(C) \cup C$, and $\bm{\ext[C]} = \ext(C)\cup C$.
If both $\Int(C)$ and $\ext(C)$ contain vertices, then $C$ is a \textbf{Jordan separating cycle}.
Consider a subgraph $H$ of a graph $G$. 
A \textbf{chord} of $H$ in $G$ is an edge $uv$ such that $u,v \in V(G)$ and $uv \in E(G) - E(H)$.
The \textbf{girth} of a graph is the length of its shortest cycle.

It is well known that a plane graph is 2-connected if and only if each face is bounded by a cycle, so all pentagulations are 2-connected.
The following remark is easy to prove, and follows immediately from some simple lemmas in \cite{pglf_dupreez_21}.

\begin{rem}
	Let $G$ be a pentagulation of diameter 3, and $C$ a cycle of $G$. 
	If $C$ is a Jordan separating cycle, then $C$ dominates its interior, or dominates its exterior.
	Further, if $C$ has length 3 or 4, then it is a Jordan separating cycle. 
	\label{rem:cycle_sep}
\end{rem}

\begin{lem}
	Every cycle of length 6 or 7 in a pentagulation is a Jordan separating cycle.
	\label{lem_slightly_long_cycle}
\end{lem}

\begin{proof}
	Let $C$ be a cycle of length 6 or 7 in a pentagulation $G$. 
	The cycle $C$ does not bound any face of $G$, so its interior either contains a vertex, or has some chord $e$. 
	Since the length of $C$ is at most 7, $C\cup \{e\}$ induces a cycle of length 3 or 4.
	Applying Remark \ref{rem:cycle_sep}, we see that $\Int(C)$ contains some vertex.
	Similarly, $\ext(C)$ contains a vertex.
\end{proof}

For a cycle $C$ of length 5 in a pentagulation $G$, there are three distinct possibilities:
\begin{enumerate}[topsep=-\parskip]
	\item The cycle $C$ Jordan separates $G$,
	\item $C$ is a face-cycle that separates $G$, but necessarily does not Jordan separate $G$,
	\item $C$ is a face-cycle that does not separate $G$.
\end{enumerate}

\section{There are no 3-cycles}
\label{sec:no_3_cycles}

The following lemmas show that no 3-cycle in a pentagulation dominates its interior (or exterior).
We phrase our lemmas in terms of cycle interiors, but the same results are easily seen to hold for exteriors.

\begin{lem}
	Let $G$ be a pentagulation.
	If $C$ is a 3-cycle in $G$, then no single vertex of $C$ dominates the interior of $C$.
	\label{lem_5_3_1}
\end{lem}
\begin{proof}
	For the sake of contradiction, let $C:v_1, v_2, v_3$ be a 3-cycle, the interior of which is dominated by the single vertex $v_1$.
	Choose $C$ to be minimal, so there is no 3-cycle $C'$ such that $v_1$ dominates the interior of $C'$, and for which $\Int(C')\subset \Int(C)$. 
	By Remark \ref{rem:cycle_sep}, the cycle $C$ Jordan separates $G$, so there is some vertex $u \in \Int(C)$. 
	By assumption, $u$ and $v_1$ are adjacent. 
	As $G$ is a pentagulation, and thus 2-connected, the vertex $u$ has some neighbor other than $v_1$ in $\Int[C]$.
	This neighbor is not $v_2$, as then $v_1, v_2, u$ is a 3-cycle, contradicting the minimality of $C$. 
	Similarly, $u$ and $v_3$ are not adjacent. (see Figure \ref{fig:5_proof1}). 
	
	\begin{figure}[ht]
\centering
\begin{tikzpicture}
[scale = 0.8, inner sep=0.8mm, 
vertex/.style={circle,thick,draw},
dvertex/.style={rectangle,thick,draw, inner sep=1.3mm}, 
thickedge/.style={line width=1.5pt}] 

\fill[color=black!15] (0,0)--(-2.2,3)--(0, 2)--(0,0);

\node[vertex, thickedge, fill=black!50] (1) at (0,0) [label=below:$v_1$] {};
\node[vertex, thickedge, fill=white] (2) at (-2.2, 3) [label=above:$v_2$] {};
\node[vertex, thickedge, fill=white] (3) at (2.2, 3) [label=above:$v_3$] {};

\node[vertex, fill=white] (u) at (0, 2) [label=above:$u$] {};

\draw (1)--(u);

\draw[thickedge] (1)--(2)--(3)--(1);

\draw[dashed, color=black] (u)--(2);

\node at (0, -1) {(1)};

\begin{scope}[shift={(6,0)}]
\fill[color=black!15] (0,0)--(2.2,3)--(0, 2)--(0,0);

\node[vertex, thickedge, fill=black!50] (1) at (0,0) [label=below:$v_1$] {};
\node[vertex, thickedge, fill=white] (2) at (-2.2, 3) [label=above:$v_2$] {};
\node[vertex, thickedge, fill=white] (3) at (2.2, 3) [label=above:$v_3$] {};

\node[vertex, fill=white] (u) at (0, 2) [label=above:$u$] {};

\draw (1)--(u);

\draw[thickedge] (1)--(2)--(3)--(1);

\draw[dashed, color=black] (u)--(3);

\node at (0, -1) {(2)};
\end{scope}

\begin{scope}[shift={(12,0)}]
\fill[color=black!15] (0,0)--(-0.7, 2.3)--(0.7, 2.3)--(0,0);

\node[vertex, thickedge, fill=black!50] (1) at (0,0) [label=below:$v_1$] {};
\node[vertex, thickedge, fill=white] (2) at (-2.2, 3) [label=above:$v_2$] {};
\node[vertex, thickedge, fill=white] (3) at (2.2, 3) [label=above:$v_3$] {};

\node[vertex, fill=white] (u) at (-0.7, 2.3) [label=above:$u$] {};
\node[vertex, fill=white] (w) at (0.7, 2.3) [label=above:$w$] {};

\draw (1)--(u);
\draw (1)--(w);

\draw[thickedge] (1)--(2)--(3)--(1);

\draw[dashed, color=black] (u)--(w);

\node at (0, -1) {(3)};
\end{scope}

\end{tikzpicture}
\caption{Some steps in the proof of Lemma \ref{lem_5_3_1}.}
		\label{fig:5_proof1}
	\end{figure}
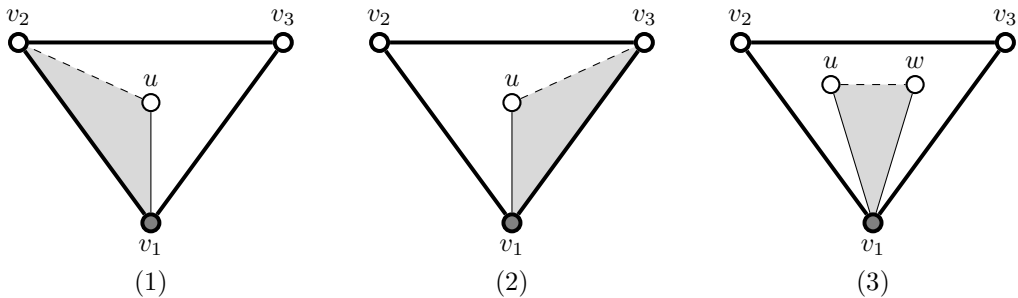
	
	Thus there is some other vertex $w$ in $\Int(C)$ that is adjacent to $u$.
	Since $v_1$ dominates $\Int(C)$, the vertices $v_1$, $u$ and $w$ form a 3-cycle, contradicting the minimality of $C$.
\end{proof}

\begin{lem}
	Let $G$ be a pentagulation, and let $C$ be a 3-cycle in $G$. 
	The interior of $C$ is not dominated by any two vertices of $C$.
	\label{lem_5_3_2}
\end{lem}

\begin{proof}
	Let $C= v_1, v_2, v_3$ be a 3-cycle. 
	Assume to the contrary and without loss of generality that every vertex in $\Int(C)$ is dominated by $\{v_1,v_2\}$. 
	We claim that no vertex in $\Int(C)$ is adjacent to $v_3$.
	Assume to the contrary there is a vertex $v$ adjacent to $v_3$.
	Without loss of generality, $v$ is adjacent to $v_1$ as well, since $\{v_1,v_2\}$ dominates $\Int(C)$.
	Thus the triangle $v_1, v, v_3$ is dominated by $v_1$, contradicting Lemma \ref{lem_5_3_1} and proving the claim.
	
	The edge $v_1v_2$ lies on the boundary of two faces, one of which is in the interior of $C$.
	Call this interior face $f$, and note that the boundary of $f$ is a 5-cycle.
	Per Lemma \ref{lem_5_3_1}, the interior of $C$ is not dominated by a single vertex, so both $v_1$ and $v_2$ have some neighbor in $\Int(C)$. 
	Thus the cycle bounding $f$ is of the form $u, v_1, v_2, w, x$, where $u$, $w$ and $x$ are vertices in the interior of $C$. 
	As $\{v_1, v_2\}$ dominates $\Int(C)$, the vertex $x$ is adjacent to either $v_1$ or $v_2$. 
	If $x$ is adjacent to $v_1$, then $u,x,v_1$ is a triangle whose interior is dominated by $v_1$, and similarly if $x$ is adjacent to $v_2$ then $w,x,v_2$ is a triangle whose interior is dominated by $v_2$.
	Both possibilities contradict Lemma \ref{lem_5_3_1}, completing the proof.
\end{proof}

\begin{lem}
	Let $G$ be a pentagulation and $C$ be a 4-cycle in $G$.
	Then no vertex of $C$ dominates $\Int(C)$.
	\label{lem_5_4_1}
\end{lem}

\begin{proof}
	Let $C = v_1, v_2, v_3, v_4$ be a 4-cycle. 
	Assume for the sake of contradiction that $v_1$ dominates $\Int(C)$, and choose $C$ to be minimal, i.e., there is no 4-cycle $C'$ dominated by $v_1$ such that $\Int(C') \subset \Int(C)$.
	Per Remark \ref{rem:cycle_sep}, $\Int(C) \neq \emptyset$.
	Let $u$ be the neighbor of $v_1$ in the interior of $C$ such that $uv_1$ and $v_1v_2$ both lie on the boundary of some common face. 
	Since $u$ is not an end-vertex, it is adjacent to some vertex $w$ in $\Int[C] - v_1$. 
	Up to symmetry, there are three possible choices for the vertex $w$.
	
	\textit{Case 1:} $w = v_2$ or $w = v_4$.\\
	If $w=v_2$, we obtain a 3-cycle $v_1, u, v_2$, the interior of which is dominated by $v_1$, contradicting Lemma \ref{lem_5_3_1}. 
	The situation is similar if $u$ is adjacent to $v_4$. 
	
	\textit{Case 2:} $w = v_3$. \\
	The interior of the 4-cycle $v_1, u, v_3, v_2$ is dominated by $v_1$, contradicting minimality of $C$.
	
	\textit{Case 3:} $w$ is a vertex in $\Int(C)$.\\
	By assumption, the vertex $v_1$ dominates $\Int(C)$, so $v_1$ and $w$ are adjacent. 
	Thus $v_1, u, w$ is a 3-cycle whose interior is dominated by $v_1$, contradicting Lemma \ref{lem_5_3_1}.
\end{proof}

\begin{lem}
	Let $C$ be a 4-cycle in a pentagulation. 
	No pair of vertices of $C$, that are adjacent in $C$, dominate $\Int(C)$.
	\label{lem_5_4_2consecutive}
\end{lem}

\begin{proof}
	Assume for the sake of contradiction that $C = v_1, v_2, v_3, v_4$ is a 4-cycle in a pentagulation whose interior is dominated by $\{v_1,v_2\}$. 
	By Lemma \ref{lem_5_4_1}, both $v_1$ and $v_2$ have at least one neighbor in $\Int(C)$. 
	Thus there is a face $f$ in the interior of $C$, bounded by a 5-cycle of the form $u, v_1, v_2, w, x$, where $u$ and $w$ are vertices in $\Int(C)$ and $x$ is a vertex in $\Int[C]$.
	If $x$ is either $v_3$ or $v_4$, then $\Int[C]$ contains a triangle whose interior is dominated by $v_1$ or $v_2$ respectively, contradicting Lemma \ref{lem_5_3_1}.
	If $x$ lies in $\Int(C)$, then it is adjacent to either $v_1$ or $v_2$.
	If $x$ is adjacent to $v_1$, then $v_1, u, x$ is a triangle whose interior is dominated by $v_1$, and if $x$ is adjacent to $v_2$, then the interior of the triangle $v_2, w, x$ is dominated by $v_2$.
	In any case, we obtain a triangle whose interior is dominated by a single vertex, contradicting Lemma \ref{lem_5_3_1}.
\end{proof}

\begin{lem}
	A 3-cycle in a pentagulation does not dominate its interior (or exterior).
	\label{lem_5_3_3}
\end{lem}

\begin{proof}
	Let $C:v_1, v_2, v_3$ be a 3-cycle in a pentagulation $G$. 
	Assume for the sake of contradiction that $C$ dominates its interior.
	By Lemmas \ref{lem_5_3_1} and \ref{lem_5_3_2}, no proper subset of $V(C)$ dominates $\Int(C)$, so every vertex of $C$ has at least one neighbor in $\Int(C)$. 
	Thus there exists a neighbor $u$ of $v_1$ in $\Int(C)$.
	Since $G$ is 2-connected, the vertex $u$ has some neighbor $w$ in $\Int[C]-v_1$.
	By Lemma \ref{lem_5_3_2}, the vertex $w$ is neither $v_2$ nor $v_3$, as this induces a 3-cycle whose interior is dominated by two vertices.
	Per Lemma \ref{lem_5_3_1}, $w$ is not adjacent to $v_1$, as this creates a 3-cycle whose interior is dominated by $v_1$.
	By Lemma \ref{lem_5_4_2consecutive}, neither $v_2$ nor $v_3$ is adjacent to $w$, since this induces a 4-cycle, the interior of which is dominated by two adjacent vertices. 
	Thus $u$ does not have a neighbor in $\Int[C] - v_1$, a contradiction.
\end{proof}

Lemma \ref{lem_5_3_3} and Remark \ref{rem:cycle_sep} easily yield the following corollary, which we make extensive use of.

\begin{cor}
	Pentagulations of diameter 3 contain no 3-cycles.
	\label{cor:5_3_notriangle}
\end{cor}

\section{The structure of separating cycles}
\label{sec:sep_cycles}

We have shown that diameter 3 pentagulations do not contain 3-cycles (and, hence, that any 4-cycle or 5-cycle in a such a pentagulation is chordless). 
In this section, we describe the structure of 4-cycles and separating 5-cycles in diameter 3 pentagulations.

\begin{lem}
	If a pentagulation contains a Jordan separating 5-cycle $C$, then the interior of $C$ is dominated by neither a single vertex of $C$, nor by an adjacent pair of vertices in $C$.
	\label{lem_5_5}
\end{lem}

\begin{proof}
	Let $C=v_1, v_2, v_3, v_4, v_5$ be a Jordan separating cycle of a pentagulation $G$.
	We first prove that $\Int(C)$ is not dominated by a single vertex of $C$. 
	Assume to the contrary that $v_1$ dominates $\Int(C)$, and let $u$ be a neighbor of $v_1$ in $\Int(C)$.
	Since $G$ is 2-connected, $u$ has some neighbor in $\Int[C]-v_1$.
	If $u$ is adjacent to any neighbor of $v_1$ (including $v_2$ and $v_5$), then $G$ contains a triangle, contradicting Corollary \ref{cor:5_3_notriangle}. 
	If $u$ is adjacent to $v_3$ or $v_4$, we obtain a 4-cycle whose interior is dominated by the single vertex $v_1$, contradicting Lemma \ref{lem_5_4_1}. 
	Thus $u$ has no neighbor in $\Int[C] - \{v_1\}$, a contradiction.
	
	Now assume to the contrary that $\{v_1, v_2\}$ dominates $\Int(C)$. 
	Let $u$ be a neighbor of $v_1$ in the interior of $C$, and note that $u$ has some neighbor in $\Int[C] - v_1$.
	As in the previous argument, $u$ is not adjacent to any neighbor of $v_1$.
	If $u$ is adjacent to either $v_3$ or $v_4$, then $G$ contains a 4-cycle whose interior is either dominated by the single vertex $v_1$, or by the adjacent pair $\{v_1, v_2\}$, contradicting Lemma \ref{lem_5_4_1} or Lemma \ref{lem_5_4_2consecutive}, respectively.
	If $u$ is adjacent to some neighbor of $v_2$, then $G$ contains a 4-cycle whose interior is dominated by the adjacent pair $\{v_1, v_2\}$, yielding a contradiction.
\end{proof}

\begin{lem}
	Let $C$ be a 4-cycle of a pentagulation.
	If $C$ dominates its interior, then no two vertices which are adjacent in $C$ both have neighbors in $\Int(C)$.
	\label{lem_5_4_final}
\end{lem}

\begin{proof}
	Let $C=v_1, v_2, v_3, v_4$ be a 4-cycle in a pentagulation, and suppose that $C$ dominates its interior.
	Assume to the contrary, and without loss of generality, that both $v_1$ and $v_2$ have some neighbor in $\Int(C)$.
	The edge $v_1v_2$ lies on some face in the interior of $C$.
	This face is bounded by a 5-cycle of the form $u, v_1, v_2, w, x$, where $u$ and $w$ are neighbors of $v_1$ and $v_2$ respectively, and $x \in \Int[C]$.
	Since $C$ dominates its interior, the vertex $x$ is either a vertex of $C$, or is adjacent to a vertex of $C$.
	If $x$ is a vertex of $C$, or if $x$ is adjacent to $v_1$ or $v_2$, then there is some 3-cycle in $\Int[C]$ that dominates its interior, contradicting Lemma \ref{lem_5_3_3}.
	If $x$ is adjacent to $v_3$ or $v_4$, then there is some 4-cycle in $\Int[C]$ whose interior is dominated by two adjacent vertices of the 4-cycle, contradicting Lemma \ref{lem_5_4_2consecutive}.
\end{proof}

\begin{lem}
	Let $C$ be a 6-cycle in a pentagulation.
	If the interior of $C$ is dominated by two vertices $u$ and $v$ of $C$ such that $d_C(u,v) = 3$, then no chord of $C$ lies in the interior of $C$.
	\label{lem_5_6_chordless}
\end{lem}

\begin{proof}
	Let $C = v_1, v_2, v_3, v_4, v_5, v_6$ be a 6-cycle in a pentagulation, the interior of which is dominated by $\{v_1, v_4\}$. 
	Assume to the contrary that $e=v_iv_j$, with $|j-i| > 1 \text{ (mod 6)}$, is a chord of $C$ contained in $\Int[C]$. 
	If $|j-i| = 2$, then the chord induces a 3-cycle in $C$ that dominates its interior, contradicting Lemma \ref{lem_5_3_3}.
	Thus $|j-i| = 3$. 
	If $e=v_1v_4$ then the chord induces a 4-cycle whose interior is dominated by two adjacent vertices, contradicting Lemma \ref{lem_5_4_2consecutive}.
	If $e \neq v_1v_4$, then the chord induces a 4-cycle dominated by only one vertex, contradicting Lemma \ref{lem_5_4_1}. 
\end{proof}

\begin{lem}
	Let $C$ be a 6-cycle in a pentagulation. 
	If $\Int(C)$ is dominated by two vertices $u$ and $v$ with $d_C(u,v) = 3$, then there exists some vertex in $\Int(C)$ that is adjacent to both $u$ and $v$.
	\label{lem_5_6_vertex}
\end{lem}

\begin{proof}
	Let $G$ be a pentagulation.
	Assume to the contrary that $C = v_1, v_2, v_3, v_4, v_5, v_6$ is a 6-cycle in $G$ whose interior is dominated by $\{v_1, v_4\}$, and that no vertex in $\Int(C)$ is adjacent to both $v_1$ and $v_4$.
	Choose $C$ to be a minimal counterexample.
	That is, there is no 6-cycle $C'$ that has its interior dominated by $\{v_1,v_4\}$, and that does not contain any neighbor of both $v_1$ and $v_4$ in $\Int(C')$, and that satisfies $\Int(C')\subset \Int(C)$. 
	The cycle $C$ is chordless by Lemma \ref{lem_5_6_chordless}, and is a Jordan separating cycle by Lemma \ref{lem_slightly_long_cycle}, so there exists some vertex $w$ in $\Int(C)$.
	Without loss of generality, the vertex $w$ is adjacent to $v_1$. 
	Since $G$ is 2-connected, there is some neighbor $x$ of $w$ in $\Int[C]-\{v_1,v_4\}$.
	The vertex $x$ is neither $v_2$ nor $v_6$, as this would create a triangle $v_1,w,x,v_1$ that dominates its interior, contradicting Lemma \ref{lem_5_3_3}. 
	Further, $x$ is neither $v_3$ nor $v_5$ as either case induces a 4-cycle whose interior is dominated by $v_1$, contradicting Lemma \ref{lem_5_4_1}. 
	So $x$ lies in $\Int(C)$, and is adjacent to either $v_1$ or $v_4$. 
	If $x$ is adjacent to $v_1$, then $v_1, x, w$ is a triangle, the interior of which is dominated by $v_1$, contradicting Lemma \ref{lem_5_3_1}. 
	Thus $x$ is adjacent to $v_4$, and the two internally disjoint paths $v_1, v_2, v_3, v_4$ and $v_1, w, x, v_4$, induce a 6-cycle in $\Int[C]$. 
	The interior of this 6-cycle is dominated by $\{v_1, v_4\}$, and by assumption there is not a common neighbor of both $v_1$ and $v_4$ in the interior of this cycle, contradicting the minimality of $C$. 
\end{proof}

\begin{cor}
	Let $C$ be a Jordan separating 5-cycle in a pentagulation. 
	If $\Int(C)$ is dominated by two non-adjacent vertices $u$ and $v$ of $C$, then there is some vertex in $\Int(C)$ that is adjacent to both $v$ and $u$.
	\label{cor:5_5_vertex}
\end{cor}

\begin{proof}
	Let $G$ be a pentagulation, and let $C = v_1, v_2, v_3, v_4, v_5$ be a Jordan separating 5-cycle in $G$ whose interior is dominated by $\{v_1, v_3\}$. 
	Since $C$ is Jordan separating, there exists a vertex $w$ in $\Int(C)$ that is, without loss of generality, adjacent to $v_1$.
	If $w$ is adjacent to $v_3$, we are done. 
	Suppose $w$ is not adjacent to $v_3$
	Since $G$ is 2-connected, $w$ has some neighbor $x$ in $\Int[C]-v_1$. 
	The vertex $x$ is not any neighbor of $v_1$, as then $v_1, w, x$ is a triangle that dominates its interior, contradicting Lemma \ref{lem_5_3_3}. 
	Note that $x \neq v_4$, as this would induce a 4-cycle dominated by $v_1$, contradicting Lemma \ref{lem_5_4_1}. 
	Thus $x$ is a vertex in $\Int(C)$ that is adjacent to $v_3$. 
	The internally disjoint paths $v_1, v_5, v_4, v_3$ and $v_1, w, x, v_3$ induce a 6-cycle whose interior is dominated by $\{v_1, v_3\}$. 
	By Lemma \ref{lem_5_6_vertex}, the interior of this 6-cycle contains some vertex adjacent to both $v_1$ and $v_3$, completing the proof.
\end{proof}

\begin{lem}
	Let $G$ be a pentagulation. 
	If $C$ is a 4-cycle that dominates its interior, then every vertex $u$ in $\Int(C)$ has degree 2.
	\label{lem_5_4_degree}
\end{lem}

\begin{proof}
	Let $G$ be a pentagulation, let $C = v_1, v_2, v_3, v_4$ be a 4-cycle in $G$ that dominates its interior, and let $w$ be a vertex in $\Int(C)$.
	Since $C$ dominates its interior, we assume without loss of generality that $w$ is adjacent to $v_1$.
	Because $G$ is 2-connected, $w$ has at least one neighbor in $\Int[C]-\{v_1\}$.
	Assume to the contrary that $w$ has at least two distinct such neighbors, call them $x_1$ and $x_2$.
	Neither $x_1$ nor $x_2$ is adjacent to $v_1$, as this would induce a triangle in $\Int[C]$ that dominates its interior, contrary to Lemma \ref{lem_5_3_3}.
	Further, neither $x_1$ nor $x_2$ are both adjacent to either of $v_2$ or $v_4$, and in $\Int(C)$, per Lemma \ref{lem_5_4_final}.
	Thus $x_1$ is either $v_3$, or is adjacent to $v_3$, and likewise for $x_2$.
	If $x_1 = v_3$, then $x_2 \neq v_3$, and hence $x_2$ is adjacent to $v_3$.
	However this induces a triangle $v_3, w, x_2$ that dominates its interior, contrary to Lemma \ref{lem_5_3_3}.
	We conclude that both $x_1$ and $x_2$ are neighbors of $v_3$ in $\Int(C)$. 
	But this induces a 4-cycle $x_1, w, x_2, v_3$ that is dominated by $v_3$, contradicting Lemma \ref{lem_5_4_1}.
\end{proof}

Per Remark \ref{rem:cycle_sep}, any 4-cycle in a pentagulation of diameter 3 dominates either its interior or exterior. 
The next theorem gives a complete description of the structure of this dominated region.
An example of such a region is given by Figure \ref{fig:5_4cycles}.

\begin{thm}
	Let $G$ be a pentagulation, and $C$ a 4-cycle in $G$.
	If $C$ dominates its interior, then there exist two non-adjacent vertices $u$ and $v$ of $C$, and a positive integer $k$ such that the induced subgraph $G[Int[C]]$ consists of exactly:\\
	(1) the cycle $C$,\\
	(2) $k$ $u-v$ paths of length 3, and\\
	(3) $k-1$ $u-v$ paths of length 2.\\
	All the paths in (2) and (3) are internally disjoint, do not contain any vertices of $C-\{u,v\}$, and the paths of length 2 and 3 alternate.
	\label{thm:4_cycle_description}
\end{thm}

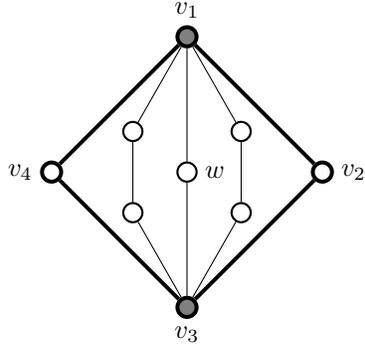
\begin{figure}[h]
\centering
\begin{tikzpicture}
[inner sep=0.9mm, scale=0.9, 
vertex/.style={circle,thick,draw},
dvertex/.style={rectangle,thick,draw, inner sep=1.3mm}, 
thickedge/.style={line width=1.5pt}] 

\node[vertex, thickedge, fill=black!50] (1) at (0,4) [label=above:$v_1$] {};
\node[vertex, thickedge] (2) at (2,2) [label=right:$v_2$] {};
\node[vertex, thickedge, fill=black!50] (3) at (0,0) [label=below:$v_3$] {};
\node[vertex, thickedge] (4) at (-2,2) [label=left:$v_4$] {};

\draw[thickedge] (1)--(2)--(3)--(4)--(1);

\node[vertex] (a) at (-0.8, 2.6) {};
\node[vertex] (b) at (-0.8, 1.4) {};
\node[vertex] (c) at (0, 2) [label=right:$w$] {};
\node[vertex] (d) at (0.8, 2.6) {};
\node[vertex] (e) at (0.8, 1.4) {};

\draw (1)--(a)--(b)--(3)--(c)--(1)--(d)--(e)--(3);

\end{tikzpicture}
\caption{A 4-cycle dominating its interior which has $k=2$ paths of length 3 and $k-1=1$ paths of length 2 between two non-cycle-adjacent vertices $v_1$ and $v_3$, illustrating Theorem \ref{thm:4_cycle_description}.}
\label{fig:5_4cycles}
\end{figure}

\begin{proof}
	Let $G$ be a pentagulation, and $C: v_1, v_2, v_3, v_4$ a 4-cycle in $G$ that dominates its interior. 
	By Lemmas \ref{lem_5_4_1} and \ref{lem_5_4_2consecutive}, exactly two non-adjacent vertices of $C$ have neighbors in $\Int(C)$.
	Suppose without loss of generality that these two vertices are $v_1$ and $v_3$.
	The interior of $C$ is chordless, as a chord would induce a 3-cycle that dominates its interior, contradicting Lemma \ref{lem_5_3_3}.
	Per Lemma \ref{lem_5_4_degree}, any vertex in $\Int(C)$ has degree 2. 
	Further, any vertex in $\Int(C)$ is adjacent to either $v_1$ or $v_3$, and there is no 3-cycle in the interior of $C$ by Lemma \ref{lem_5_3_3}.
	Thus every vertex in $\Int(C)$ lies on a $v_1-v_3$ path of length 2 or 3, and these paths are internally disjoint. 
	Since $G$ is a pentagulation and every face is bounded by a 5-cycle, the paths of length 2 and 3 must alternate.
\end{proof}

By Corollary \ref{cor:5_3_notriangle}, no diameter 3 pentagulation contains a triangle. 
Figure \ref{fig:5_HI} shows two diameter 3 pentagulations containing 4-cycles. 

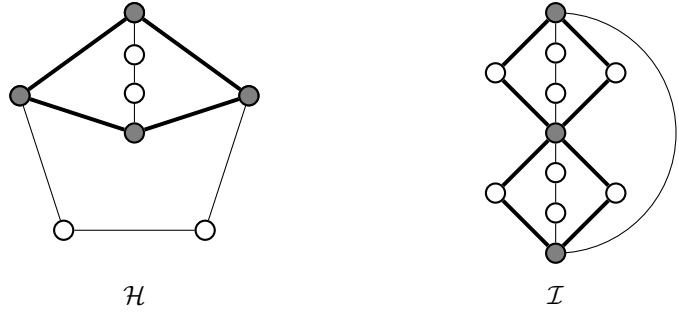
\begin{figure}[h]
\centering
\begin{tikzpicture}
[inner sep=0.9mm, scale=0.8,
vertex/.style={circle,thick,draw},
dvertex/.style={rectangle,thick,draw, inner sep=1.3mm},
thickedge/.style={line width=1.5pt}]

\foreach \x in {0,...,4} 
\node[vertex](\x) at (90+\x*72:2) {};

\node[vertex, fill=black!50] at (90+0*72:2) {};
\node[vertex, fill=black!50] at (90+1*72:2) {};
\node[vertex, fill=black!50] at (90+4*72:2) {};

\node[vertex, fill=black!50](c) at (0, 0) {};
\node[vertex](a) at (0,1.3) {};
\node[vertex](b) at (0,0.66) {};

\draw (0)--(a)--(b)--(c);
\draw (1)--(2)--(3)--(4);
\draw[thickedge] (4)--(0)--(1);
\draw[thickedge] (4)--(c)--(1);

\node at (0,-2.7) {$\mathcal{H}$};

\begin{scope}[shift={(7,0)}]

\newcommand\sca{1.35}

\draw (0,2) arc (90:-90:2);

\node[vertex, fill=black!50] (1) at (0,2) {};
\node[vertex=] (2) at (1,1) {};
\node[vertex, fill=black!50] (3) at (0,0) {};
\node[vertex] (4) at (-1,1) {};

\node[vertex] (6) at (1,-1) {};
\node[vertex, fill=black!50] (7) at (0,-2) {};
\node[vertex] (8) at (-1,-1) {};

\draw[thickedge] (1)--(2)--(3)--(4)--(1);
\draw[thickedge] (3)--(6)--(7)--(8)--(3);

\node[vertex] (a) at (0, 1.33) {};
\node[vertex] (b) at (0, 0.66) {};
\node[vertex] (c) at (0, -0.66) {};
\node[vertex] (d) at (0, -1.33) {};

\draw (1)--(a)--(b)--(3)--(c)--(d)--(7);

\node at (0,-2.7) {$\mathcal{I}$};

\end{scope}

\end{tikzpicture}
\caption{Two diameter three pentagulations that contain 4-cycles, $\mathcal{H}$ and $\mathcal{I}$.
Pairs of non-adjacent grey vertices dominate regions bounded by bold 4-cycles.}
\label{fig:5_HI}
\end{figure}

\section{Singling out a square with dislocated 4-cycles}
\label{sec:disloc_4_cycle}

In order to describe the structure of diameter 3 pentagulations, we need a new concept: \textit{dislocated} 4-cycles.
In Figure \ref{fig:5_4cycles}, consider the three 4-cycles $C_1: v_1, v_2, v_3, v_4$; $C_2: v_1, w, v_3, v_4$ and $C_3: v_1, w, v_3, v_2$.
Although these three cycles are distinct, both $C_2$ and $C_3$ are just `substructures' of $C_1$, formed by $C_1$ and the alternating paths in its interior (remember Theorem \ref{thm:4_cycle_description}).
Heuristically, two 4-cycles in a pentagulation are dislocated when they cannot be considered part of the same collection of alternating paths.

Consider two distinct 4-cycles, $C_1$ and $C_2$, in a pentagulation $G$.
We say that $C_1$ and $C_2$ are \textbf{dislocated} 4-cycles if there exist two regions $R_1\in \{\Int(C_1), \ext(C_1) \}$ and $R_2\in \{\Int(C_2), \ext(C_2) \}$, as well as two pairs of vertices $\{u_1, v_1\} \subset V(C_1)$ and $\{u_2, v_2\} \subset V(C_2)$, such that all three of the following conditions hold:
\begin{enumerate}[topsep=-\parskip]
	\item The regions $R_1$ and $R_2$ are dominated by $\{u_1, v_1\}$ and $\{u_2, v_2\}$, respectively,
	\item The sets $\{u_1, v_1\}$ and $\{u_2, v_2\}$ are not equal,
	\item The intersection $R_1\cap R_2$ is empty. 
\end{enumerate}
Note that per Lemma \ref{lem_5_4_2consecutive}, the edge $u_1v_1$ is not in $E(C_1)$, and $u_2v_2$ is not in $E(C_2)$.

\begin{figure}[h]
\centering
\begin{tikzpicture}
[scale=0.9, inner sep=0.8mm, scale=0.9, 
vertex/.style={circle,thick,draw},
dvertex/.style={rectangle,thick,draw, inner sep=1.3mm}, 
thickedge/.style={line width=1.5pt}] 

\node[vertex, thickedge, fill=black!50] (1) at (0,4) [label=above:$u_1$] {};
\node[vertex, thickedge] (2) at (2,2) [label=right:$u_2$] {};
\node[vertex, thickedge, fill=black!50] (3) at (0,0) [label=right:$u_3$] {};
\node[vertex, thickedge] (4) at (-2,2) [label=left:$u_4$] {};

\draw[thickedge] (1)--(2)--(3)--(4)--(1);

\node[vertex] (a) at (-0.8, 2.6) {};
\node[vertex] (b) at (-0.8, 1.4) {};
\node[vertex] (c) at (0, 2) [label=right:$u_5$] {};
\node[vertex] (d) at (0.8, 2.6) {};
\node[vertex] (e) at (0.8, 1.4) {};

\node[vertex] (s) at (-2,4.7) {};
\node[vertex] (t) at (2,4.7) {};
\node[vertex] (u) at (-3,2) {};
\node[vertex] (v) at (3,2) {};
\node[vertex] (w) at (0,-1) {};

\draw (1)--(a)--(b)--(3)--(c)--(1)--(d)--(e)--(3);
\draw[thickedge] (1)--(c)--(3);

\draw (4)--(s)--(t)--(2) (s)--(u)--(w);
\draw (t)--(v)--(w) (3)--(w);

\node at (0, -1.6) {$G$};

\begin{scope}[shift={(9,2)}]
\node[vertex] (v1) at (0,2) [label=above:$v_2$] {};
\node[vertex] (v3) at (0,-2) [label=below:$v_4$] {};
\node[vertex, thickedge, fill=black!50] (v2) at (-0.8,0) [label=right:$v_3$] {};
\node[vertex, thickedge, fill=black!50] (v4) at (-2.5,0) [label=left:$v_1$] {};

\node[vertex, thickedge, fill=black!50] (u2) at (2.5,0) [label=right:$v_6$] {};
\node[vertex, thickedge, fill=black!50] (u4) at (0.8,0) [label=left:$v_5$] {};

\draw[thickedge] (v3)--(v4)--(v1)--(v2)--(v3)--(u2)--(v1)--(u4)--(v3);

\node[vertex] (w1) at (-1.8,0) {};
\node[vertex] (w2) at (-1.3,0) {};
\node[vertex] (z1) at (1.8,0) {};
\node[vertex] (z2) at (1.3,0) {};

\node[vertex] (s) at (-2,2.7) {};
\node[vertex] (t) at (2,2.7) {};
\node[vertex] (u) at (0,0.7) {};
\node[vertex] (v) at (0,-0.7) {};

\draw (v4)--(w1)--(w2)--(v2); 
\draw (u4)--(z2)--(z1)--(u2);

\draw (v4)--(s)--(t)--(u2);
\draw (v1)--(u)--(v)--(v3);

\node at (0, -3.6) {$H$};

\end{scope}

\end{tikzpicture}
\caption{In $G$, there is no pair of dislocated 4-cycles.
In $H$, any pair of 4-cycles in which both cycles dominate their interior or exterior is dislocated.}
\label{fig:5_disloc_example}
\end{figure}
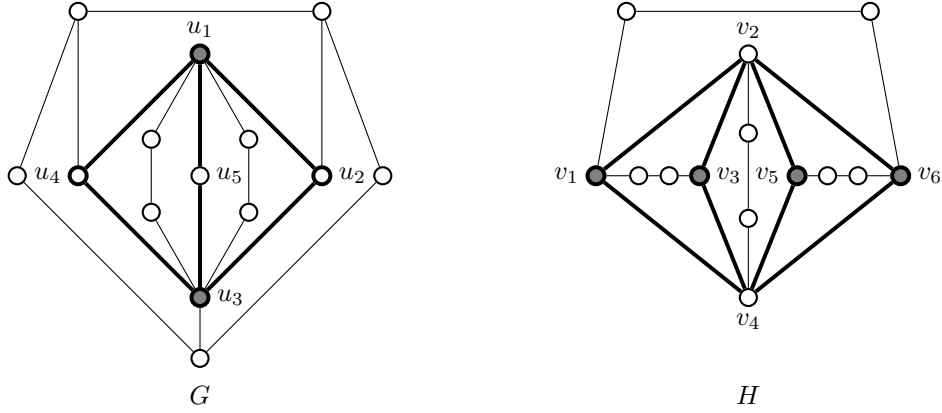

For an example, consider Figure \ref{fig:5_disloc_example}.
No two of these cycles in $G$ are dislocated, as they fail either condition (2) or (3) of the definition.
In $H$, any pair $C$ and $D$ of 4-cycles such that $C$ and $D$ both dominate one of their two regions is a dislocated pair.

\section{Bounding the order, part I: An abundance of 4-cycles}
\label{sec:bounding_1}

In this section, we consider pentagulations containing two or more dislocated 4-cycles.
But first, we handle a simple case, for which we recall the well-known theorem stating that if a graph of order $n$ and maximum degree $\Delta$ is dominated by $\gamma$ vertices, then $n\leq \gamma(\Delta + 1)$ (see, for example, \cite{gd_chartrand_96}).

\begin{lem}
	Let $G$ be a pentagulation of order $n$ and maximum degree $\Delta \geq 3$.
	If any 4-cycle of $G$ dominates $G$, then $n \leq 3\Delta -1$.
	\label{lem_4_cycle_dom_bound}
\end{lem}

\begin{proof}
	Let $G$ be a pentagulation of order $n$ and maximum degree $\Delta$ that is dominated by the 4-cycle $C: v_1, v_2, v_3, v_4$.
	Since $\Int(C)$ is dominated by $C$, we have without loss of generality, per Theorem \ref{thm:4_cycle_description}, that every vertex of $\Int(C)$ lies on a $v_1-v_3$ path of length 2 or 3.
	There are at most $\frac{\Delta - 1}{2}$ paths of length 3 in $\Int(C)$, and at most $\frac{\Delta-3}{2}$ paths of length 2 in $\Int(C)$.
	Because $\ext(C)$ is dominated by $C$, we have by Theorem \ref{thm:4_cycle_description} that every vertex of $\ext(C)$ lies on either a $v_1-v_3$ path, or a $v_2-v_4$ path, and any such path has length 2 or 3. 
	If the vertices of $\ext(C)$ lies on $v_1-v_3$ paths, then $\{v_1, v_3\}$ dominates $\ext(C)$, so $G$ is dominated by two vertices.
	Thus $n \leq 2\Delta + 2 \leq 3\Delta - 1$.
	
	Therefore the vertices of $\ext(C)$ lie on $v_2-v_4$ paths.
	As before, the number of paths of length 3 is bounded above by $\frac{\Delta-1}{2}$, and the number of paths of length 2 is at most $\frac{\Delta - 3}{2}$.
	Each path of length 3 in $\Int(C)$ ($\ext(C)$) contributes 2 to the number $|V(\Int(C))|$ ($|V(\ext(C))|$), and each path of length 2 contributes 1 to $|V(\Int(C))|$ ($|V(\Int(C))|$).
	Thus:
	\begin{align*}
	n = |V(C)| + |V(\Int(C))| + |V(\ext(C))|
	\leq 4 + 2\left[2\left(\frac{\Delta - 1}{2} \right) + 1\left(\frac{\Delta-3}{2} \right) \right] = 3\Delta - 1.
	\end{align*}
\end{proof}

In the proofs of Lemmas \ref{lem_H_as_subgraph} and \ref{lem_I_as_subgraph} to follow, we refer to specific vertices and faces of the graphs $\mathcal{H}$ and $\mathcal{I}$ by the labels given in Figure \ref{fig:5_HI_detail}.

\begin{figure}[h]
\centering
\begin{tikzpicture}
[scale = 0.9, inner sep=0.9mm, scale=0.9, 
vertex/.style={circle,thick,draw},
dvertex/.style={rectangle,thick,draw, inner sep=1.3mm}, 
thickedge/.style={line width=1.5pt}] 

\newcommand\sca{1.5}

\foreach \x in {0,...,4} 
\node[vertex](\x) at (90+\x*72:2*\sca) {};

\node[vertex, fill=black!50] at (90+0*72:2*\sca) [label=above:$v_1$]{};
\node[vertex, fill=black!50] at (90+1*72:2*\sca) [label=left:$v_2$]{};
\node[vertex, fill=black!50] at (90+4*72:2*\sca) [label=right:$v_4$]{};

\node[vertex] at (90+2*72:2*\sca) [label=below:$z_1$]{};
\node[vertex] at (90+3*72:2*\sca) [label=below:$z_2$]{};

\node[vertex, fill=black!50](c) at (0, 0) [label=below:$v_3$]{};
\node[vertex](a) at (0,1.3*\sca) [label=right:$w_1$]{};
\node[vertex](b) at (0,0.66*\sca) [label=right:$w_2$]{};

\draw (0)--(a)--(b)--(c);
\draw (1)--(2)--(3)--(4);
\draw[thickedge] (4)--(0)--(1);
\draw[thickedge] (4)--(c)--(1);

\node at (0,-3.8) {$\mathcal{H}$};

\node at (45:1.65) {\bm{$r_2$}};
\node at (135:1.65) {\bm{$r_1$}};
\node at (270:\sca) {\bm{$r_3$}};

\begin{scope}[shift={(7,0)}]

\draw (0,2*\sca) arc (90:-90:2*\sca);

\node[vertex, fill=black!50] (1) at (0,2*\sca) [label=left:$v_1$]{};
\node[vertex=] (2) at (\sca,\sca) [label=right:$v_2$]{};
\node[vertex, fill=black!50] (3) at (0,0) [label=left:$v_3$]{};
\node[vertex] (4) at (-\sca,\sca) [label=left:$v_7$]{};

\node[vertex] (6) at (\sca,-\sca) [label=right:$v_4$]{};
\node[vertex, fill=black!50] (7) at (0,-2*\sca) [label=left:$v_5$]{};
\node[vertex] (8) at (-\sca,-\sca) [label=left:$v_6$]{};

\draw[thickedge] (1)--(2)--(3)--(4)--(1);
\draw[thickedge] (3)--(6)--(7)--(8)--(3);

\node[vertex] (a) at (0, 1.33*\sca) [label=right:$w_1$]{};
\node[vertex] (b) at (0, 0.66*\sca) [label=right:$w_2$]{};
\node[vertex] (c) at (0, -0.66*\sca) [label=right:$z_1$]{};
\node[vertex] (d) at (0, -1.33*\sca) [label=right:$z_2$]{};

\draw (1)--(a)--(b)--(3)--(c)--(d)--(7);

\node at (0,-3.8) {$\mathcal{I}$};

\node at (-0.5*\sca, \sca) {\bm{$r_1$}};
\node at (0.5*\sca, \sca) {\bm{$r_2$}};
\node at (-0.5*\sca, -\sca) {\bm{$r_3$}};
\node at (0.5*\sca, -\sca) {\bm{$r_4$}};
\node at ( 1*\sca,0) {\bm{$r_5$}};
\node at ( -1*\sca,0) {\bm{$r_0$}};

\end{scope}

\end{tikzpicture}
\caption{The graphs $\mathcal{H}$ and $\mathcal{I}$, with the labels used in the proofs of Lemmas \ref{lem_H_as_subgraph} and \ref{lem_I_as_subgraph}.}
\label{fig:5_HI_detail}
\end{figure}
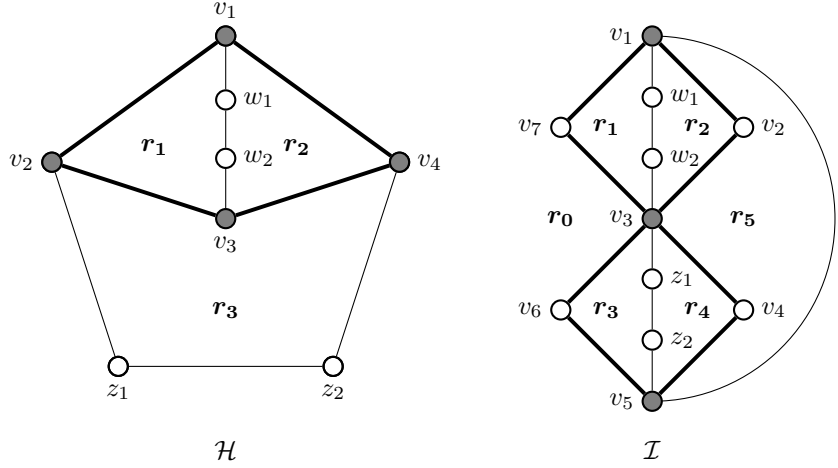

\begin{lem}
	Let $G$ be a pentagulation of diameter 3, order $n$ and maximum degree $\Delta$.
	If $G$ contains $\mathcal{H}$ as a subgraph, then $n \leq 3\Delta - 1$.
	\label{lem_H_as_subgraph}
\end{lem}

\begin{proof}
	Assume $G$ contains $\mathcal{H}$ (Figure \ref{fig:5_HI}) as a subgraph, and let $C:v_1, v_2, v_3, v_4$ be the 4-cycle of $H$.
	Label the remaining vertices of $\mathcal{H}$ so that $v_1, w_1, w_2, v_3$ and $v_2, z_1, z_2, v_4$ are paths of length 3 (see Figure \ref{fig:5_HI_detail}), with $w_1$ and $w_2$ lying in $\Int(C)$ and $z_1$ and $z_2$ lying in $\ext(C)$.
	Since $G$ has diameter 3, we know that, without loss of generality, the cycle $C$ dominates its interior by Remark \ref{rem:cycle_sep}. 
	Assume to the contrary that $C$ does not dominate its exterior. 
	Then there is a vertex $u\in \ext(C)$ such that $d(u,C) \geq 2$. 
	If $u$ lies in the outer face of $\mathcal{H}$, then $d(u,w_2) \geq 4$.
	If $u$ lies in $r_3$, then $d(u,w_1) \geq 4$.
	In either case, we obtain a contradiction, so $C$ dominates its exterior and is thus a dominating 4-cycle.
	That $n\leq 3\Delta - 1$ follows immediately from Lemma \ref{lem_4_cycle_dom_bound}.
\end{proof}

\begin{thm}
	Let $G$ be a pentagulation of diameter 3, order $n$, and maximum degree $\Delta \geq 3$. 
	If $G$ contains two dislocated 4-cycles, $C_1$ and $C_2$, then $G$ contains $\mathcal{I}$ as a subgraph (see Figure \ref{fig:5_HI}), or $n \leq 3\Delta - 1$.
	\label{thm:2_disloc_cycles_HI_subgraph}
\end{thm}

\begin{proof}
	Let $G$ be a pentagulation of diameter 3, order $n$ and maximum degree $\Delta \geq 3$.
	Suppose that $G$ contains two dislocated 4-cycles $C_1 : v_1, v_2, v_3, v_4$ and $C_2 : u_1, u_2, u_3, u_4$.
	We consider all the possible configurations for the two dislocated 4-cycles.
	Note that if any 4-cycle dominates $G$, or if $G$ contains an $\mathcal{H}$ subgraph, then $n \leq 3\Delta-1$ by Lemmas \ref{lem_4_cycle_dom_bound} and \ref{lem_H_as_subgraph}.
	Assume without loss of generality that $C_1$ dominates its interior.
	Per Theorem \ref{thm:4_cycle_description}, and without loss of generality, the region $\Int(C_1)$ is dominated by $\{v_1, v_3\}$, and there exist vertices $w_1$ and $w_2$ in $\Int(C_1)$ such that $P_1: v_1, w_1, w_2, v_3$ is a path in $G$.
	
	\textit{Case 1:} The cycles $C_1$ and $C_2$ have exactly two adjacent vertices in common.\\
	By symmetry, we may assume without loss of generality that $v_2 = u_1$ and $v_3 = u_4$ (See Figure \ref{fig:5_H1_Configs}, (1)).
	
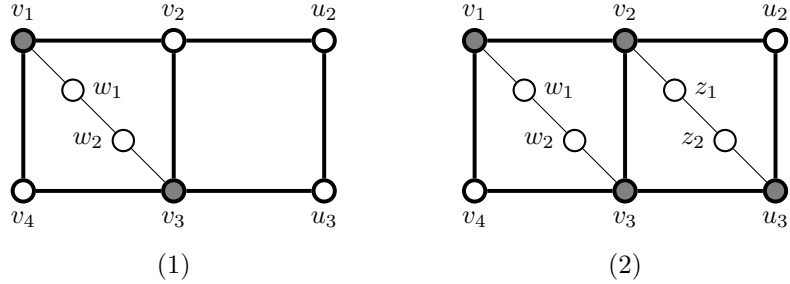
\begin{figure}[h]
\centering
\begin{tikzpicture}
[inner sep=1mm, 
vertex/.style={circle,thick,draw},
dvertex/.style={rectangle,thick,draw, inner sep=1.3mm}, 
thickedge/.style={line width=1.5pt}] 

\node[vertex, fill=black!50, thickedge](v1) at (0,0) [label=above:$v_1$]{};
\node[vertex, thickedge](v2) at (2,0) [label=above:$v_2$]{};
\node[vertex, fill=black!50, thickedge](v3) at (2,-2) [label=below:$v_3$]{};
\node[vertex, thickedge](v4) at (0,-2) [label=below:$v_4$]{};
\node[vertex, thickedge](u2) at (4,0) [label=above:$u_2$]{};
\node[vertex, thickedge](u3) at (4,-2) [label=below:$u_3$]{};

\node[vertex](w1) at (0.66,-0.66) [label=right:$w_1$]{};
\node[vertex](w2) at (1.33,-1.33) [label=left:$w_2$]{};

\draw[thickedge] (v1)--(v2)--(u2)--(u3)--(v3)--(v4)--(v1) (v2)--(v3);
\draw (v1)--(w1)--(w2)--(v3);

\node at (2, -3) {(1)};

\begin{scope}[shift={(6,0)}]

\node[vertex, fill=black!50, thickedge](v1) at (0,0) [label=above:$v_1$]{};
\node[vertex, thickedge, fill=black!50](v2) at (2,0) [label=above:$v_2$]{};
\node[vertex, fill=black!50, thickedge](v3) at (2,-2) [label=below:$v_3$]{};
\node[vertex, thickedge](v4) at (0,-2) [label=below:$v_4$]{};
\node[vertex, thickedge](u2) at (4,0) [label=above:$u_2$]{};
\node[vertex, thickedge, fill=black!50](u3) at (4,-2) [label=below:$u_3$]{};

\node[vertex](w1) at (0.66,-0.66) [label=right:$w_1$]{};
\node[vertex](w2) at (1.33,-1.33) [label=left:$w_2$]{};

\node[vertex](z1) at (2.66,-0.66) [label=right:$z_1$]{};
\node[vertex](z2) at (3.33,-1.33) [label=left:$z_2$]{};

\draw[thickedge] (v1)--(v2)--(u2)--(u3)--(v3)--(v4)--(v1) (v2)--(v3);
\draw (v1)--(w1)--(w2)--(v3);
\draw (v2)--(z1)--(z2)--(u3);

\node at (2, -3) {(2)};

\end{scope}

\end{tikzpicture}
\caption{Two dislocated 4-cycles, $C_1$ and $C_2$, that share an edge, as in Case 1 of the proof of Theorem \ref{thm:2_disloc_cycles_HI_subgraph}.}
\label{fig:5_H1_Configs}
\end{figure}
	
	Since $C_1$ and $C_2$ are dislocated, both $u_2$ and $u_3$ lie in $\ext(C_1)$.
	By Corollary \ref{cor:5_3_notriangle}, the pentagulation $G$ is triangle-free, so $d_G(w_1, C_2) =2$. 
	Since $C_2$ dominates either its interior or exterior, we have that $C_2$ dominates its interior. 
	By Theorem \ref{thm:4_cycle_description}, there exist vertices $z_1$ and $z_2$ in $\Int(C_2)$ such that either $P_2: v_2, z_1, z_2, u_3$ is a path in $G$, or $P_2': u_2, z_1, z_2, v_3$ is a path in $G$. 
	If $G$ contains the path $P_2'$, then there is a $z_1-w_1$ path $R$ of length at most 3 in $G$.
	Since $G$ is triangle free, the vertex $w_1$ is only adjacent to $v_1$ and $w_2$, and $z_1$ is only adjacent to $u_2$ and $z_2$.
	Thus, since $G$ is a plane graph and $d_G(w_1,z_1) \leq 3$, $v_1$ and $u_2$ are adjacent. 
	This induces a triangle, which is impossible.
	Therefore $G$ contains the path $P_2$, not the path $P_2'$ (see Figure \ref{fig:5_H1_Configs}, (2)).
	Since $G$ has diameter 3, there exists some $w_1-z_2$ path of length at most 3. 
	By the same argument as in the prior paragraph, we deduce that $v_1$ and $u_3$ are adjacent 
	But now we have induced $\mathcal{H}$ as a subgraph of $G$, with the 4-cycle of $\mathcal{H}$ corresponding to the 4-cycle of $G$ on $v_1, v_2=u_1, v_3=u_4, u_3$.
	By Lemma \ref{lem_H_as_subgraph}, we have $n \leq 3\Delta - 1$.
	
	\textit{Case 2:} The dislocated cycles $C_1$ and $C_2$ have exactly three vertices in common.\\
	Up to symmetry, there are two different ways that $C_1$ could share three vertices with $C_2$: the cycles may share both the dominating vertices $v_1$ and $v_3$, or only one of them.
	
	\textit{Case 2.1:} The vertices $v_1$ and $v_3$ are in both $C_1$ and $C_2$.\\
	Assume without loss of generality that $v_1 = u_1$, $v_2 = u_4$ and $v_3 = u_3$ 
	(see Figure \ref{fig:5_H2_Configs} (1)).
	
\begin{figure}[h]
\centering
\begin{tikzpicture}
[inner sep=0.9mm, scale=0.9, 
vertex/.style={circle,thick,draw},
dvertex/.style={rectangle,thick,draw, inner sep=1.3mm}, 
thickedge/.style={line width=1.5pt}] 

\node[vertex, thickedge, fill=black!50](v1) at (0,1.5) [label=above:$v_1$] {};
\node[vertex, thickedge, fill=black!50](v3) at (0,-1.5) [label=below:$v_3$] {};
\node[vertex, thickedge](v2) at (1.5,0) [label=right:$v_2$] {};
\node[vertex, thickedge](v4) at (-1.5,0) [label=left:$v_4$] {};
\node[vertex](w1) at (0, 0.4) [label=right:$w_1$] {};
\node[vertex](w2) at (0,-0.4) [label=right:$w_2$] {};

\draw[thickedge] (v1)--(v2)--(v3)--(v4)--(v1);
\draw (v1)--(w1)--(w2)--(v3);

\node[vertex, thickedge](u2) at (3,0) [label=right:$u_2$] {};

\draw[thickedge] (v1)--(u2)--(v3);

\node at (0.75, -3) {(1)};

\begin{scope}[shift={(7,1)}]

\node[vertex, thickedge, fill=black!50](v1) at (0,1.5) [label=above:$v_1$] {};
\node[vertex, thickedge, fill=black!50](v3) at (0,-1.5) [label=below:$v_3$] {};
\node[vertex, thickedge](v2) at (1.5,0) [label=right:$v_2$] {};
\node[vertex, thickedge](v4) at (-1.5,0) [label=left:$v_4$] {};
\node[vertex](w1) at (0, 0.4) [label=right:$w_1$] {};
\node[vertex](w2) at (0,-0.4) [label=right:$w_2$] {};

\draw[thickedge] (v1)--(v2)--(v3)--(v4)--(v1);
\draw (v1)--(w1)--(w2)--(v3);

\node[vertex, thickedge](u3) at (0,-3) [label=right:$u_3$] {};

\draw[thickedge] (v2)--(u3)--(v4);

\node at (0, -4) {(2)};

\end{scope}

\end{tikzpicture}
\caption{Case 2.1 in the proof of Theorem \ref{thm:2_disloc_cycles_HI_subgraph} has the two dislocated 4-cycles $C_1$ and $C_2$ sharing $v_1$, $v_2$ and $v_3$. 
	Case 2.2 has the cycles sharing $v_2$, $v_3$ and $v_4$.}
\label{fig:5_H2_Configs}
\end{figure}
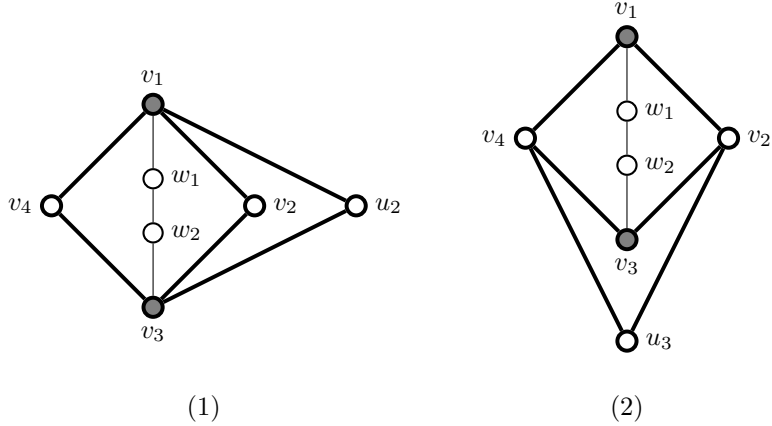
	
	Since $C_1$ and $C_2$ are dislocated, the set $\{u_2, v_2\}$ dominates either the interior or exterior of $C_2$. 
	We claim the set dominates the interior of $C_2$.
	By Lemma \ref{lem_5_4_final}, the vertex $v_2$ does not have any neighbor in $\Int(C_1)$, and thus has no neighbors in $\ext(C_2)$. 
	Per Lemma \ref{lem_5_4_1}, no single vertex of $C_2$ dominates the exterior of $C_2$, so the set $\{v_2, u_2\}$ does not dominate $\ext(C_2)$, proving the claim.
	
	Since $\{u_2, v_2\}$ dominates $\Int(C_2)$, there are two vertices $z_1$ and $z_2$ in $\Int(C_2)$ such that $P_2: v_2, z_1, z_2, u_2$ is a path in $G$.
	The vertices of $C_1\cup C_2 \cup P_1 \cup P_2$ induce an $\mathcal{H}$ subgraph in $G$.
	Thus $n\leq 3\Delta - 1$ per Lemma \ref{lem_H_as_subgraph}.
	
	\textit{Case 2.2:} Only one of $v_1$ and $v_3$ is common to both $C_1$ and $C_2$.\\
	Assume without loss of generality that $v_2=u_2$, $v_3=u_1$ and $v_4=u_4$ (see Figure \ref{fig:5_H2_Configs} (2)).
	Since $G$ is triangle-free, the distance $d_G(w_1, C_2) = 2$, so $C_2$ does not dominate its exterior and thus dominates its interior.
	By Theorem \ref{thm:4_cycle_description}, there are vertices $z_1$ and $z_2$ in $\Int(C_2)$ such that either $P_2: v_3, z_1, z_2, u_3$ is a path of $G$, or $P_2': v_2, z_1, z_2, v_4$ is a path of $G$.
	In the latter case, we obtain an $\mathcal{H}$ subgraph on $C_1\cup C_2\cup P_1\cup P_2'$.
	The former case, we have $d(w_1,z_2) > 3$.
	
	\textit{Case 3:} The cycles $C_1$ and $C_2$ have exactly one vertex in common.\\
	Since $C_1$ and $C_2$ only share one vertex, and $G$ is triangle-free, either $d(w_1, V(C_2)) \geq 2$ or $d(w_2, V(C_2)) \geq 2$. 
	As such, $C_2$ does not dominate its exterior, and thus dominates its interior.
	Up to symmetry, there are four possible cases.
	
	\textit{Case 3.1:} The dislocated cycles $C_1$ and $C_2$ share the vertex $v_2 = u_4$ and $\Int(C_2)$ is dominated by $\{u_1, u_3\}$ (see Figure \ref{fig:5_H3_Configs} (1)).\\
	
\begin{figure}[h]
\centering
\begin{tikzpicture}
[inner sep=0.9mm, scale=0.9, 
vertex/.style={circle,thick,draw},
dvertex/.style={rectangle,thick,draw, inner sep=1.3mm}, 
thickedge/.style={line width=1.5pt}] 

\node[vertex, thickedge, fill=black!50] (v1) at (0, 1.5) [label=above:$v_1$] {};
\node[vertex, thickedge, fill=black!50] (v3) at (0, -1.5) [label=below:$v_3$] {};
\node[vertex, thickedge] (v2) at (1.5, 0) [label=above:$v_2$] {};
\node[vertex, thickedge] (v4) at (-1.5, 0) [label=left:$v_4$] {};

\node[vertex, thickedge, fill=black!50] (u1) at (3, 1.5) [label=above:$u_1$] {};
\node[vertex, thickedge, fill=black!50] (u3) at (3, -1.5) [label=below:$u_3$] {};
\node[vertex, thickedge] (u2) at (4.5, 0) [label=right:$u_2$] {};

\draw[thickedge] (v1)--(v2)--(u1)--(u2)--(u3)--(v2)--(v3)--(v4)--(v1);

\node[vertex] (w1) at (0, 0.4) [label=right:$w_1$] {};
\node[vertex] (w2) at (0, -0.4) [label=right:$w_2$] {};
\node[vertex] (z1) at (3, 0.4) [label=right:$z_1$] {};
\node[vertex] (z2) at (3, -0.4) [label=right:$z_2$] {};

\draw (v1)--(w1)--(w2)--(v3) (u1)--(z1)--(z2)--(u3) ;

\node at (1.5, -2.5) {(1)};

\begin{scope}[shift={(8,0)}]

\node[vertex, thickedge, fill=black!50] (v1) at (0, 1.5) [label=above:$v_1$] {};
\node[vertex, thickedge, fill=black!50] (v3) at (0, -1.5) [label=below:$v_3$] {};
\node[vertex, thickedge, fill=black!50] (v2) at (1.5, 0) [label=above:$v_2$] {};
\node[vertex, thickedge] (v4) at (-1.5, 0) [label=left:$v_4$] {};

\node[vertex, thickedge] (u1) at (3, 1.5) [label=above:$u_1$] {};
\node[vertex, thickedge] (u3) at (3, -1.5) [label=below:$u_3$] {};
\node[vertex, thickedge, fill=black!50] (u2) at (4.5, 0) [label=right:$u_2$] {};

\draw[thickedge] (v1)--(v2)--(u1)--(u2)--(u3)--(v2)--(v3)--(v4)--(v1);

\node[vertex] (w1) at (0, 0.4) [label=right:$w_1$] {};
\node[vertex] (w2) at (0, -0.4) [label=right:$w_2$] {};
\node[vertex] (z1) at (2.6, 0) [label=above:$z_1$] {};
\node[vertex] (z2) at (3.4, 0) [label=above:$z_2$] {};

\draw (v1)--(w1)--(w2)--(v3) (v2)--(z1)--(z2)--(u2) ;

\node at (1.5, -2.5) {(2)};

\end{scope}

\end{tikzpicture}
\caption{In both figures, the dislocated 4-cycles $C_1$ and $C_2$ share the vertex $v_2 = u_4$. 
	We have (1) when the interior of $C_2$ is dominated by $\{u_1, u_3\}$, as in Case 3.1, and we have (2) when the interior of $C_2$ is dominated by $\{u_2, u_4\}$, as in Case 3.2 of the proof of Theorem \ref{thm:2_disloc_cycles_HI_subgraph}.}
\label{fig:5_H3_Configs}
\end{figure}
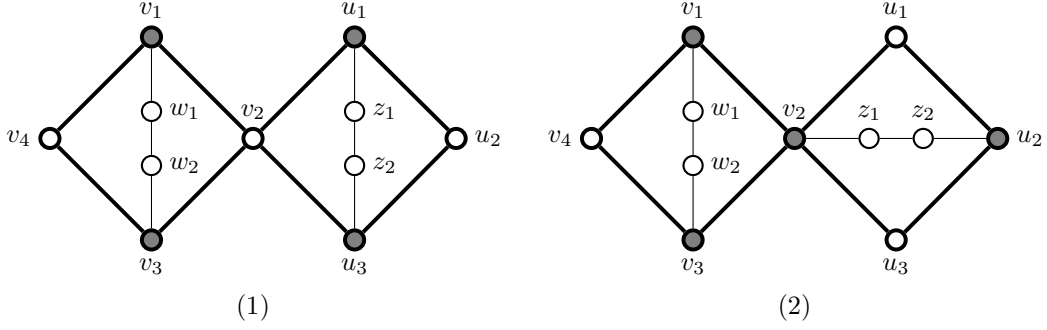
	
	By Theorem \ref{thm:4_cycle_description}, there is a vertex $z_1$ in $\Int(C_2)$ that is adjacent to $u_1$, but not to any other vertex of $C_2$. 
	But then $d_G(w_1, z_1) > 3$, a contradiction.
	
	\textit{Case 3.2:} The dislocated cycles $C_1$ and $C_2$ share the vertex $v_2 = u_4$ and $\Int(C_2)$ is dominated by $\{u_2, u_4\}$ (see Figure \ref{fig:5_H3_Configs} (2)).\\
	By Theorem \ref{thm:4_cycle_description}, there are two vertices $z_1$ and $z_2$ in the interior of $C_2$ such that $P_2: v_2, z_1, z_2, u_2$ is a path in $G$.
	Since $G$ is a triangle-free plane graph, and both $d_G(z_2, w_1) \leq 3$ and $d_G(z_2, w_w) \leq 3$, we have that $u_2$ is adjacent to both $v_1$ and $v_3$. 
	Thus $G$ contains $\mathcal{H}$ as a subgraph.
	
	\textit{Case 3.3:} The dislocated cycles $C_1$ and $C_2$ share the vertex $v_3 = u_1$ and $\Int(C_2)$ is dominated by $\{u_2, u_4\}$ (see Figure \ref{fig:5_H3_Configs2} (1)).\\
	
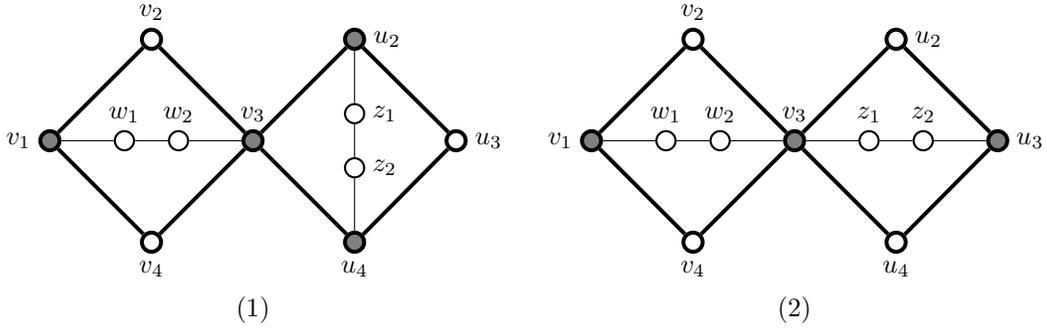
\begin{figure}[h]
\centering
\begin{tikzpicture}
[inner sep=0.9mm, scale=0.9, 
vertex/.style={circle,thick,draw},
dvertex/.style={rectangle,thick,draw, inner sep=1.3mm}, 
thickedge/.style={line width=1.5pt}] 

\node[vertex, thickedge, fill=black!50] (v1) at (-1.5, 0) [label=left:$v_1$] {};
\node[vertex, thickedge, fill=black!50] (v3) at (1.5, 0) [label=above:$v_3$] {};
\node[vertex, thickedge] (v2) at (0, 1.5) [label=above:$v_2$] {};
\node[vertex, thickedge] (v4) at (0,-1.5) [label=below:$v_4$] {};

\node[vertex, thickedge] (u3) at (4.5, 0) [label=right:$u_3$] {};
\node[vertex, thickedge, fill=black!50] (u4) at (3, -1.5) [label=below:$u_4$] {};
\node[vertex, thickedge, fill=black!50] (u2) at (3,1.5) [label=right:$u_2$] {};

\draw[thickedge] (v3)--(v4)--(v1)--(v2)--(v3)--(u2)--(u3)--(u4)--(v3);

\node[vertex] (w1) at (-0.4,0) [label=above:$w_1$] {};
\node[vertex] (w2) at (0.4,0) [label=above:$w_2$] {};
\node[vertex] (z1) at (3, 0.4) [label=right:$z_1$] {};
\node[vertex] (z2) at (3, -0.4) [label=right:$z_2$] {};

\draw (v1)--(w1)--(w2)--(v3) (u2)--(z1)--(z2)--(u4) ;

\node at (1.5, -2.5) {(1)};

\begin{scope}[shift={(8,0)}]

\node[vertex, thickedge, fill=black!50] (v1) at (-1.5, 0) [label=left:$v_1$] {};
\node[vertex, thickedge, fill=black!50] (v3) at (1.5, 0) [label=above:$v_3$] {};
\node[vertex, thickedge] (v2) at (0, 1.5) [label=above:$v_2$] {};
\node[vertex, thickedge] (v4) at (0,-1.5) [label=below:$v_4$] {};

\node[vertex, thickedge, fill=black!50] (u3) at (4.5, 0) [label=right:$u_3$] {};
\node[vertex, thickedge] (u4) at (3, -1.5) [label=below:$u_4$] {};
\node[vertex, thickedge] (u2) at (3,1.5) [label=right:$u_2$] {};

\draw[thickedge] (v3)--(v4)--(v1)--(v2)--(v3)--(u2)--(u3)--(u4)--(v3);

\node[vertex] (w1) at (-0.4,0) [label=above:$w_1$] {};
\node[vertex] (w2) at (0.4,0) [label=above:$w_2$] {};
\node[vertex] (z1) at (2.6, 0) [label=above:$z_1$] {};
\node[vertex] (z2) at (3.4, 0) [label=above:$z_2$] {};

\draw (v1)--(w1)--(w2)--(v3) (v3)--(z1)--(z2)--(u3) ;

\node at (1.5, -2.5) {(2)};

\end{scope}

\end{tikzpicture}
\caption{In both figures, the dislocated 4-cycles $C_1$ and $C_2$ share the vertex $v_3 = u_1$. 
	When the interior of $C_2$ is dominated by $u_2$ and $u_4$, as in Case 3.3 of the proof of Theorem \ref{thm:2_disloc_cycles_HI_subgraph}, (1) occurs.
	When the interior of $C_2$ is dominated by $u_1$ and $u_3$, as in Case 3.4, (2) occurs.}
\label{fig:5_H3_Configs2}
\end{figure}
	
	Reversing the roles of the cycles $C_1$ and $C_2$, we observe that this case is identical to Case 3.2, hence $G$ contains $\mathcal{H}$ as a subgraph, so $n \leq 3\Delta - 1$.
	
	\textit{Case 3.4:} The dislocated cycles $C_1$ and $C_2$ share the vertex $v_3 = u_1$ and $\Int(C_2)$ is dominated by $\{u_1, u_3\}$ (see Figure \ref{fig:5_H3_Configs2} (2)).\\
	By Theorem \ref{thm:4_cycle_description}, there are vertices $z_1$ and $z_2$ in $\Int(C_2)$ such that $P_2: v_3, z_1, z_2, u_3$ is a path in $G$.
	Since $d(w_1, z_2) \leq 3$ $w_1-z_2$, we have that $v_1$ and $u_3$ are adjacent. 
	Thus $\mathcal{I}$ is a subgraph of $G$.
	
	\textit{Case 4:} The dislocated cycles $C_1$ and $C_2$ are disjoint.\\
	In this case, no vertex of $C_2$ is adjacent to $w_1$, so $C_2$ dominates its interior.
	Per Theorem \ref{thm:4_cycle_description}, and without loss of generality, there are vertices $z_1$ and $z_2$ in the interior of $C_2$ and edges $u_1z_1$, $z_1z_2$ and $z_2u_3$.
	Since $G$ has diameter 3, we have that $d_G(w_i,z_j) \leq 3$ for any indices $i$ and $j$ in $\{1,2\}$.
	Since $G$ is triangle-free, it contains all four edges of the form $u_iw_k$, where $i$ and $k$ are in $\{1,3\}$.
	However, noting the 4-cycle on $v_1, u_1, v_3, u_3$, we see that $G$ contains $\mathcal{H}$ as a subgraph.
	
	\textit{Case 5:} The dislocated cycles $C_1$ and $C_2$ share exactly two non-adjacent vertices.\\
	
\begin{figure}[h]
\centering
\begin{tikzpicture}
[inner sep=0.9mm, scale=1.0, 
vertex/.style={circle,thick,draw},
dvertex/.style={rectangle,thick,draw, inner sep=1.3mm}, 
thickedge/.style={line width=1.5pt}] 

\node[vertex, thickedge, fill=black!50] (v1) at (0,2) [label=above:$v_1$] {};
\node[vertex, thickedge, fill=black!50] (v3) at (0,-2) [label=below:$v_3$] {};
\node[vertex, thickedge] (v2) at (-0.8,0) [label=right:$v_2$] {};
\node[vertex, thickedge] (v4) at (-2.5,0) [label=left:$v_4$] {};

\node[vertex, thickedge, fill=black!50] (u2) at (2.5,0) [label=right:$u_2$] {};
\node[vertex, thickedge, fill=black!50] (u4) at (0.8,0) [label=left:$u_4$] {};

\draw[thickedge] (v3)--(v4)--(v1)--(v2)--(v3)--(u2)--(v1)--(u4)--(v3);

\node[vertex] (w1) at (-1.2,0.5) [label=210:$w_1$] {};
\node[vertex] (w2) at (-1.2,-0.5) [label=150:$w_2$] {};

\draw (v1)--(w1)--(w2)--(v3);

\node at (0, -3.2) {(1)};

\begin{scope}[shift={(7,0)}]
\node[vertex] (v1) at (0,2) [label=above:$v_2$] {};
\node[vertex] (v3) at (0,-2) [label=below:$v_4$] {};
\node[vertex, thickedge, fill=black!50] (v2) at (-0.8,0) [label=right:$v_3$] {};
\node[vertex, thickedge, fill=black!50] (v4) at (-2.5,0) [label=left:$v_1$] {};

\node[vertex, thickedge, fill=black!50] (u2) at (2.5,0) [label=right:$u_2$] {};
\node[vertex, thickedge, fill=black!50] (u4) at (0.8,0) [label=left:$u_4$] {};

\draw[thickedge] (v3)--(v4)--(v1)--(v2)--(v3)--(u2)--(v1)--(u4)--(v3);

\node[vertex] (w1) at (-1.8,0) [label=85:$w_1$] {};
\node[vertex] (w2) at (-1.3,0) [label=280:$w_2$] {};

\draw (v4)--(w1)--(w2)--(v2); 

\node at (0, -3.2) {(2)};

\end{scope}

\end{tikzpicture}
\caption{In (1), the dislocated 4-cycles $C_1$ and $C_2$ share vertices $v_1 = u_1$ and $v_3 = u_3$, as in Case 5.1 of Theorem \ref{thm:2_disloc_cycles_HI_subgraph}.
In Figure (2), the cycles share vertices $v_2 = u_1$ and $v_4 = u_3$, as in Case 5.2.}
\label{fig_5_H3_Configs3}
\end{figure}
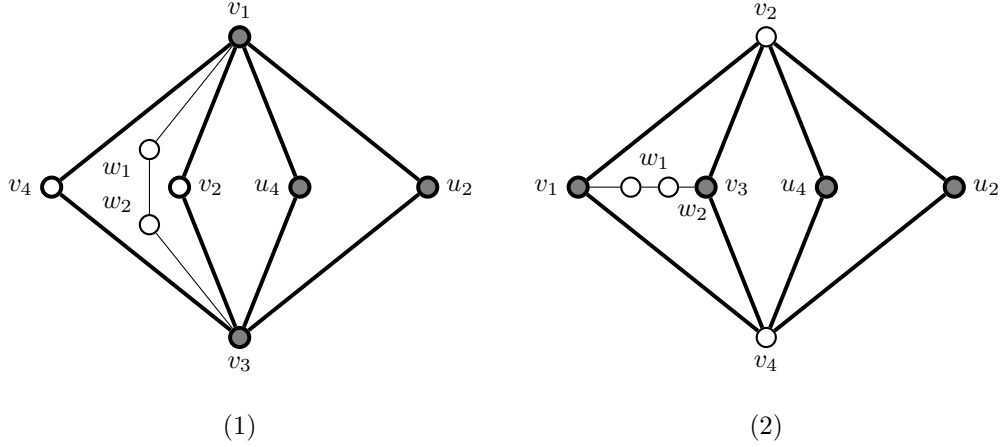
	
	Up to symmetry, there are two subcases to consider. 
	Either $v_1 = u_1$ and $v_3 = u_3$ are common to both $C_1$ and $C_2$, or the vertices $v_2 = u_2$ and $v_4 = u_4$ are.
	In both cases, since $C_1$ and $C_2$ are dislocated, the set $\{u_2, u_4\}$ of vertices dominates the interior of $C_2$ (it does not dominate the exterior, as neither is adjacent to $w_1$).
	Thus, in both cases, per Theorem \ref{thm:4_cycle_description}, there are vertices $z_1$ and $z_2$ in $\Int(C_2)$ such that $P_2: u_2, z_1, z_2, u_4$ is a path in $G$.
	
	\textit{Case 5.1:} The vertices $v_1$ and $v_3$ are common to $C_1$ and $C_2$ (See Figure \ref{fig_5_H3_Configs3} (1)).\\
	Consider the cycle $C: v_1, v_2, v_3, u_4$.
	Since $z_1$ is not adjacent to a vertex of $C$, the cycle $C$ dominates its interior.
	If $\{v_1, v_3\}$ dominates $\Int(C)$, then $C$ and $C_2$ are dislocated 4-cycles sharing three vertices, and by Case 2 we have that $n \leq 3\Delta - 1$.
	Similarly, if $\{v_2, u_4\}$ dominates $\Int(C)$, then $C$ and $C_1$ are dislocated.
	
	\textit{Case 5.2:} The vertices $v_2$ and $v_4$ are common to $C_1$ and $C_2$ (See Figure \ref{fig_5_H3_Configs3} (2)).\\
	Denote by $C'$ the cycle on $v_2, v_3, v_4, u_4$. 
	By the argument of the preceding paragraph, $C'$ and $C_1$ are dislocated 4-cycles.
	Thus, by Case 2, $n\leq 3\Delta - 1$.
\end{proof}

\begin{lem}
	Let $G$ be a pentagulation of diameter 3, order $n$ and maximum degree $\Delta$.
	If $G$ contains $\mathcal{I}$ as a subgraph, then $n \leq 3\Delta - 1$.
	\label{lem_I_as_subgraph}
\end{lem}

\begin{proof}
	Let $G$ be a pentagulation of diameter 3, order $n$ and maximum degree $\Delta$ that contains $\mathcal{I}$ as a subgraph. 
	Let the vertices of $\mathcal{I}$ be labeled as they are in Figure \ref{fig:5_HI_detail}, such that the vertices $w_1$ and $z_1$ lie in the interiors of the 4-cycles $C_1 : v_1, v_2, v_3, v_7$ and $C_2 : v_3, v_4, v_5, v_6$, respectively.
	Since $G$ is triangle-free (Corollary \ref{cor:5_3_notriangle}), the subgraph $\mathcal{I}$ is an induced subgraph of $G$. 
	Therefore, $d_G(z_2, C_1) = 2$, and by a similar argument, $d_G(w_1, C_2) = 2$.
	Hence the cycles $C_1$ and $C_2$ dominate their interiors per Remark \ref{rem:cycle_sep}. 
	In particular, the set $\{v_1, v_3\}$ dominates $\Int(C_1)$, and $\{v_3, v_5\}$ dominates $\Int(C_2)$.
	We refine our choice of embedding of $G$ (or equivalently, our choice of subgraph isomorphic to $\mathcal{I}$), so that the interiors of the cycles $C_1$ and $C_2$ are maximal.
	In other words, there does not exist a 4-cycle $C_1'$ such that $\Int(C_1) \subset \Int(C_1')$ and $\Int(C_1')$ is dominated by $\{v_1, v_3\}$, and likewise for $C_2$.
	Assume for the sake of contradiction that $n> 3\Delta - 1$.
	Suppose that every vertex of $V(G)-V(\mathcal{I})$ is adjacent to at least one of $v_1$, $v_3$ or $v_5$. Then:
	\begin{align*}
	n &= |V(\mathcal{I})| + |V(G)-V(\mathcal{I})|\\
	&\leq 11 + (d(v_1) -4) + (d(v_3) - 6) + (d(v_5) - 4)\\
	&\leq 11 + 3\Delta - 14 = 3\Delta - 3 < 3\Delta - 1
	\end{align*}
	
	Thus assume that $G$ contains vertices in $V(G)-V(\mathcal{I})$ that are not adjacent to any of $v_1$, $v_3$ or $v_5$. 
	Let $x$ be such a vertex, and label the faces $r_0, r_1, \dots, r_5$ of $\mathcal{I}$ as they are labeled in Figure \ref{fig:5_HI_detail}. 
	The regions $r_1\cup r_2$, and $r_3\cup r_4$ are dominated by the 4-cycles $C_1$ and $C_2$, respectively, and as such any vertex added to these regions is adjacent to a vertex in the set $\{v_1, v_3, v_5\}$. 
	Thus we assume that $x$ is not in any of the regions $r_1$, $r_2$, $r_3$ or $r_4$.
	By the symmetry of $r_0$ and $r_5$, we assume without loss of generality that $x$ is in $r_5$. 
	If $x$ is adjacent to $v_2$ and $v_4$, then we induce a 4-cycle $C: v_2, x, v_4, v_3$ which shares an edge with the cycle $C_1$.
	Since $d(w_1, C) = 2$, $C$ dominates its interior.
	Thus $C$ and $C_1$ are dislocated 4-cycles that share an edge, so $n \leq 3\Delta - 1$ by Theorem \ref{thm:2_disloc_cycles_HI_subgraph}, a contradiction. 
	Hence we assume that $x$ is not adjacent to both $v_2$ and $v_4$. 
	There are two cases to consider.
	
	\textit{Case 1:} The vertex $x$ is not adjacent to either $v_2$ or $v_4$.\\
	Since the diameter of $G$ is 3, $x$ is within distance 3 of each of $w_1, w_2, z_1, z_2$.
	Thus $x$ has neighbors $y_1$, $y_2$ and $y_3$ in $r_5$ such that $y_1v_1$, $y_2v_3$ and $y_3v_5$ are all edges in $G$. 
	Note that $y_1\neq y_3$ as this induces a triangle with vertex set $\{v_1, y_1, v_5\}$. 
	We claim that $y_1\neq y_2$.
	Assume to the contrary that $y_1 = y_2$, and let $C$ be the 4-cycle on $v_1, v_2, v_3, y_1, v_1$. 
	Note that that $d_G(z_2, C) = 2$, so $C$ dominates its interior. 
	By the maximality of $C_1$, we deduce that $C$ and $C_1$ are dislocated 4-cycles that share more than one vertex.
	Thus $n \leq 3\Delta - 1$ by Theorem \ref{thm:2_disloc_cycles_HI_subgraph}, proving the claim.
	Similarly $y_2 \neq y_3$, so the three vertices $y_1$, $y_2$ and $y_3$ are distinct.
	The paths $Q_1: v_1, y_1, x$, $Q_2: v_3, y_2, x$ and $Q_3: v_5, y_3, x$ divide $r_5$ up into three sub-regions.
	Let $r_6$ denote the region with vertices $v_1, v_2, v_3, y_2, x, y_1$ on its boundary, let $r_7$ be bounded by $v_3, y_2, x, y_3, v_5, v_4$, and let $r_8$ be bounded by $v_1, y_1, x, y_3, v_5$. 
	
	We claim that the subgraph $\mathcal{I}' = \mathcal{I}\cup Q_1 \cup Q_2 \cup Q_3$ of $G$ is an induced subgraph. 
	Any edge between two vertices on the boundary of any region $r_0, \dots, r_4$ induces a triangle, which is not possible since $G$ is triangle-free. 
	Similarly, no edge crosses $r_8$.
	Any edge crossing $r_6$ either creates a triangle, which is not possible, or a 4-cycle $C$ such that $C_1$ and $C$ are two dislocated 4-cycles which share at least two vertices.
	By Theorem \ref{thm:2_disloc_cycles_HI_subgraph}, we have $n \leq 3\Delta - 1$, contrary to assumption. 
	The argument that no edges cross the region $r_7$ is similar to the argument for $r_6$, just replace the role of $C_1$ with $C_2$.
	This proves the claim.
	
	If there exists a vertex in $r_6$, it is adjacent to $v_1$ or $v_3$ since it is within distance 3 of $z_2$.
	Similarly, any vertex in $r_7$ is adjacent to $v_3$ or $v_5$ as it is within distance 3 of $w_1$.
	No vertex lies in $r_8$, as it would be adjacent to both $v_1$ and $v_5$ to be within distance 3 of $w_2$ and $z_1$ respectively, inducing a triangle on $y_4, v_1, v_5$. 
	Any vertex of $r_0$ is distance at most 3 from $x$, and thus adjacent to one of $v_1, v_3$ or $v_5$. 
	The subgraph $\mathcal{I}'$ has 15 vertices, and every vertex of $G -\mathcal{I}'$ is adjacent to one of $v_1$, $v_3$ or $v_5$. 
	Noting that $d_{\mathcal{I}'}(v_1) = 5$, $d_{\mathcal{I}'}(v_3) = 7$ and $d_{\mathcal{I}'}(v_5) = 5$, we can bound the order of $G$:
	\begin{align*}
	n	&\leq 15 + (d(v_1) - 5) + (d(v_3) - 7) + (d(v_5) - 5)\\
		&\leq 3\Delta -2 < 3\Delta - 1.
	\end{align*}
	
	\textit{Case 2:} The vertex $x$ is adjacent to $v_2$.\\
	By assumption, $x$ is not adjacent to any of $v_1$, $v_3$, $v_4$ or $v_5$, and $d(x, z_2) \leq 3$.
	As no two vertices on the boundary of $r_5$ are adjacent, there exists some vertex $y_1$ in $r_5$ such that there is a path $S_1: v_2, x, y_1, v_5$ in $G$.
	We claim that $\mathcal{I}\cup S_1$ is an induced subgraph of $G$.
	Since $G$ is triangle-free, no edges crosses a region bounded by a 5-cycle. 
	Thus the only possible region of $\mathcal{I}\cup S_1$ with a chord is the region bounded by the two paths $S_1$ and $v_2, v_3, v_4, v_5$. 
	However, any edge between the vertices bounding this region creates either a triangle, which is impossible, or two 4-cycles $A_1$ and $A_2$.
	In all cases, every vertex of $A_1$ and $A_2$ is distance at least 2 from $w_1$, so $A_1$ and $A_2$ dominate their interiors.
	Thus, for some $i$ and $j$ in $\{1,2\}$, the cycles $C_i$ and $A_j$ are a pair of dislocated 4-cycles that share at least two vertices.
	By Theorem \ref{thm:2_disloc_cycles_HI_subgraph}, we have $n \leq 3\Delta - 1$, proving the claim.
	
	Because $d_G(y_1, w_2) \leq 3$, and since $\mathcal{I}\cup S_1$ is an induced subgraph of $G$, there exists some vertex $y_2$ in $r_5 - \{x, y_1\}$ such that $G$ contains the path $S_2: y_1, y_2, v_3$.
	Let $\mathcal{I}'' = \mathcal{I}\cup S_1 \cup S_2$, and note that the paths $S_1$ and $S_2$ divide $r_5$ into three sub-regions: $r_6= \Int(v_1, v_2, x, y_1, v_5)$, $r_7= \Int(v_2, v_3, y_2, y_1, x)$ and $r_8 = \Int(v_3, y_2, y_1, v_5, v_4)$. 
	We show that any vertex in $G - \mathcal{I}''$ is adjacent to one of $v_1$, $v_3$ or $v_5$. 
	Since $G$ is triangle-free, and every face of $\mathcal{I}''$ is bounded by a 5-cycle, $\mathcal{I}''$ is an induced subgraph of $G$.
	As such, the only vertices on the boundary of $r_6$ within distance 2 of $w_2$ are $v_1$ and $v_2$.
	The region $r_6$ is empty per Lemma \ref{lem_5_5}, as it is dominated by two adjacent vertices.
	Similarly $r_7$ is empty, as the only vertices on the boundary of $r_7$ within distance 2 of $w_1$ are the adjacent pair $v_2$ and $v_3$.
	Any vertex in $r_8$ is adjacent to either $v_3$ or $v_5$, as it is distance at most 3 from $w_1$.
	Any vertex in $r_0$ is adjacent to one of $v_1$, $v_3$ or $v_5$ as it is distance at most 3 from $x$.
	Note that $\mathcal{I}''$ has 14 vertices, and that $d_{\mathcal{I}''}(v_1) = 4$, $d_{\mathcal{I}''}(v_3) = 7$ and $d_{\mathcal{I}''}(v_5) = 5$.
	Any vertex of $G-\mathcal{I}''$ is adjacent to one of $v_1$, $v_2$ or $v_3$, so we can bound the order of $G$: 
	\begin{align*}
		n \leq 14 + (d(v_1) - 4) + (d(v_3) - 7) + (d(v_5) - 5) \leq 3\Delta - 2.
	\end{align*}
	In every case, we have derived a contradiction, completing the proof.
\end{proof}

Theorem \ref{thm:5_two4cycles} follows immediately from Lemma \ref{lem_4_cycle_dom_bound}, Theorem \ref{thm:2_disloc_cycles_HI_subgraph} and Lemma \ref{lem_I_as_subgraph}

\begin{thm}
	Let $G$ be a pentagulation of diameter 3, order $n$ and maximum degree $\Delta \geq 8$. 
	If $G$ contains either a dominating 4-cycle, or two dislocated 4-cycles, then $n \leq 3\Delta - 1$.
	\label{thm:5_two4cycles}
\end{thm}

\section{Bounding the order, part II: The lonely 4-cycle}
\label{sec:bounding_2}

We show that if a pentagulation contains some 4-cycle, but no dislocated pair of them, then it satisfies $n \leq 3\Delta - 1$.
Throughout this section, we work with pentagulations of diameter 3 that contain some 4-cycle $C$. 
Assume without loss of generality that $C$ dominates its interior. 
This motivates the following terminology.
The 4-cycle $C$ of a plane graph is \textbf{interior maximal} if it dominates its interior, and there does not exist any other 4-cycle $C'$ such that $C'$ dominates its interior, and $\Int(C) \subset \Int(C')$.

\begin{lem}
	Let $G$ be a pentagulation of diameter 3 that does not contain two dislocated 4-cycles, and let $C$ be an interior maximal 4-cycle of $G$.
	If $D$ is any cycle in $\ext[C]$ of length at most 7, then $D$ is chordless.
	\label{lem_chordless_part_2}
\end{lem} 

\begin{proof}
	Assume to the contrary $D$ has some chord $e$.
	By Corollary \ref{cor:5_3_notriangle}, $D\cup \{e\}$ has no 3-cycle, so $D\cup \{e\}$ induces a 4-cycle.
	Either this 4-cycle contradicts the maximality of $C$, or is dislocated from $C$, and both cases yield a contradiction.
\end{proof}

\begin{lem}
	Let $G$ be a pentagulation of diameter 3 that does not contain two dislocated 4-cycles, and let $C$ be an interior maximal 4-cycle of $G$.
	If $D$ is any 5-cycle in $G$ such that both $\Int(D) \subset \ext(C)$ and $\Int(D)$ is dominated by two or fewer vertices of $D$, then $\Int(D)$ does not contain any vertex of $G$.
	\label{lem_5_cycle_empty}
\end{lem}

\begin{proof}
	By Lemma \ref{lem_5_5}, the interior of $D$ is not dominated by either a single vertex of $D$, or an adjacent pair of vertices in $D$.
	Assume to the contrary that there is a vertex $w$ in $\Int(D)$, and let $u$ and $v$ be two non-adjacent vertices of $D$ that dominate $\Int(D)$. 
	Per Corollary \ref{cor:5_5_vertex}, the vertex $w$ is adjacent to both $u$ and $v$. 
	Thus, there exists some 4-cycle $A$ in $\Int[D]$ that dominates its interior. 
	The cycle $A$ either contradicts the maximality of $C$, or $A$ and $C$ are dislocated.
\end{proof}

\begin{thm}
	Let $G$ be a pentagulation of diameter 3, order $n$ and maximum degree $\Delta \geq 8$. 
	If $G$ contains a 4-cycle, then $n\leq 3\Delta - 1$.
	\label{thm:5_one4cycle}
\end{thm}

\begin{proof}
	Assume to the contrary that $G$ contains a 4-cycle $C_1 = v_1, v_2, v_3, v_4$, and has order $n > 3\Delta -1$. 
	By Theorem \ref{thm:5_two4cycles}, there are no two dislocated 4-cycles in $G$.
	Assume without loss of generality that $C_1$ is interior maximal, and that $\Int(C_1)$ is dominated by $\{v_1, v_3\}$.
	By Theorem \ref{thm:4_cycle_description}, there exist vertices $w_1$ and $w_2$ in $\Int(C_1)$ such that $P_1 : v_1, w_1, w_2, v_3$ is a path in $G$. 
	If every vertex of $G$ is adjacent to either $v_1$ or $v_3$, then $n\leq 2\Delta < 3\Delta - 1$, so there exists some vertex of $G$ in $\ext(C_1)$ which is not adjacent to $v_1$ or to $v_3$. 
	We consider two cases, according to whether or not the vertices $v_2$ and $v_4$ have neighbors in $\ext(C_1)$.
	
	\begin{figure}[h]
\centering
\begin{tikzpicture}
[inner sep=0.9mm, scale=0.9, 
vertex/.style={circle,thick,draw},
dvertex/.style={rectangle,thick,draw, inner sep=1.3mm}, 
thickedge/.style={line width=1.5pt}] 

\node[vertex, thickedge, fill=black!50] (v1) at (0, 1.5) [label=above:$v_1$] {};
\node[vertex, thickedge, fill=black!50] (v3) at (0, -1.5) [label=below:$v_3$] {};
\node[vertex, thickedge] (v2) at (1.5, 0) [label=left:$v_2$] {};
\node[vertex, thickedge] (v4) at (-1.5,0) [label=left:$v_4$] {};

\draw[thickedge] (v1)--(v2)--(v3)--(v4)--(v1);

\node[vertex] (w1) at (0,0.4) [label=left:$w_1$] {};
\node[vertex] (w2) at (0, -0.4) [label=left:$w_2$] {};

\draw (v1)--(w1)--(w2)--(v3) ;

\node[vertex] (y1) at (2.5, 0) [label=below:$y_1$] {};
\node[vertex] (y2) at (3.5, 0) [label=below:$y_2$] {};

\draw (v2)--(y1)--(y2);

\node at (1, -2.5) {(1)};

\begin{scope}[shift={(7,0)}]

\node[vertex, thickedge, fill=black!50] (v1) at (0, 1.5) [label=above:$v_1$] {};
\node[vertex, thickedge, fill=black!50] (v3) at (0, -1.5) [label=below:$v_3$] {};
\node[vertex, thickedge] (v2) at (1.5, 0) [label=left:$v_2$] {};
\node[vertex, thickedge] (v4) at (-1.5,0) [label=left:$v_4$] {};

\draw[thickedge] (v1)--(v2)--(v3)--(v4)--(v1);

\node[vertex] (w1) at (0,0.4) [label=left:$w_1$] {};
\node[vertex] (w2) at (0, -0.4) [label=left:$w_2$] {};

\draw (v1)--(w1)--(w2)--(v3) ;

\node[vertex] (y1) at (2.5, 0) [label=below:$y_1$] {};
\node[vertex] (y2) at (3.5, 0) [label=below:$y_2$] {};

\draw (v2)--(y1)--(y2);

\node[vertex] (y3) at (2, 1.5) [label=below:$y_3$] {};
\node[vertex] (y4) at (2, -1.5) [label=above:$y_4$] {};

\draw (y2)--(y3)--(v1) (y2)--(y4)--(v3);

\node at (1, -2.5) {(2)};

\end{scope}

\end{tikzpicture}
\caption{In Case 1, since the vertex $y_1$ is not an end-vertex, there exists some neighbour $y_2$ of $y_1$ (1). 
Since the diameter of $G$ is 3, it contains $y_2-w_1$ and $y_2-w_2$ paths, forcing the subgraph $\mathcal{G}$ (2).}
		\label{fig:5_one4cycle}
	\end{figure}
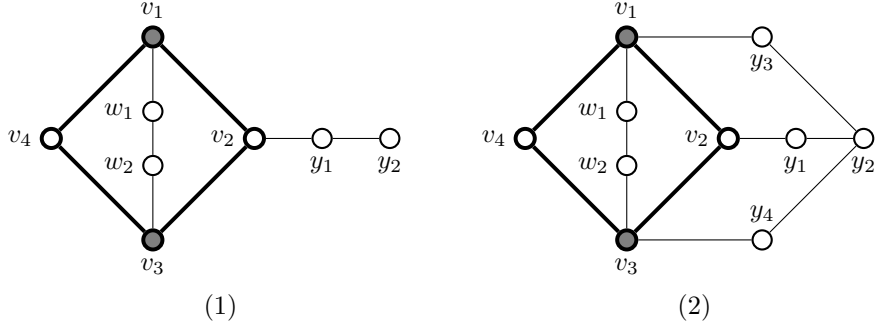
	
	\textit{Case 1:} The vertex $v_2$ has at least one neighbor in $\ext(C_1)$.\\
	Let $y_1$ be a vertex in the exterior of $C_1$ that is adjacent to $v_2$. 
	The vertex $y_1$ is not adjacent to either $v_1$ or $v_3$ as this induces a triangle, contradicting Corollary \ref{cor:5_3_notriangle}. 
	Further, $y_1$ is not adjacent to $v_4$ as this induces a 4-cycle on the vertices $v_2, y_1, v_3, v_4$, contradicting the fact that $G$ does not contain two dislocated 4-cycles. 
	Since $G$ is 2-connected, there is some vertex $y_2$ in $\ext(C_1)$ to which $y_1$ is adjacent (See Figure \ref{fig:5_one4cycle} (1)).
	
	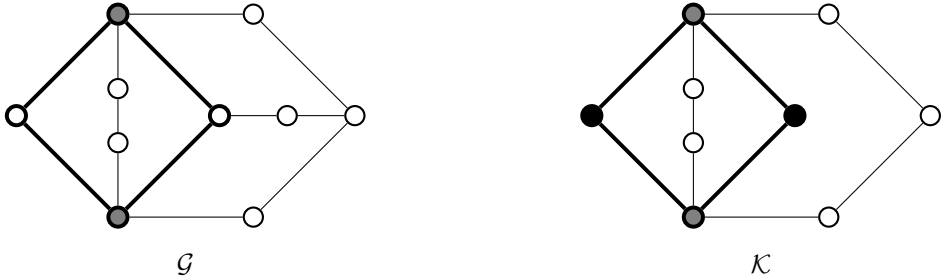
\begin{figure}[h]
\centering
\begin{tikzpicture}
[inner sep=0.9mm, scale=0.9, 
vertex/.style={circle,thick,draw},
dvertex/.style={rectangle,thick,draw, inner sep=1.3mm}, 
thickedge/.style={line width=1.5pt}] 

\node[vertex, thickedge, fill=black!50] (v1) at (0, 1.5) {};
\node[vertex, thickedge, fill=black!50] (v3) at (0, -1.5) {};
\node[vertex, thickedge] (v2) at (1.5, 0) {};
\node[vertex, thickedge] (v4) at (-1.5,0) {};

\draw[thickedge] (v1)--(v2)--(v3)--(v4)--(v1);

\node[vertex] (w1) at (0,0.4) {};
\node[vertex] (w2) at (0, -0.4) {};

\draw (v1)--(w1)--(w2)--(v3) ;

\node[vertex] (y1) at (2.5, 0) {};
\node[vertex] (y2) at (3.5, 0) {};

\draw (v2)--(y1)--(y2);

\node[vertex] (y3) at (2, 1.5) {};
\node[vertex] (y4) at (2, -1.5) {};

\draw (y2)--(y3)--(v1) (y2)--(y4)--(v3);

\node at (1, -2.2) {$\mathcal{G}$};

\begin{scope}[shift={(8.5,0)}]

\node[vertex, thickedge, fill=black!50] (v1) at (0, 1.5) {};
\node[vertex, thickedge, fill=black!50] (v3) at (0, -1.5) {};
\node[vertex, thickedge, fill=black!100] (v2) at (1.5, 0) {};
\node[vertex, thickedge, fill=black!100] (v4) at (-1.5,0)  {};

\draw[thickedge] (v1)--(v2)--(v3)--(v4)--(v1);

\node[vertex] (w1) at (0,0.4) {};
\node[vertex] (w2) at (0, -0.4) {};

\draw (v1)--(w1)--(w2)--(v3) ;

\node[vertex] (y2) at (3.5, 0) {};

\node[vertex] (y3) at (2, 1.5) {};
\node[vertex] (y4) at (2, -1.5) {};

\draw (y2)--(y3)--(v1) (y2)--(y4)--(v3);

\node at (1, -2.2) {$\mathcal{K}$};

\end{scope}

\end{tikzpicture}
\caption{If $G$ is a diameter 3 pentagulation that contains some 4-cycle, but no two dislocated 4-cycles, it must contain one of $\mathcal{G}$ or $\mathcal{K}$ as a subgraph, per Cases 1 and 2 respectively in the proof of Theorem \ref{thm:5_one4cycle}. 
	The black vertices of $\mathcal{K}$ are not adjacent to any vertices of $G-\mathcal{K}$.}
		\label{fig:5_one4cycle_preproof}
	\end{figure}
	
	Note that $d(y_2) \geq 2$, and there exist $y_2-w_1$ and $y_2-w_2$ paths of length at most 3. 
	Since $G$ is triangle-free, the vertices $y_2$ and $v_2$ are not adjacent. 
	Further, $y_2$ is not adjacent to either $v_1$ or $v_3$, as this induces a 4-cycle dislocated from $C_1$ on the vertices $v_1, v_2, y_1, y_2$ or $v_3, v_2, y_1, y_2$ respectively. 
	Finally, $y_1$ is not adjacent to $v_4$, as this induces $\mathcal{H}$ as a subgraph of $G$, which yields a contradiction by Lemma \ref{lem_H_as_subgraph}.
	Since no $y_2-w_1$ or $y_2-w_2$ geodesic can be formed with the vertices mentioned thus far, there exist vertices $y_3$ and $y_4$ in $\ext(C_1)$ such that $y_2y_3$, $y_3v_1$, $y_2y_4$ and $y_4v_3$ are edges in $G$ (See Figure \ref{fig:5_one4cycle} (2)). 
	Note that $y_3\neq y_4$, as this would again induce $\mathcal{H}$ as a subgraph of $G$. 
	Let $\mathcal{G}$ denote the subgraph of $G$ constructed thus far (See Figure \ref{fig:5_one4cycle_preproof}).
	Applying Lemma \ref{lem_chordless_part_2}, deduce that $\mathcal{G}$ is an induced subgraph of $G$.
	Thus, the only two vertices of the 5-cycle $C_3: v_1, v_2, y_1, y_2, y_3$ within distance 2 of $w_2$ are $v_1$ and $v_2$, so $\{v_1, v_2\}$ dominates $\Int(C_2)$. 
	Hence, by Lemma \ref{lem_5_5}, there is no vertex in $\Int(C)$. 
	Similarly, there is no vertex in the region bounded by the cycle $C_4: v_2, y_1, y_2, y_4, v_3$.
	Any vertex of $G$ not adjacent to $v_1$ or $v_3$ for which we have not yet accounted lies in the external region of the cycle $C_2 : v_1, y_3, y_2, y_4, v_3, v_4$.
	There are four subcases to consider.
	
	\textit{Case 1.1:} There exists some vertex $u_1$ in $\ext(C_2)$ adjacent to $v_4$.\\
	Since $G$ is triangle-free, $u_1$ is not adjacent to either $v_1$ or $v_3$.
	Because $G$ does not contain two dislocated 4-cycles, $u_1$ is adjacent to neither $y_3$ nor $y_4$.
	Thus, any $u_1-y_1$ geodesic contains the vertex $y_2$. 
	Either $u_1$ is adjacent to $y_2$, or there exists a vertex $u_2$ in the exterior of $C_2$ such that $P_2: u_1, u_2, y_2, y_1$ is a geodesic in $G$.
	If $u_1$ and $y_2$ are adjacent, then $\ext(C_2)$ is subdivided into 2 regions: the region $r_1$ with vertices $u_1$, $v_4$, $v_1$, $y_3$ and $y_2$ on its boundary, and the region $r_2$ with $u_1$, $v_4$, $v_3$, $y_4$ and $y_2$ on its boundary. 
	The subgraph $\mathcal{G}\cup \{u_1, u_1v_4, u_1y_2\}$ is an induced subgraph of $G$, so the only vertices on the boundary of $r_1$ within distance 2 of $w_2$ are the adjacent pair $v_1$ and $v_4$. 
	The region $r_1$ is dominated by two adjacent vertices of the 5-cycle bounding it, so by Lemma \ref{lem_5_5}, $r_1$ is empty. 
	Similarly, the region $r_2$ is empty, so every vertex of $G$ not yet mentioned is adjacent to either $v_1$ or $v_3$, and we can bound the order of $G$:
	\begin{align*}
	n &= |V(\mathcal{G})\cup \{u_1\}| + |V(G) - V(\mathcal{G}) - \{u_1\}| \\
	&\leq 11 + (d(v_1) - 4) + (d(v_3) - 4) \leq 2\Delta + 3 \leq 3\Delta - 1.
	\end{align*}
	
	This contradicts our assumption, and so the geodesic contains $u_2$ (see Figure \ref{fig:5_one4cycle_2}). 
	Let $\mathcal{G}' = \mathcal{G}\cup P_2 \cup \{u_1v_4\}$. 
	By Lemma \ref{lem_chordless_part_2}, $\mathcal{G}'$ is an induced subgraph of $G$.
	Since $d(u_2, w_1) \leq 3$, there is some vertex $u_3$ that is adjacent to both $u_2$ and $v_1$.
	Similarly, because $d(u_2, w_2) \leq 3$, there exists a vertex $u_4 \neq u_3$ that is adjacent to $u_2$ and $v_4$ (see Figure \ref{fig:5_one4cycle_2}).
	
	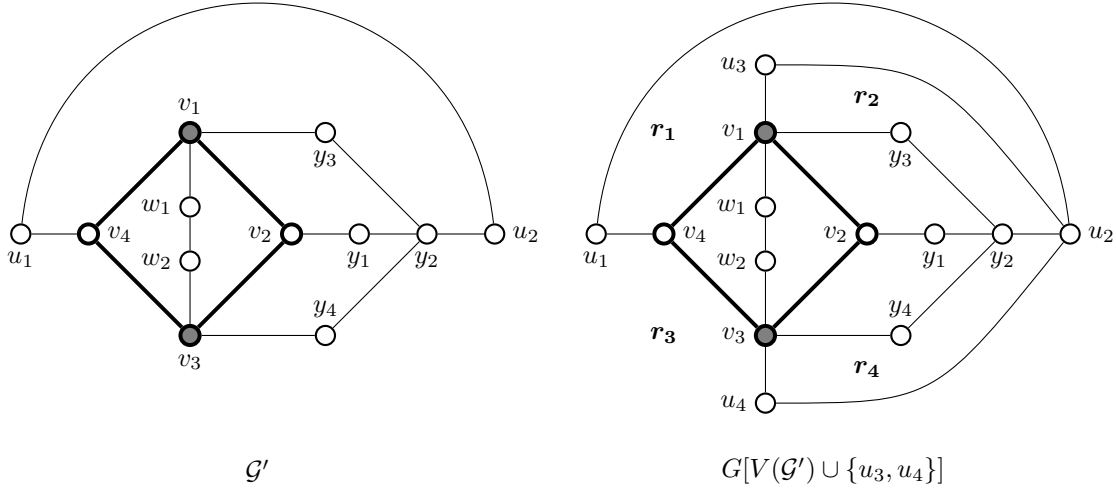
\begin{figure}[h]
\centering
\begin{tikzpicture}
[inner sep=0.9mm, scale=0.9, 
vertex/.style={circle,thick,draw},
dvertex/.style={rectangle,thick,draw, inner sep=1.3mm}, 
thickedge/.style={line width=1.5pt}] 

\node[vertex, thickedge, fill=black!50] (v1) at (0, 1.5) [label=above:$v_1$] {};
\node[vertex, thickedge, fill=black!50] (v3) at (0, -1.5) [label=below:$v_3$] {};
\node[vertex, thickedge] (v2) at (1.5, 0) [label=left:$v_2$] {};
\node[vertex, thickedge] (v4) at (-1.5,0) [label=right:$v_4$] {};

\draw[thickedge] (v1)--(v2)--(v3)--(v4)--(v1);

\node[vertex] (w1) at (0,0.4) [label=left:$w_1$] {};
\node[vertex] (w2) at (0, -0.4) [label=left:$w_2$] {};

\draw (v1)--(w1)--(w2)--(v3) ;

\node[vertex] (y1) at (2.5, 0) [label=below:$y_1$] {};
\node[vertex] (y2) at (3.5, 0) [label=below:$y_2$] {};

\draw (v2)--(y1)--(y2);

\node[vertex] (u1) at (-2.5, 0) [label=below:$u_1$] {};
\node[vertex] (u2) at (4.5, 0) [label=right:$u_2$] {};

\node[vertex] (y3) at (2, 1.5) [label=below:$y_3$] {};
\node[vertex] (y4) at (2, -1.5) [label=above:$y_4$] {};

\draw (y2)--(y3)--(v1) (y2)--(y4)--(v3);

\draw (u1)--(v4) (u2)--(y2);
\draw (u1) .. controls (-2, 4.5) and (4,4.5) .. (u2);

\node at (1, -3.5) {$\mathcal{G}'$};

\begin{scope}[shift={(8.5,0)}]

\node[vertex, thickedge, fill=black!50] (v1) at (0, 1.5) [label=left:$v_1$] {};
\node[vertex, thickedge, fill=black!50] (v3) at (0, -1.5) [label=left:$v_3$] {};
\node[vertex, thickedge] (v2) at (1.5, 0) [label=left:$v_2$] {};
\node[vertex, thickedge] (v4) at (-1.5,0) [label=right:$v_4$] {};

\draw[thickedge] (v1)--(v2)--(v3)--(v4)--(v1);

\node[vertex] (w1) at (0,0.4) [label=left:$w_1$] {};
\node[vertex] (w2) at (0, -0.4) [label=left:$w_2$] {};

\draw (v1)--(w1)--(w2)--(v3) ;

\node[vertex] (y1) at (2.5, 0) [label=below:$y_1$] {};
\node[vertex] (y2) at (3.5, 0) [label=below:$y_2$] {};

\draw (v2)--(y1)--(y2);

\node[vertex] (u1) at (-2.5, 0) [label=below:$u_1$] {};
\node[vertex] (u2) at (4.5, 0) [label=right:$u_2$] {};

\node[vertex] (y3) at (2, 1.5) [label=below:$y_3$] {};
\node[vertex] (y4) at (2, -1.5) [label=above:$y_4$] {};

\draw (y2)--(y3)--(v1) (y2)--(y4)--(v3);

\draw (u1)--(v4) (u2)--(y2);
\draw (u1) .. controls (-2, 4.5) and (4,4.5) .. (u2);

\node[vertex] (u3) at (0, 2.5) [label=left:$u_3$] {};
\node[vertex] (u4) at (0, -2.5) [label=left:$u_4$] {};

\draw (v1)--(u3) .. controls (2.5, 2.5) .. (u2);
\draw (v3)--(u4) .. controls (2.5, -2.5) .. (u2);

\node at (-1.5, 1.5) {\bm{$r_1$}};
\node at (-1.5, -1.5) {\bm{$r_3$}};
\node at (1.5, 2) {\bm{$r_2$}};
\node at (1.5, -2) {\bm{$r_4$}};

\node at (1, -3.5) {$G[V(\mathcal{G}')\cup \{u_3, u_4\}]$};

\end{scope}

\end{tikzpicture}
\caption{In Case 1.1 of the proof of Theorem \ref{thm:5_one4cycle}, we assume that there is a vertex $u_1$ adjacent to $v_4$. 
As a result, we obtain first that $\mathcal{G}'$ is a subgraph of $G$ (left), and then that $G$ also contains the vertices $u_3$ and $u_4$ (right).}
		\label{fig:5_one4cycle_2}
	\end{figure}
	
	The region $\ext(C_2)$ is divided into four subregions, all of which are bounded by 5-cycles.
	Label these regions: $r_1 = \Int(u_1, v_4, v_1, u_3, u_2)$, $r_2 = \Int(v_1, u_3, u_2, y_2, y_3)$, $r_3 = \ext(u_1, v_4, v_3, u_4, u_2)$, $r_4 = \Int(v_3, y_4, y_2, u_2, u_4)$ (see Figure \ref{fig:5_one4cycle_2}). 
	The only two vertices on the boundary of $r_1$ within distance 2 of $w_2$ are $v_1$ and $v_4$. 
	Thus the adjacent pair $\{v_1, v_4\}$ dominates $r_1$, and by Lemma \ref{lem_5_5}, $r_1$ is empty. 
	Similarly, $r_3$ is empty.
	The only vertex on the boundary of $r_2$ within distance 2 of $w_2$ is $v_1$, and so $r_2$ is dominated by $v_1$. 
	Per Lemma \ref{lem_5_5}, the regions $r_2$ and $r_4$ are empty.
	We deduce that all vertices of $G$ not yet mentioned lie in the interior of $C_1$, and hence are adjacent to either $v_1$ or $v_3$.
	This allows us to bound the order of $G$:
	\begin{align*}
	n &= |V(\mathcal{G}')\cup \{u_3, u_4\}| + |V(G) - V(\mathcal{G})' - \{u_3, u_4\}|\\
	&\leq 14 + (d(v_1) - 5) + (d(v_3) - 5) \leq 2\Delta +4 \leq 3\Delta - 1.
	\end{align*}
	
	\textit{Case 1.2:} There is some vertex $u_1$ in $\ext(C_2)$ that is adjacent to $y_2$, but no vertex in $\ext(C_2)$ adjacent to $v_4$.\\
	Since $G$ is triangle-free, $u_1$ is adjacent to neither $y_3$ nor $y_4$. 
	Because $G$ does not contain two dislocated 4-cycles, $u_1$ is adjacent to neither $v_1$ nor $v_3$.
	Because $d(u_1, w_1) \leq 3$ and $d(u_1, w_2) \leq 3$, there are vertices $u_2$ and $u_3$ in $\ext(C_2)$ such that $Q_1: u_1, u_2, v_1$ and $Q_2: u_1, u_3, v_3$ are paths in $G$.
	Note that $u_2\neq u_3$, as this would induce a 4-cycle on the vertex set $\{u_2, v_1, v_4, v_3\}$.
	This 4-cycle is either dislocated from $C_1$, contradicting our assumption, or it is not dislocated from $C_1$, contradicting the maximality of $C_1$. 
	Denote by $\mathcal{G}^*$ the graph $\mathcal{G} \cup Q_1 \cup Q_2 \cup \{y_2u_1\}$, and observe that $\mathcal{G}^*$ is chordless per Lemma \ref{lem_chordless_part_2} (See Figure \ref{fig:5_one4cycle_3}).
	
	\begin{figure}[ht]
\centering
\begin{tikzpicture}
[inner sep=0.9mm, scale=0.9, 
vertex/.style={circle,thick,draw},
dvertex/.style={rectangle,thick,draw, inner sep=1.3mm}, 
thickedge/.style={line width=1.5pt}] 

\node[vertex, thickedge, fill=black!50] (v1) at (0, 1.5) [label=left:$v_1$] {};
\node[vertex, thickedge, fill=black!50] (v3) at (0, -1.5) [label=left:$v_3$] {};
\node[vertex, thickedge] (v2) at (1.5, 0) [label=left:$v_2$] {};
\node[vertex, thickedge, fill=black!100] (v4) at (-1.5,0) [label=right:$v_4$] {};

\draw[thickedge] (v1)--(v2)--(v3)--(v4)--(v1);

\node[vertex] (w1) at (0,0.4) [label=left:$w_1$] {};
\node[vertex] (w2) at (0, -0.4) [label=left:$w_2$] {};

\draw[double] (v1)--(w1)--(w2)--(v3) ;

\node[vertex] (y1) at (2.5, 0) [label=below:$y_1$] {};
\node[vertex] (y2) at (3.5, 0) [label=below:$y_2$] {};

\draw (v2)--(y1)--(y2);

\node[vertex] (u1) at (4.5, 0) [label=right:$u_1$] {};

\node[vertex] (y3) at (2, 1.5) [label=below:$y_3$] {};
\node[vertex] (y4) at (2, -1.5) [label=above:$y_4$] {};

\draw (y2)--(y3)--(v1) (y2)--(y4)--(v3);

\draw (u1)--(y2);

\node[vertex] (u2) at (0, 2.5) [label=left:$u_2$] {};
\node[vertex] (u3) at (0, -2.5) [label=left:$u_3$] {};

\draw (v1)--(u2) .. controls (2.5, 2.5) .. (u1);
\draw (v3)--(u3) .. controls (2.5, -2.5) .. (u1);

\node at (1, -4.1) {$\mathcal{G}^*$};

\begin{scope}[shift={(8.5,0)}]

\node[vertex, thickedge, fill=black!50] (v1) at (0, 1.5) [label=left:$v_1$] {};
\node[vertex, thickedge, fill=black!50] (v3) at (0, -1.5) [label=left:$v_3$] {};
\node[vertex, thickedge] (v2) at (1.5, 0) [label=left:$v_2$] {};
\node[vertex, thickedge, fill=black!100] (v4) at (-1.5,0) [label=right:$v_4$] {};

\draw[thickedge] (v1)--(v2)--(v3)--(v4)--(v1);

\node[vertex] (w1) at (0,0.4) [label=left:$w_1$] {};
\node[vertex] (w2) at (0, -0.4) [label=left:$w_2$] {};

\draw[double] (v1)--(w1)--(w2)--(v3) ;

\node[vertex] (y1) at (2.5, 0) [label=below:$y_1$] {};
\node[vertex] (y2) at (3.5, 0) [label=below:$y_2$] {};

\draw (v2)--(y1)--(y2);

\node[vertex] (u1) at (4.5, 0) [label=below:$u_1$] {};

\node[vertex] (y3) at (2, 1.5) [label=below:$y_3$] {};
\node[vertex] (y4) at (2, -1.5) [label=above:$y_4$] {};

\draw (y2)--(y3)--(v1) (y2)--(y4)--(v3);

\draw (u1)--(y2);

\node[vertex] (u2) at (0.5, 2.5) [label=left:$u_2$] {};
\node[vertex] (u3) at (0.5, -2.5) [label=left:$u_3$] {};

\draw (v1)--(u2) .. controls (2.5, 2.5) .. (u1);
\draw (v3)--(u3) .. controls (2.5, -2.5) .. (u1);

\node[vertex] (x1) at (5.5, 0) [label=right:$x_1$] {};
\node[vertex] (x2) at (-1, 3.5) [label=left:$x_2$] {};
\node[vertex] (x3) at (-1, -3.5) [label=left:$x_3$] {};

\draw (u1)--(x1);
\draw (v1)--(x2) .. controls (2.5, 3.5) .. (x1);
\draw (v3)--(x3) .. controls (2.5, -3.5) .. (x1);

\node at (0.5, 3) {\bm{$r_1$}};
\node at (1.5, 2) {\bm{$r_2$}};
\node at (0.5, -3) {\bm{$r_4$}};
\node at (1.5, -2) {\bm{$r_3$}};

\node at (1, -4.1) {$\mathcal{G}^{**}$};

\end{scope}

\end{tikzpicture}
\caption{In Case 1.2, we obtain first that $\mathcal{G}^*$, and then $\mathcal{G}^{**}$, are subgraphs of $G$.
The black vertex $v_4$ does not have any neighbours in $G$ besides $v_1$ and $v_3$.}
		\label{fig:5_one4cycle_3}
	\end{figure}
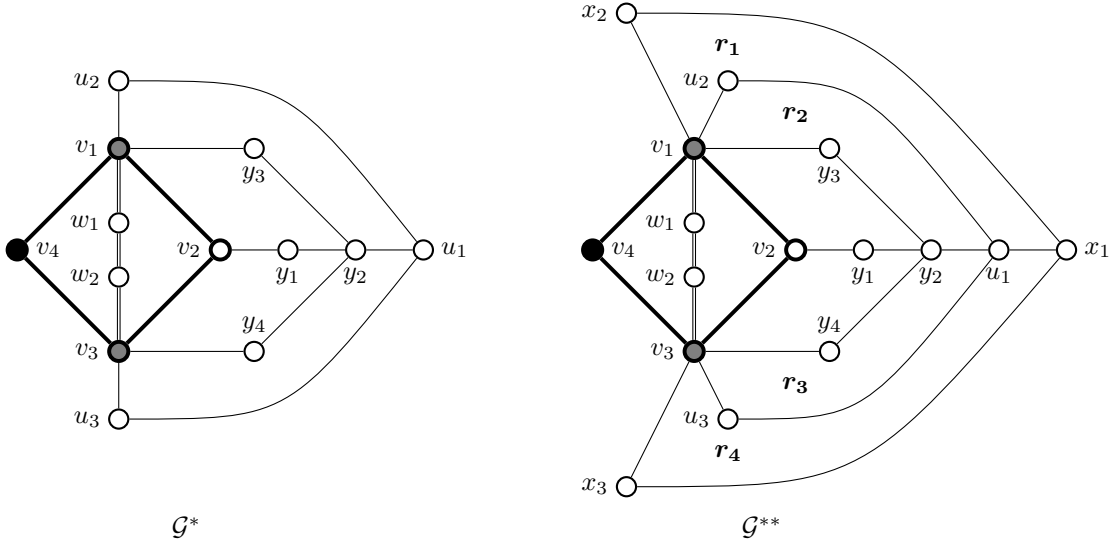
	
	Consider the cycle $C_5: v_1, u_2, u_1, y_2, y_3$. 
	The only vertex on the boundary of $\Int(C_5)$ that is within distance 2 of $w_2$ is $v_1$, so $v_1$ dominates $\Int(C_5)$. 
	By Lemma \ref{lem_5_5}, $\Int(C_5)$ is empty.
	Similarly, the interior of the cycle $C_6: v_3, u_3, u_1, y_2, y_4$ is empty.
	Observe that if every vertex of $G-\mathcal{G}^*$ were adjacent to $v_1$ or $v_3$, then the order of $G$ would be bounded as follows:
	\begin{align*}
	n &= |V(\mathcal{G}^*)| + |V(G) - V(\mathcal{G}^*)|\\
	n &\leq 13 + (d(v_1) - 5) + (d(v_3) - 5) \leq 2\Delta + 3 \leq 3\Delta - 1.
	\end{align*}
	This contradicts our assumption, and thus there is a vertex $x_1$ of $G-\mathcal{G}^*$ not adjacent to $v_1$ or $v_3$. 
	This vertex lies in the face of $\mathcal{G}^*$ bounded by $C_7 = u_2, u_1, u_3, v_3, v_4, v_1$, which we will refer to, without loss of generality, as the exterior of $C_7$.
	Since $\mathcal{G}^*$ is an induced subgraph of $G$, the distance $d_G(y_1, C_7) = 2$, and $\{v_1, v_3, u_1\}$ is the set of vertices of $C_7$ that are distance exactly 2 from $y_1$.
	Because $G$ has diameter 3, we conclude that $x_1$ is adjacent to $u_1$.
	Since $G$ is both triangle-free and does not contain a pair of dislocated 4-cycles, the vertex $x_1$ is not adjacent to any of the vertices of $V(C_7) - \{ u_1 \}$. 
	As $d_G(x_1, w_1)\leq 3$ and $d_G(x_1, w_2)\leq 3$, there exist vertices $x_2$ and $x_3$ in $\ext(C_7)$ such that $Q_3: x_1, x_2, v_1$ and $Q_4: x_1, x_3, v_3$ are paths in $G$.
	These two vertices are distinct, for if they were not, the 4-cycle on $x_2, v_1, v_4, v_3$ would be dislocated from $C_1$, a contradiction.
	Let $\mathcal{G}^{**} = \mathcal{G}^* \cup Q_3 \cup Q_4$ (See Figure \ref{fig:5_one4cycle_3} ($\mathcal{G}^{**}$)).
	We now label the regions of $\mathcal{G}^{**}$ as follows.
	Let $r_1 = \Int(v_1, x_2, x_1, u_1, u_2)$, $r_2 = \Int(v_1, u_2, u_1, y_2, y_3)$, $r_3 = \Int(v_3, u_3,\allowbreak u_1, y_2, y_4)$, $r_4 = \Int(v_3, x_3,\allowbreak x_1, u_1, u_3)$ and $r_0 = \ext(v_1, x_2, x_1, x_3, v_3, v_4)$. 
	Other than $r_0$, all of these regions are bounded by 5-cycles. 
	The regions $r_1$ and $r_2$ are both empty, as the only vertex on either of their boundaries within distance 2 of $w_2$ is $v_1$, and by Lemma \ref{lem_5_5}, no one vertex of a Jordan separating 5-cycle dominates that cycle's interior. 
	Similarly, the regions $r_3$ and $r_4$ are empty as the only vertex on their boundaries within distance 2 of $w_1$ is $v_3$. 
	Any vertex of $r_0$ is adjacent to one of $v_1$ or $v_3$, as these are the only two vertices on the boundary of $r_0$ within distance 2 of $y_1$. 
	Thus all vertices of $G-\mathcal{G}^{**}$ are adjacent to either $v_1$ or $v_3$.
	This contradicts yields the following contradiction, and shows no vertex of $\ext(C_2)$ is adjacent to $y_2$:
	\begin{align*}
	n 	&= |V(\mathcal{G}^{**})| + |V(G) - V(\mathcal{G}^{**})|\\
		&\leq 16 + (d(v_1) -6) + (d(v_3) - 6)) \leq 2\Delta - 4 \leq 3\Delta - 1.
	\end{align*}
	
	\textit{Case 1.3:} There exists some vertex $u_1$ in $\ext(C_2)$ that is adjacent to $y_3$, and no vertex of $\ext(C_2)$ is adjacent to either $y_2$ or $v_4$.\\
	Since $G$ contains neither any 3-cycles, nor any pair of dislocated 4-cycles, the vertex $u_1$ is not adjacent to any vertex of $C_2-\{v_3\}$.
	Thus there are only two ways we can have $d(u_1, w_2) \leq 3$.
	Either $G$ contains the edge $u_1v_3$, or there is some vertex $u_2$ in $\ext(C_2)$ such that $S_1: y_3, u_1, u_2, v_3$ is a path in $G$ (see Figure \ref{fig:5_one4cycle_4}).
	
	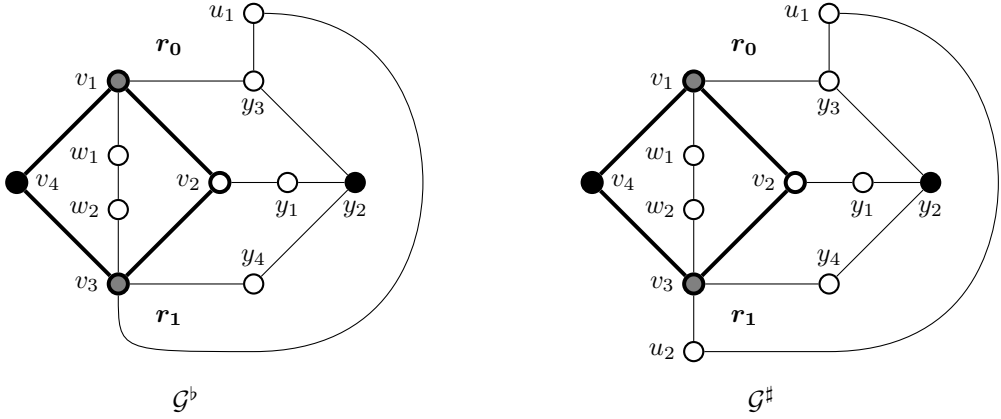
\begin{figure}[h]
\centering
\begin{tikzpicture}
[inner sep=0.9mm, scale=0.9, 
vertex/.style={circle,thick,draw},
dvertex/.style={rectangle,thick,draw, inner sep=1.3mm}, 
thickedge/.style={line width=1.5pt}] 

\node[vertex, thickedge, fill=black!50] (v1) at (0, 1.5) [label=left:$v_1$] {};
\node[vertex, thickedge, fill=black!50] (v3) at (0, -1.5) [label=left:$v_3$] {};
\node[vertex, thickedge] (v2) at (1.5, 0) [label=left:$v_2$] {};
\node[vertex, thickedge, fill=black!100] (v4) at (-1.5,0) [label=right:$v_4$] {};

\draw[thickedge] (v1)--(v2)--(v3)--(v4)--(v1);

\node[vertex] (w1) at (0,0.4) [label=left:$w_1$] {};
\node[vertex] (w2) at (0, -0.4) [label=left:$w_2$] {};

\draw (v1)--(w1)--(w2)--(v3) ;

\node[vertex] (y1) at (2.5, 0) [label=below:$y_1$] {};
\node[vertex, fill=black!100] (y2) at (3.5, 0) [label=below:$y_2$] {};

\draw (v2)--(y1)--(y2);

\node[vertex] (y3) at (2, 1.5) [label=below:$y_3$] {};
\node[vertex] (y4) at (2, -1.5) [label=above:$y_4$] {};

\draw (y2)--(y3)--(v1) (y2)--(y4)--(v3);

\node[vertex] (u1) at (2, 2.5) [label=left:$u_1$] {};

\draw (u1)--(y3);

\draw (u1) .. controls (4.0, 2.5) and (4.5, 1) .. (4.5,0) .. controls (4.5, -1) and (4, -2.5) .. (2, -2.5) .. controls (0, -2.5) .. (v3);

\node at (0.75, 2) {\bm{$r_0$}};
\node at (0.75, -2) {\bm{$r_1$}};

\node at (1, -3.2) {$\mathcal{G}^\flat$};

\begin{scope}[shift={(8.5,0)}]

\node[vertex, thickedge, fill=black!50] (v1) at (0, 1.5) [label=left:$v_1$] {};
\node[vertex, thickedge, fill=black!50] (v3) at (0, -1.5) [label=left:$v_3$] {};
\node[vertex, thickedge] (v2) at (1.5, 0) [label=left:$v_2$] {};
\node[vertex, thickedge, fill=black!100] (v4) at (-1.5,0) [label=right:$v_4$] {};

\draw[thickedge] (v1)--(v2)--(v3)--(v4)--(v1);

\node[vertex] (w1) at (0,0.4) [label=left:$w_1$] {};
\node[vertex] (w2) at (0, -0.4) [label=left:$w_2$] {};

\draw (v1)--(w1)--(w2)--(v3) ;

\node[vertex] (y1) at (2.5, 0) [label=below:$y_1$] {};
\node[vertex, fill=black!100] (y2) at (3.5, 0) [label=below:$y_2$] {};

\draw (v2)--(y1)--(y2);

\node[vertex] (y3) at (2, 1.5) [label=below:$y_3$] {};
\node[vertex] (y4) at (2, -1.5) [label=above:$y_4$] {};

\draw (y2)--(y3)--(v1) (y2)--(y4)--(v3);

\node[vertex] (u1) at (2, 2.5) [label=left:$u_1$] {};
\node[vertex] (u2) at (0, -2.5) [label=left:$u_2$] {};

\draw (u1)--(y3);
\draw (u2)--(v3);

\draw (u1) .. controls (4.0, 2.5) and (4.5, 1) .. (4.5,0) .. controls (4.5, -1) and (4, -2.5) .. (2, -2.5) -- (u2);

\node at (0.75, 2) {\bm{$r_0$}};
\node at (0.75, -2) {\bm{$r_1$}};

\node at (1, -3.2) {$\mathcal{G}^\sharp$};

\end{scope}

\end{tikzpicture}
\caption{Case 1.3 assumes that there is a vertex $u_1$ adjacent to $y_1$. 
	In this case, either $\mathcal{G}^\flat$ or $\mathcal{G}^\sharp$ is a subgraph of $G$.
The black vertices may not have neighbours in $G$ not shown in the diagrams. }
		\label{fig:5_one4cycle_4}
	\end{figure}
	
	Suppose that $u_1$ and $v_3$ are adjacent. 
	Denote by $S_2$ the path $y_3, u_1, v_3$, and let $\mathcal{G}^{\flat} = \mathcal{G}\cup S_2$.
	By applications of Lemma \ref{lem_chordless_part_2}, $\mathcal{G}^\flat$ is an induced subgraph of $G$.
	The path $S_2$ divides $\ext(C_2)$ into two regions bounded by 5-cycles, $r_0 = \ext(v_1, y_3, u_1, v_3, v_4)$ and $r_1 = \Int(y_3, u_1, v_3, y_4, y_2)$. 
	The only vertices on the boundary of $r_0$ within distance 2 of $y_1$ are $v_1$, $v_3$ and $y_3$, so any vertex in $r_0$ is adjacent to one of these three. 
	The only vertices on the boundary of $r_1$ within distance 2 of $w_1$ are $v_3$ and $y_3$, so the set $\{v_3, y_3\}$ dominates $r_1$, and we can bound the order of $G$.
	\begin{align*}
	n 	&= |V(\mathcal{G}^\flat)| + |V(G) - (V(\mathcal{G}^\flat)|\\
		&\leq 11 + (d(v_1) - 4) + (d(v_3) - 5) + (d(y_3) - 3) \leq 3\Delta - 1.
	\end{align*}
	Since this contradicts our assumption, the graph $G$ contains the path $S_1$. 
	Let $\mathcal{G}^\sharp = \mathcal{G}\cup S_1$, and observe per Lemma \ref{lem_chordless_part_2} that $\mathcal{G}^\sharp$ is an induced subgraph of $G$.
	The region $\ext(C_2)$ is divided into two sub-regions bounded by 6-cycles, $r_0 = \ext(v_1, y_3, u_1, u_2, v_3, v_4)$ and $r_1 = \Int(y_3, u_1, u_2, v_3, y_4, y_2)$. 
	The are only two vertices, $y_3$ and $v_3$, on the 6-cycle bounding $r_1$ within distance 2 of $w_1$.
	Thus $\{y_3, v_3\}$ dominates $r_1$, and so by Lemma \ref{lem_5_6_vertex}, there is some vertex $u_3$ in $r_1$ that is adjacent to both $y_3$ and $v_3$. 
	Let $\mathcal{G}^{\sharp \sharp} =\mathcal{G}^{\sharp}\cup \{u_3, u_3y_3, u_3v_3\}$.
	The only vertices on the boundary of $r_0$ within distance 2 of $y_1$ are $v_1$, $v_3$ and $y_3$, so every vertex of $r_0$ is adjacent to one of these three vertices. 
	Thus:
	\begin{align*}
	n &= |V(\mathcal{G}^{\sharp \sharp})| + |V(G) - (V(\mathcal{G}^{\sharp \sharp})|\\
	&\leq 13 + (d(v_1) - 4) + (d(v_3) - 6) + (d(y_3) - 4) \leq 3\Delta - 1.
	\end{align*}
	This contradicts our assumption, and hence $y_3$ does not have a neighbor in $\ext(C_2)$.
	By the same argument, the vertex $y_4$ also does not have a neighbor in $\ext(C_2)$.
	
	\textit{Case 1.4:} The vertices $v_4$, $y_2$, $y_3$ and $y_4$ do not have any neighbors in $\ext(C_2)$.\\
	By cases 1.1 to 1.3, the only vertices of $C_2$ that can have neighbors in $\ext(C_2)$ are $v_1$ and $v_3$. 
	Further, both $v_1$ and $v_3$ are at distance 2 from $y_1$, so any vertex in $\ext(C_2)$ is adjacent to either $v_1$ or $v_3$ in order to be within distance 3 of $y_1$. 
	Hence we get the following bound on $n$:
	\begin{align*}
	n &= |V(\mathcal{G})| + |V(G) - V(\mathcal{G})|\\
	&\leq 10 + (d(v_1) - 4) + (d(v_3) - 4) \leq 2\Delta + 2 \leq 3\Delta - 1.
	\end{align*}
	In all subcases, $n\leq 3\Delta - 1$, and so the vertex $v_2$ does not have a neighbor in $\ext(C_1)$.
	By symmetry, we further conclude that $v_4$ does not have any neighbors in $\ext(C_1)$.
	
	\textit{Case 2:} Neither $v_2$ nor $v_4$ have any neighbors in $G$ besides $v_1$ and $v_3$.\\
	As $n>3\Delta - 1$, there is some vertex $y_1$ in $G$ that is not adjacent to either $v_1$ or $v_3$. 
	Note that $d(y_1, C_1) > 1$, but $d(y_1, w_1) \leq 3$ and $d(y_1, w_2) \leq 3$.
	Therefore, there exist vertices $y_2$ and $y_3$ in the exterior of $C_1$ such that $P_2: y_1, y_2, v_1$ and $P_3: y_1, y_3, v_3$ are paths in $G$ (See Figure \ref{fig:5_one4cycle_5} ($\mathcal{K}$)). 
	Note that $y_2\neq y_3$.
	If $y_2 = y_3$, then there is a 4-cycle on $y_2, v_1, v_2, v_3$, contradicting either the maximality of $C_1$, or the assumption that $G$ does not contain two dislocated 4-cycles. 
	Let $\mathcal{K} = C_1 \cup P_1 \cup P_2 \cup P_3$, and name the cycle $C_2: v_1, y_2, y_1, y_3, v_3, v_4$ (See Figure \ref{fig:5_one4cycle_5}).
	Observe that, by Lemma \ref{lem_chordless_part_2}, the subgraph $\mathcal{K}$ is an induced subgraph of $G$.
	Since $n> 3\Delta - 1$ by assumption, there exists some vertex $u_1$ in $G-\mathcal{K}$ that is not adjacent to either $v_1$ or $v_3$. 
	We may assume without loss of generality that $u_1$ is in $\ext(C_2)$. 
	The vertex $u_1$ is not adjacent to both of $y_2$ and $y_3$ as this creates a 4-cycle dislocated from $C_1$, contradicting our assumption. 
	There are two cases to consider. 
	
	\textit{Case 2.1:} The vertex $u_1$ is adjacent to $y_2$.\\
	Since $G$ contains neither triangles nor dislocated 4-cycles, $u_1$ is not adjacent to any vertex of $C_2-\{y\}$. 
	Since $d_G(u_1, w_2) \leq 3$, there is some vertex $u_2$ in $\ext(C_2)$ such that $Q_1: y_2, u_1, u_2, v_3$ is a path in $G$. 
	Per Lemma \ref{lem_chordless_part_2}, the graph $\mathcal{K}\cup Q_1$ is an induced subgraph of $G$.
	Thus the interior of the 6-cycle $C_3: y_2, u_1, u_2, v_3, y_3, y_1$ is dominated by $y_2$ and $v_3$, as these are the only vertices of the cycle within distance 2 of $w_1$. 
	By Lemma \ref{lem_5_6_vertex}, there exists a vertex $u_3$ in $\Int(C_3)$ such that $Q_2: y_2, u_3, v_3$ is a path in $G$. 
	The path $Q_2$ divides the region bounded by $C_3$ into two regions, each bounded by a 5-cycle. 
	By Corollary \ref{cor:5_5_vertex}, neither region contains any vertex of $G$.
	Let $\mathcal{K}'$ denote the graph $\mathcal{K}\cup Q_1 \cup Q_2$  (See Figure \ref{fig:5_one4cycle_5}), and observe by Lemma \ref{lem_chordless_part_2} that it is an induced subgraph of $G$.
	
	\begin{figure}[h]
\centering
\begin{tikzpicture}
[inner sep=0.9mm, scale=0.9, 
vertex/.style={circle,thick,draw},
dvertex/.style={rectangle,thick,draw, inner sep=1.3mm}, 
thickedge/.style={line width=1.5pt}] 

\node[vertex, thickedge, fill=black!50] (v1) at (0, 1.5) [label=left:$v_1$] {};
\node[vertex, thickedge, fill=black!50] (v3) at (0, -1.5) [label=left:$v_3$] {};
\node[vertex, thickedge, fill=black!100] (v2) at (1.5, 0) [label=left:$v_2$] {};
\node[vertex, thickedge, fill=black!100] (v4) at (-1.5,0) [label=right:$v_4$] {};

\draw[thickedge] (v1)--(v2)--(v3)--(v4)--(v1);

\node[vertex] (w1) at (0,0.4) [label=left:$w_1$] {};
\node[vertex] (w2) at (0, -0.4) [label=left:$w_2$] {};

\draw (v1)--(w1)--(w2)--(v3) ;

\node[vertex] (y1) at (3.5, 0) [label=below:$y_1$] {};

\node[vertex] (y2) at (2, 1.5) [label=below:$y_2$] {};
\node[vertex] (y3) at (2, -1.5) [label=above:$y_3$] {};

\draw (y1)--(y2)--(v1) (y1)--(y3)--(v3);

\node at (1, -3.0) {$\mathcal{K}$};


\begin{scope}[shift={(8.5,0)}]

\node[vertex, thickedge, fill=black!50] (v1) at (0, 1.5) [label=left:$v_1$] {};
\node[vertex, thickedge, fill=black!50] (v3) at (0, -1.5) [label=left:$v_3$] {};
\node[vertex, thickedge, fill=black!100] (v2) at (1.5, 0) [label=left:$v_2$] {};
\node[vertex, thickedge, fill=black!100] (v4) at (-1.5,0) [label=right:$v_4$] {};

\draw[thickedge] (v1)--(v2)--(v3)--(v4)--(v1);

\node[vertex] (w1) at (0,0.4) [label=left:$w_1$] {};
\node[vertex] (w2) at (0, -0.4) [label=left:$w_2$] {};

\draw (v1)--(w1)--(w2)--(v3) ;

\node[vertex] (y1) at (3.5, 0) [label=below:$y_1$] {};

\node[vertex] (y2) at (2.5, 1.5) [label=below:$y_2$] {};
\node[vertex] (y3) at (2.5, -1.5) [label=above:$y_3$] {};

\draw (y1)--(y2)--(v1) (y1)--(y3)--(v3);

\node[vertex, fill=white] (u1) at (4.5, 0.5) [label=above:$u_1$] {};
\node[vertex] (u2) at (4.5, -1.5) [label=below:$u_2$] {};
\node[vertex] (u3) at (4.1, -0.85) [label=below:$u_3$] {};

\draw (v3) to[bend right] (u2);
\draw (u2) to[bend right] (u1);
\draw (u1) to[bend right] (y2);
\draw (v3) to[bend right = 40] (u3);
\draw (u3) to[bend right] (y2);

\node at (1, -3.0) {$\mathcal{K}'$};

\end{scope}

\end{tikzpicture}
\caption{In Case 2, neither $v_4$ nor $v_2$ have neighbours other than $v_1$ and $v_3$.
In this Case, $G$ contains $\mathcal{K}$ as a subgraph.
In Case 2.1, $G$ contains $\mathcal{K}'$ as a subgraph.}
		\label{fig:5_one4cycle_5}
	\end{figure}
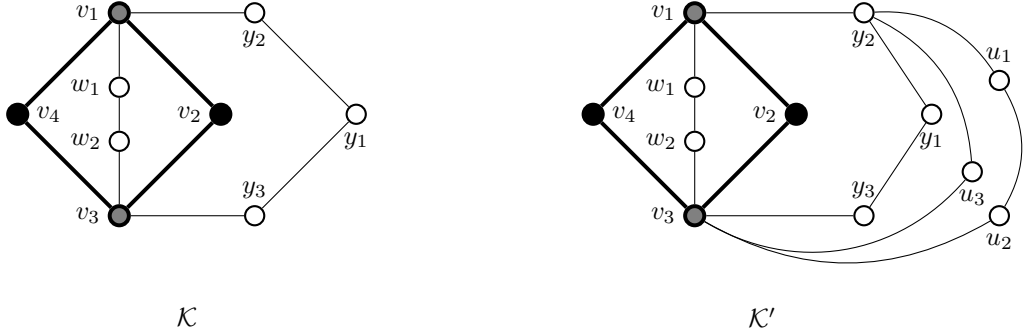
	
	If every vertex of $G-\mathcal{K}'$ is adjacent to one of $v_1$, $v_3$ or $y_2$, then we obtain the following contradiction:
	\begin{align*}
	n \leq 12 + (d(v_1) - 4) + (d(v_3) - 6) + (d(y_2) - 4) \leq 3\Delta - 2.
	\end{align*}
	So there exists some vertex $x_1$ not adjacent to any of $v_1$, $v_3$ or $y_2$. 
	Noting the symmetry between the interior of the cycle $C_4: v_1, y_2, y_1, y_3, v_3, v_2$ and the exterior of the cycle $C_5: v_1, y_2, u_1, u_2, v_3, v_4$, we may assume without loss of generality that $x_1$ is in the interior of $C_4$. 
	
	\begin{figure}[h]
\centering
\begin{tikzpicture}
[inner sep=0.9mm, scale=0.9, 
vertex/.style={circle,thick,draw},
dvertex/.style={rectangle,thick,draw, inner sep=1.3mm}, 
thickedge/.style={line width=1.5pt}] 

\node[vertex, thickedge, fill=black!50] (v1) at (0, 1.5) [label=left:$v_1$] {};
\node[vertex, thickedge, fill=black!50] (v3) at (0, -1.5) [label=left:$v_3$] {};
\node[vertex, thickedge, fill=black!100] (v2) at (1.5, 0) [label=left:$v_2$] {};
\node[vertex, thickedge, fill=black!100] (v4) at (-1.5,0) [label=right:$v_4$] {};

\draw[thickedge] (v1)--(v2)--(v3)--(v4)--(v1);

\node[vertex] (w1) at (0,0.4) [label=left:$w_1$] {};
\node[vertex] (w2) at (0, -0.4) [label=left:$w_2$] {};

\draw[double] (v1)--(w1)--(w2)--(v3) ;

\node[vertex] (y1) at (3.5, 0) [label=below:$y_1$] {};

\node[vertex] (y2) at (2.5, 1.5) [label=below:$y_2$] {};
\node[vertex] (y3) at (2.5, -1.5) [label=above:$y_3$] {};

\draw (y1)--(y2)--(v1) (y1)--(y3)--(v3);

\node[vertex, fill=white] (u1) at (4.5, 0.5) [label=above:$u_1$] {};
\node[vertex] (u2) at (4.5, -1.5) [label=below:$u_2$] {};
\node[vertex] (u3) at (4.1, -0.85) [label=below:$u_3$] {};

\draw (v3) to[bend right] (u2);
\draw (u2) to[bend right] (u1);
\draw (u1) to[bend right] (y2);
\draw (v3) to[bend right = 40] (u3);
\draw (u3) to[bend right] (y2);

\node[vertex] (x1) at (2.75, 0) [label=below:$x_1$] {};
\node[vertex] (x2) at (1.75, 1) [label=below:$x_2$] {};
\node[vertex] (x3) at (1.75, -1) [label=above:$x_3$] {};

\draw (y1) -- (x1) -- (x2) -- (v1) (x1) -- (x3) -- (v3);

\node at (2, -3.0) {$\mathcal{K}''$};


\begin{scope}[shift={(7.5,0)}]

\node[vertex, thickedge, fill=black!50] (v1) at (0, 1.5) [label=left:$v_1$] {};
\node[vertex, thickedge, fill=black!50] (v3) at (0, -1.5) [label=left:$v_3$] {};
\node[vertex, thickedge, fill=black!100] (v2) at (1.5, 0) [label=left:$v_2$] {};
\node[vertex, thickedge, fill=black!100] (v4) at (-1.5,0) [label=right:$v_4$] {};

\draw[thickedge] (v1)--(v2)--(v3)--(v4)--(v1);

\node[vertex] (w1) at (0,0.4) [label=left:$w_1$] {};
\node[vertex] (w2) at (0, -0.4) [label=left:$w_2$] {};

\draw[double] (v1)--(w1)--(w2)--(v3) ;

\node[vertex] (y1) at (4.5, 0) [label=below:$y_1$] {};

\node[vertex] (y2) at (3.5, 1.5) [label=below:$y_2$] {};
\node[vertex] (y3) at (3.5, -1.5) [label=right:$y_3$] {};

\draw (y1)--(y2)--(v1) (y1)--(y3)--(v3);

\node[vertex, fill=white] (u1) at (5.5, 0.5) [label=above:$u_1$] {};
\node[vertex] (u2) at (5.5, -1.5) [label=below:$u_2$] {};
\node[vertex] (u3) at (5.1, -0.85) [label=below:$u_3$] {};

\draw (v3) to[bend right] (u2);
\draw (u2) to[bend right] (u1);
\draw (u1) to[bend right] (y2);
\draw (v3) to[bend right = 40] (u3);
\draw (u3) to[bend right] (y2);

\node[vertex] (x1) at (2.85, -0.5) [label=left:$x_1$] {};
\node[vertex] (x2) at (2.3, 0.5) [label=left:$x_2$] {};
\node[vertex] (x3) at (3.4, 0.25) [label=right:$x_3$] {};

\draw (y3) -- (x1) --(x2) -- (v1);
\draw (y3) to[bend right=15] (x3);
\draw (x3) to[bend right=15] (v1);

\node at (3, -3.0) {$\mathcal{K}'''$};

\end{scope}

\end{tikzpicture}
\caption{In Case 2.1.1, $G$ has the graph $\mathcal{K}''$ as a subgraph.
In Case 2.1.2, the graph $\mathcal{K}'''$ is a subgraph of $G$.}
		\label{fig:5_one4cycle_6}
	\end{figure}
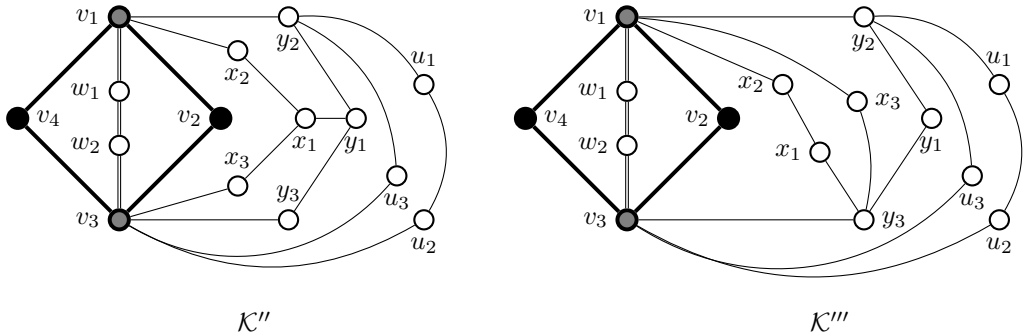
	
	\textit{Case 2.1.1:} The vertex $x_1$ is adjacent to $y_1$.\\
	Since $G$ contains neither triangles nor dislocated 4-cycles, $x_1$ has no neighbors in $C_4-\{y_1\}$.
	Since there exist $x_1-w_1$ and $x_1-w_2$ geodesics, there are vertices $x_2$ and $x_3$ in $\Int(C_4)$ such that $Q_3: y_1, x_1, x_2, v_1$ and $Q_4: y_1, x_1, x_3, v_3$ are paths in $G$. 
	Since $C_1$ is maximal and $G$ does not contain dislocated 4-cycles, the vertices $x_2$ and $x_3$ are distinct.
	Denote $\mathcal{K}'' = \mathcal{K'} \cup Q_3 \cup Q_4$ (See Figure \ref{fig:5_one4cycle_6}).
	
	The exterior of the cycle on $v_1, y_2, u_1, u_2, v_3, v_4$ is dominated by $\{v_1, v_3, y_2\}$, as these are the only vertices of the cycle within distance 2 of $x_1$. 
	The two regions bounded by the 5-cycles on $v_1, y_2, y_1, x_1, x_2$ and $v_3, y_3, y_1, x_1, x_3$ do not contain any vertices by Lemma \ref{lem_5_5}, as only $v_1$ of the former cycle is within distance 2 of $w_2$, and only $v_3$ of the latter is within distance 2 of $w_1$. 
	Finally, the 6-cycle on the vertices $v_1, x_2, x_1, x_3, v_3, v_2$ is dominated by $v_1$ and $v_3$, as these are the only two vertices of the cycle within distance 2 of $u_1$. 
	Thus every vertex of $G-\mathcal{K}''$ is adjacent to $v_1$, $v_3$ or $y_2$, and we obtain a contradiction:
	\begin{align*}
	n &= |V(\mathcal{K}'')| + |V(G) - V(\mathcal{K}'')|\\
	&\leq 15 + (d(v_1) - 5) + (d(v_3) - 7) + (d(y_2) - 4) \leq 3\Delta -1.
	\end{align*}
	
	\textit{Case 2.1.2:} The vertex $x_1$ is adjacent to $y_3$.\\
	The vertex $x_1$ is not adjacent to any vertex of $\mathcal{K}-\{y_3\}$.
	Since $d_G(x_1, w_1) \leq 3$, there exists a vertex $x_2$ such that $Q_5: y_3, x_1, x_2, v_1$ a path in $G$. 
	Consider the 6-cycle $C_6: v_1, y_2, y_1, y_3, x_1, x_2$. 
	The only vertices of $C_6$ within distance 2 of $w_2$ are $v_1$ and $y_3$. 
	So by Lemma \ref{lem_5_6_vertex}, there is a vertex $x_3$ in $\Int(C_6)$ such that $Q_6: v_1, x_3, y_3$ is a path in $G$. 
	The path $Q_6$ divides $\Int(C_6)$ into two regions bounded by 5-cycles, both dominated by $\{v_1, y_3\}$.
	Denote $C_7 : v_1, y_2, u_1, u_2, v_3, v_4$.
	The only vertices of $C_7$ within distance 2 of $x_1$ are $v_1$ and $v_3$, so $\ext(C_7)$ is dominated by $\{v_1, v_3\}$. 
	The interior of the 6-cycle on $v_1, x_2, x_1, y_3, v_3, v_2$ is dominated by $v_1$ and $v_3$, as these are the only two vertices of the cycle within distance 2 of $u_1$. 
	Thus, letting $\mathcal{K}''' = \mathcal{K}' \cup Q_5 \cup Q_6$ (See Figure \ref{fig:5_one4cycle_6}), we derive a contradiction:
	\begin{align*}
	n &= |V(\mathcal{K}''')| + |V(G) - V(\mathcal{K}''')|\\
	&\leq 15 + (d(v_1) + 6) + (d(v_3) - 6) + (d(y_3)-4) \leq 3\Delta - 1.
	\end{align*}
	
	\textit{Case 2.1.3:} The vertex $x_1$ is not adjacent to any vertex of $\mathcal{K}'$.\\
	By the same argument as in Case 2.1.1, there are distinct vertices $x_1$ and $x_2$ in $\Int(C_2)$ such that $Q_7: x_1, x_2, v_1$ and $Q_8: x_1, x_3, v_3$ are paths in $G$. 
	Denote $\mathcal{K}^* = \mathcal{K}'\cup Q_7\cup Q_8$ of $G$ and consider the cycle $C_8: x_1, x_2, v_1, y_2, y_1, y_3, v_3, x_3.$
	By Lemma \ref{lem_chordless_part_2}, the interior of $C_8$ is the only region of $\mathcal{K}^*$ that may contain a chord of $\mathcal{K}^*$.
	Because $G$ contains neither triangles nor dislocated 4-cycles, and $x_1$ is not adjacent to $y_1$, the only possible chords of $\mathcal{K}^*$ are $x_2y_3$ and $x_3y_2$.
	Since $d(u_1, x_1) \leq 3$, either $x_3$ is adjacent to $y_2$, or there is some vertex $y_4$ adjacent to both $x_1$ and $y_2$.
	
	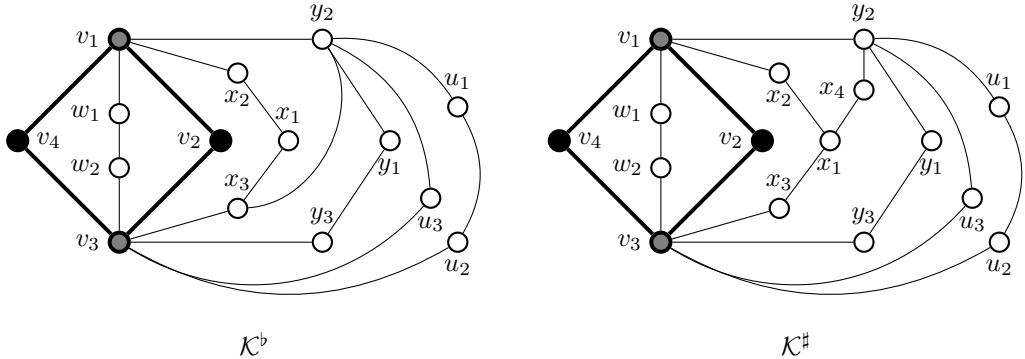
\begin{figure}[h]
\centering
\begin{tikzpicture}
[inner sep=0.9mm, scale=0.9, 
vertex/.style={circle,thick,draw},
dvertex/.style={rectangle,thick,draw, inner sep=1.3mm}, 
thickedge/.style={line width=1.5pt}] 

\node[vertex, thickedge, fill=black!50] (v1) at (0, 1.5) [label=left:$v_1$] {};
\node[vertex, thickedge, fill=black!50] (v3) at (0, -1.5) [label=left:$v_3$] {};
\node[vertex, thickedge, fill=black!100] (v2) at (1.5, 0) [label=left:$v_2$] {};
\node[vertex, thickedge, fill=black!100] (v4) at (-1.5,0) [label=right:$v_4$] {};

\draw[thickedge] (v1)--(v2)--(v3)--(v4)--(v1);

\node[vertex] (w1) at (0,0.4) [label=left:$w_1$] {};
\node[vertex] (w2) at (0, -0.4) [label=left:$w_2$] {};

\draw (v1)--(w1)--(w2)--(v3) ;

\node[vertex] (y1) at (4, 0) [label=below:$y_1$] {};

\node[vertex] (y2) at (3, 1.5) [label=above:$y_2$] {};
\node[vertex] (y3) at (3, -1.5) [label=above:$y_3$] {};

\draw (y1)--(y2)--(v1) (y1)--(y3)--(v3);

\node[vertex, fill=white] (u1) at (5, 0.5) [label=above:$u_1$] {};
\node[vertex] (u2) at (5, -1.5) [label=below:$u_2$] {};
\node[vertex] (u3) at (4.6, -0.85) [label=below:$u_3$] {};

\draw (v3) to[bend right] (u2);
\draw (u2) to[bend right] (u1);
\draw (u1) to[bend right] (y2);
\draw (v3) to[bend right = 40] (u3);
\draw (u3) to[bend right] (y2);

\node[vertex] (x1) at (2.5, 0) [label=above:$x_1$] {};
\node[vertex] (x2) at (1.75, 1) [label=below:$x_2$] {};
\node[vertex] (x3) at (1.75, -1) [label=above:$x_3$] {};

\draw (x1) -- (x2) -- (v1) (x1) -- (x3) -- (v3);
\draw (x3) to[bend right=60] (y2);

\node at (2, -3.0) {$\mathcal{K}^\flat$};


\begin{scope}[shift={(8,0)}]

\node[vertex, thickedge, fill=black!50] (v1) at (0, 1.5) [label=left:$v_1$] {};
\node[vertex, thickedge, fill=black!50] (v3) at (0, -1.5) [label=left:$v_3$] {};
\node[vertex, thickedge, fill=black!100] (v2) at (1.5, 0) [label=left:$v_2$] {};
\node[vertex, thickedge, fill=black!100] (v4) at (-1.5,0) [label=right:$v_4$] {};

\draw[thickedge] (v1)--(v2)--(v3)--(v4)--(v1);

\node[vertex] (w1) at (0,0.4) [label=left:$w_1$] {};
\node[vertex] (w2) at (0, -0.4) [label=left:$w_2$] {};

\draw (v1)--(w1)--(w2)--(v3) ;

\node[vertex] (y1) at (4, 0) [label=below:$y_1$] {};

\node[vertex] (y2) at (3, 1.5) [label=above:$y_2$] {};
\node[vertex] (y3) at (3, -1.5) [label=above:$y_3$] {};

\draw (y1)--(y2)--(v1) (y1)--(y3)--(v3);

\node[vertex, fill=white] (u1) at (5, 0.5) [label=above:$u_1$] {};
\node[vertex] (u2) at (5, -1.5) [label=below:$u_2$] {};
\node[vertex] (u3) at (4.6, -0.85) [label=below:$u_3$] {};

\draw (v3) to[bend right] (u2);
\draw (u2) to[bend right] (u1);
\draw (u1) to[bend right] (y2);
\draw (v3) to[bend right = 40] (u3);
\draw (u3) to[bend right] (y2);

\node[vertex] (x1) at (2.5, 0) [label=below:$x_1$] {};
\node[vertex] (x2) at (1.75, 1) [label=below:$x_2$] {};
\node[vertex] (x3) at (1.75, -1) [label=above:$x_3$] {};
\node[vertex] (x4) at (3, 0.75) [label=left:$x_4$] {};

\draw (x1) -- (x2) -- (v1) (x1) -- (x3) -- (v3);
\draw (x1) -- (x4) -- (y2);

\node at (2, -3.0) {$\mathcal{K}^\sharp$};

\end{scope}

\end{tikzpicture}
\caption{In the first sub-case of 2.1.3, the vertices $y_2$ and $x_3$ are adjacent, and $G$ contains the subgraph $\mathcal{K}^\flat$.
	In the second sub-case, there is a vertex $x_4$ adjacent to both $x_1$ and $y_2$, and $G$ contains the subgraph $\mathcal{K}^\sharp$.}
		\label{fig:5_one4cycle_7}
	\end{figure}
	
	\textit{Subcase 2.1.3 - 1:} The vertices $x_3$ and $y_2$ are adjacent.\\
	Observe by Lemma \ref{lem_chordless_part_2} that $\mathcal{K}^* \cup \{x_3y_2\}$ is an induced subgraph of $G$.
	Since $d(x_2, y_3) \leq 3$, there is a vertex $x_4$ adjacent to both $v_3$ and $x_2$. 
	Denote $\mathcal{K}^\flat = \mathcal{K}^* \cup \{x_4, x_3y_2, v_3x_4, x_2x_4\}$.
	The exterior of the cycle on $v_1, y_2, u_1, u_2, v_3, v_4$ is dominated by $v_1$, $v_3$ and $y_2$, as these are the only vertices of the cycle within distance 2 of $x_1$. 
	The interior of the 5-cycle on $v_1, x_2, x_4, v_3, v_2$ is dominated by $v_1$ and $v_3$, as only these vertices of the cycle are within distance 2 of $u_1$. 
	The cycle on $x_2, x_1, x_3, v_3, x_4$ is dominated by $x_3$ and $v_3$ as these are the only two vertices within distance 2 of $u_1$, and so by Lemma \ref{lem_5_5} the interior of this cycle contains no vertices. 
	The interior of the 5-cycle on $y_2, y_1, y_3, v_3, x_3$ is dominated by $v_3$ and $y_2$, as only these vertices of the cycle are distance 2 from $w_1$. 
	The interior of 5-cycle on $v_1, y_2, x_3, x_1, x_2$ is also empty by Lemma \ref{lem_5_5}, as only $y_2$ and $x_3$ are within distance 2 of $y_3$. 
	Since the vertices of $G$ not in $\mathcal{K}^\flat$ are all adjacent to one of $v_1$, $v_3$ or $y_2$, we can bound the order of $G$.
	\begin{align*}
	n 	&= |V(\mathcal{K}^\flat)| + |V(G) - V(\mathcal{K}^\flat)|\\
		&\leq 16 + (d(v_1) - 5) + (d(v_3) - 8) + (d(y_2) - 5) \leq 3\Delta - 2.
	\end{align*}
	
	\textit{Subcase 2.1.3 - 2:} The graph $G$ contains a vertex $x_4$ that is adjacent to $x_1$ and $y_2$.\\
	Let $\mathcal{K}^\sharp$ be the subgraph $\mathcal{K}^*\cup \{x_4, x_1x_4, y_2x_4\}$ of $G$, and observe per Lemma \ref{lem_chordless_part_2} that $\mathcal{K}^\sharp$ is an induced subgraph of $G$.
	The exterior of the cycle on $v_1, y_2, u_1, u_2, v_3, v_4$ is dominated by $v_1$, $v_3$ and $y_2$, as these are the only vertices of the cycle within distance 2 of $x_1$. 
	The 7-cycle on $y_2, y_1, y_3, v_3, x_3, x_1, x_4$ is dominated by $y_2$ and $v_3$ as these are the only vertices within distance 2 of $w_1$. 
	The interior of the 5-cycle on $v_1, y_2, x_4, x_1, x_2$ is empty by Lemma \ref{lem_5_5}, as it is dominated by $v_1$, the only vertex of the cycle within distance 2 of $w_2$. 
	The interior of the 6-cycle on $v_1, x_2, x_1, x_3, v_3, v_2$ is dominated by $v_1$ and $v_3$, the only vertices of the cycle within distance 2 of $u_1$. 
	Every vertex of $G$ that is not in $\mathcal{K}^\sharp$ is adjacent to one of $v_1$, $v_3$ or $y_2$, so the order of $G$ is bounded above:
	\begin{align*}
	n &= |V(\mathcal{K}^\sharp)| + |V(G) - V(\mathcal{K}^\sharp)|\\
	&\leq 16 + (d(v_1) - 5) + (d(v_3) - 7) + (d(y_2) - 5) \leq 3\Delta - 1.
	\end{align*}
	
	\textit{Case 2.2:} The vertex $u_1$ is not adjacent to $y_2$ or $y_3$.\\
	Since $d_G(u_1, w_1) \leq 3$ and $d_G(u_1, w_2) \leq 3$, there exist vertices $u_2$ and $u_3$ in $G$ such that $S_1: u_1, u_2, v_1$ and $S_2: u_1, u_3, v_3$ are paths in $G$. 
	The vertices $u_2$ and $u_3$ are distinct, by the maximality of $C_1$ and the fact that $G$ contains no dislocated 4-cycles.
	By Case 2.1, neither $y_2$ nor $y_3$ can have a neighbor in $G-\mathcal{K}$ which is not adjacent to $v_1$ or to $v_3$. 
	By symmetry, neither $u_2$ nor $u_3$ can have any neighbor in $G-\{u_1\}$ that is not adjacent to $v_1$ or to $v_3$. 
	Since $G$ contains neither triangles nor dislocated 4-cycles, the only possible chords of the cycle on $v_1, u_2, u_1, u_3, v_3, y_3, y_1, y_2$ are $y_1u_1$, $y_2u_3$ and $y_3u_2$. 
	Up to symmetry, this leaves three possible ways to construct a $u_1-y_1$ geodesic in $G$: with the edge $y_2u_3$, with the edge $u_1y_1$, or by (possibly repeated) subdivision of the edge $u_1y_1$.
	We let $\mathcal{L} = \mathcal{K}\cup S_1 \cup S_2$ (See Figure \ref{fig:5_one4cycle_8}).
	
	\begin{figure}[h]
\centering
\begin{tikzpicture}
[inner sep=0.9mm, scale=0.9, 
vertex/.style={circle,thick,draw},
dvertex/.style={rectangle,thick,draw, inner sep=1.3mm}, 
thickedge/.style={line width=1.5pt}] 

\node[vertex, thickedge, fill=black!50] (v1) at (0, 1.5) [label=left:$v_1$] {};
\node[vertex, thickedge, fill=black!50] (v3) at (0, -1.5) [label=left:$v_3$] {};
\node[vertex, thickedge, fill=black!100] (v2) at (1.5, 0) [label=left:$v_2$] {};
\node[vertex, thickedge, fill=black!100] (v4) at (-1.5,0) [label=right:$v_4$] {};

\draw[thickedge] (v1)--(v2)--(v3)--(v4)--(v1);

\node[vertex] (w1) at (0,0.4) [label=left:$w_1$] {};
\node[vertex] (w2) at (0, -0.4) [label=left:$w_2$] {};

\draw (v1)--(w1)--(w2)--(v3) ;

\node[vertex] (y1) at (3.5, 0) [label=below:$y_1$] {};
\node[vertex] (y2) at (2.5, 1.5) [label=right:$y_2$] {};
\node[vertex] (y3) at (2.5, -1.5) [label=right:$y_3$] {};

\draw (y1)--(y2)--(v1) (y1)--(y3)--(v3);

\node[vertex] (u1) at (5, 0) [label=right:$u_1$] {};
\node[vertex] (u2) at (4, 2.5) [label=above:$u_2$] {};
\node[vertex] (u3) at (4, -2.5) [label=below:$u_3$] {};

\draw (u1)--(u2)--(v1) (u1)--(u3)--(v3);

\node at (2, -3.0) {$\mathcal{L}$};


\begin{scope}[shift={(8,0)}]

\node[vertex, thickedge, fill=black!50] (v1) at (0, 1.5) [label=left:$v_1$] {};
\node[vertex, thickedge, fill=black!50] (v3) at (0, -1.5) [label=left:$v_3$] {};
\node[vertex, thickedge, fill=black!100] (v2) at (1.5, 0) [label=left:$v_2$] {};
\node[vertex, thickedge, fill=black!100] (v4) at (-1.5,0) [label=right:$v_4$] {};

\draw[thickedge] (v1)--(v2)--(v3)--(v4)--(v1);

\node[vertex] (w1) at (0,0.4) [label=left:$w_1$] {};
\node[vertex] (w2) at (0, -0.4) [label=left:$w_2$] {};

\draw (v1)--(w1)--(w2)--(v3) ;

\node[vertex] (y1) at (3.5, 0) [label=below:$y_1$] {};
\node[vertex] (y2) at (2.5, 1.5) [label=right:$y_2$] {};
\node[vertex] (y3) at (2.5, -1.5) [label=right:$y_3$] {};

\draw (y1)--(y2)--(v1) (y1)--(y3)--(v3);

\node[vertex] (u1) at (5, 0) [label=right:$u_1$] {};
\node[vertex] (u2) at (4, 2.5) [label=above:$u_2$] {};
\node[vertex] (u3) at (4, -2.5) [label=below:$u_3$] {};

\draw (u1)--(u2)--(v1) (u1)--(u3)--(v3);

\draw (u3) to[bend right=45] (y2);

\node[vertex] (x1) at (2.25, 0.5) [label=above:$x_1$] {};

\draw (y3) -- (x1) -- (v1);

\node at (2, -3.0) {$\mathcal{L}'$};

\end{scope}

\end{tikzpicture}
\caption{The graph $G$ contains the subgraph $\mathcal{L}$ in Case 2.2.
It contains the subgraph $\mathcal{L}'$ in Case 2.2.1.}
		\label{fig:5_one4cycle_8}
	\end{figure}
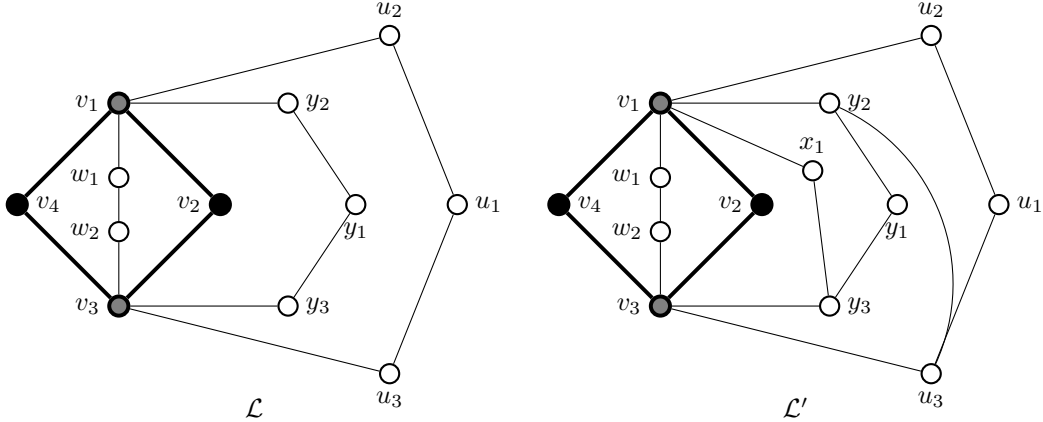
	
	\textit{Case 2.2.1:} The vertices $y_2$ and $u_3$ are adjacent.\\
	Per Lemma \ref{lem_chordless_part_2}, the subgraph $\mathcal{L}\cup \{y_2u_3\}$ is an induced subgraph of $G$.
	Since $d_G(y_3, u_2) \leq 3$, there exists some vertex $x_1$ in $G$ such that either $S_3: y_3, x_1, v_1$ or $S_4: y_3, v_3, x_1, u_2$ is a path in $G$. 
	Up to relabeling of the vertices and choosing the region bounded by $v_1, y_2, y_1, y_3, v_3, v_2$ to be the exterior region of our subgraph, these possibilities are the same. 
	Hence we assume without loss of generality that $S_3$ is a $y_3-u_2$ geodesic, and we denote by $\mathcal{L}'$ the graph $\mathcal{L}\cup \{y_2u_3\} \cup S_3$ (See Figure \ref{fig:5_one4cycle_8}). 
	The interior of the 5-cycle on $v_1, v_2, v_3, y_3, x_1$ is dominated by $v_1$ and $v_3$ as these are the only vertices of the cycle within distance 2 of $u_1$. 
	The interiors of the two 5-cycles on $v_1, y_2, y_1, y_3, x_1$ and $v_1, u_2, u_1, u_3, y_2$ are dominated by the pairs $v_1$, $y_3$ and $v_1$, $u_3$ respectively, as these are the only vertices on the cycles within distance 2 of $w_2$. 
	The interior of the 5-cycle on $y_2, u_3, v_3, y_3, y_1$ is dominated by $y_2$ and $v_3$, these being the only vertices of the cycle within distance 2 of $w_1$.
	By Lemma \ref{lem_5_cycle_empty}, all four of the regions mentioned are empty.
	All vertices of $G$ not in $\mathcal{L}'$ lie in the exterior of the cycle on $v_1, u_2, u_1, u_3, v_3, v_4$. 
	The vertices of this cycle within distance 2 of $y_1$ are $v_1$, $v_3$ and $u_3$. 
	Hence:
	\begin{align*}
	n &= |V(\mathcal{L}')| + |V(G) - (V(\mathcal{L}'))|\\
	&\leq 13 + (d(v_1) - 6) + (d(v_3) - 5) + (d(u_3) - 3) \leq 3\Delta - 1.
	\end{align*}
	This contradicts our assumption, so $y_2$ and $u_3$ are not adjacent. 
	By symmetry, $y_3$ and $u_2$ are not adjacent.
	
	\textit{Case 2.2.2:} The vertices $u_1$ and $y_1$ are adjacent.\\
	Note the interiors of the two 5-cycles on $v_1, u_2, u_1, y_1, y_2$ and $v_3, u_3, u_1, y_1, y_3$ are dominated by only the vertices $v_1$ and $v_3$ respectively, these being the only vertices of the cycles within distance 2 of $w_2$ and $w_1$ respectively. 
	Thus by Lemma \ref{lem_5_5}, both interiors are empty.
	Since $n> 3\Delta - 1$, there exists some vertex $x_1$ in $G-\mathcal{L}$ that is not adjacent to $v_1$ or $v_3$. 
	By symmetry between the exterior of the cycle on $v_1, u_2, u_1, u_3, v_3, v_4$ and the interior of the cycle on $v_1, y_2, y_1, y_3, v_3, v_2$, we assume without loss of generality that $x_1$ is in the interior of the latter cycle. 
	By Case 2.1, the vertex $x_1$ is not adjacent to $y_2$ or $y_3$.
	By the same argument as the one at the start of Case 2.2, there exist distinct vertices $x_2$ and $x_3$ in $G$ such that $S_5: x_1, x_2, v_1$ and $S_6: x_1, x_3, v_3$ are paths in $G$.
	Let $\mathcal{L}''$ denote the graph $\mathcal{L} \cup \{y_1u_1\} \cup S_5 \cup S_6$.
	Using both Lemma \ref{lem_chordless_part_2}, and the fact that $G$ contains neither triangles nor dislocated 4-cycles, we see that the only possible chords of $\mathcal{L}''$ are $x_1y_1$, $x_2y_3$ and $x_3y_2$. 
	The only possibilities for an $x_1-u_1$ geodesic of length at most 3 require that $G$ contains the edge $x_1y_1$, or path $x_1, z_1, y_1$, containing some new vertex $z_1$.
	Let $\mathcal{L}^\flat = \mathcal{L}'' \cup \{x_1y_1\}$ and $\mathcal{L}^\sharp = \mathcal{L}'' \cup \{z_1, x_1z_1, z_1y_1\}$ (See Figure \ref{fig:5_one4cycle_9}). 
	
	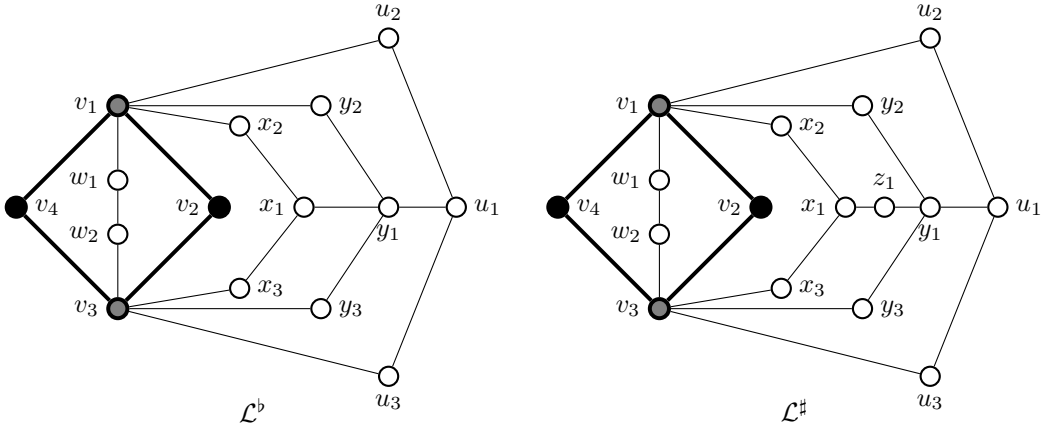
\begin{figure}[h]
\centering
\begin{tikzpicture}
[inner sep=0.9mm, scale=0.9, 
vertex/.style={circle,thick,draw},
dvertex/.style={rectangle,thick,draw, inner sep=1.3mm}, 
thickedge/.style={line width=1.5pt}] 

\node[vertex, thickedge, fill=black!50] (v1) at (0, 1.5) [label=left:$v_1$] {};
\node[vertex, thickedge, fill=black!50] (v3) at (0, -1.5) [label=left:$v_3$] {};
\node[vertex, thickedge, fill=black!100] (v2) at (1.5, 0) [label=left:$v_2$] {};
\node[vertex, thickedge, fill=black!100] (v4) at (-1.5,0) [label=right:$v_4$] {};

\draw[thickedge] (v1)--(v2)--(v3)--(v4)--(v1);

\node[vertex] (w1) at (0,0.4) [label=left:$w_1$] {};
\node[vertex] (w2) at (0, -0.4) [label=left:$w_2$] {};

\draw (v1)--(w1)--(w2)--(v3) ;

\node[vertex] (y1) at (4, 0) [label=below:$y_1$] {};
\node[vertex] (y2) at (3, 1.5) [label=right:$y_2$] {};
\node[vertex] (y3) at (3, -1.5) [label=right:$y_3$] {};

\draw (y1)--(y2)--(v1) (y1)--(y3)--(v3);

\node[vertex] (u1) at (5, 0) [label=right:$u_1$] {};
\node[vertex] (u2) at (4, 2.5) [label=above:$u_2$] {};
\node[vertex] (u3) at (4, -2.5) [label=below:$u_3$] {};

\draw (u1) -- (y1);

\draw (u1)--(u2)--(v1) (u1)--(u3)--(v3);

\node[vertex] (x1) at (2.75, 0) [label=left:$x_1$] {};
\node[vertex] (x2) at (1.8, 1.2) [label=right:$x_2$] {};
\node[vertex] (x3) at (1.8, -1.2) [label=right:$x_3$] {};

\draw (v1)--(x2)--(x1)--(x3)--(v3);
\draw (x1) -- (y1);

\node at (2, -3.0) {$\mathcal{L}^\flat$};


\begin{scope}[shift={(8,0)}]

\node[vertex, thickedge, fill=black!50] (v1) at (0, 1.5) [label=left:$v_1$] {};
\node[vertex, thickedge, fill=black!50] (v3) at (0, -1.5) [label=left:$v_3$] {};
\node[vertex, thickedge, fill=black!100] (v2) at (1.5, 0) [label=left:$v_2$] {};
\node[vertex, thickedge, fill=black!100] (v4) at (-1.5,0) [label=right:$v_4$] {};

\draw[thickedge] (v1)--(v2)--(v3)--(v4)--(v1);

\node[vertex] (w1) at (0,0.4) [label=left:$w_1$] {};
\node[vertex] (w2) at (0, -0.4) [label=left:$w_2$] {};

\draw (v1)--(w1)--(w2)--(v3) ;

\node[vertex] (y1) at (4, 0) [label=below:$y_1$] {};
\node[vertex] (y2) at (3, 1.5) [label=right:$y_2$] {};
\node[vertex] (y3) at (3, -1.5) [label=right:$y_3$] {};

\draw (y1)--(y2)--(v1) (y1)--(y3)--(v3);

\node[vertex] (u1) at (5, 0) [label=right:$u_1$] {};
\node[vertex] (u2) at (4, 2.5) [label=above:$u_2$] {};
\node[vertex] (u3) at (4, -2.5) [label=below:$u_3$] {};

\draw (u1) -- (y1);

\draw (u1)--(u2)--(v1) (u1)--(u3)--(v3);

\node[vertex] (x1) at (2.75, 0) [label=left:$x_1$] {};
\node[vertex] (x2) at (1.8, 1.2) [label=right:$x_2$] {};
\node[vertex] (x3) at (1.8, -1.2) [label=right:$x_3$] {};

\draw (v1)--(x2)--(x1)--(x3)--(v3);

\node[vertex] (z1) at (3.325, 0) [label=above:$z_1$] {};

\draw (x1) -- (z1) -- (y1);

\node at (2, -3.0) {$\mathcal{L}^\sharp$};

\end{scope}

\end{tikzpicture}
\caption{In Cases 2.2.2 and 2.2.3, the graph $G$ always contains $\mathcal{L}^\flat$ as a subgraph. 
If, in Case 2.2.2, $G$ contains $\mathcal{L}^\sharp$ as a subgraph, it will inevitably also have a $\mathcal{L}^\flat$ subgraph.} 
		\label{fig:5_one4cycle_9}
	\end{figure}
	
	Suppose that $G$ contains the path $x_1, z_1, y_1$.
	By Lemma \ref{lem_chordless_part_2}, the subgraph $\mathcal{L}^\sharp$ is an induced subgraph of $G$.
	Since $d_G(z_1, w_1) \leq 3$ and $d_G(z_1, w_2) \leq 3$, there exist vertices $z_2$ and $z_3$ such that $S_7: z_1, z_2, v_1$ and $S_8: z_1, z_3, v_3$ are paths in $G$. 
	By swapping the labels $z_1 \leftrightarrow x_1$, $z_2 \leftrightarrow x_2$ and $z_3 \leftrightarrow x_3$, we obtain $\mathcal{L}^\flat$ as a subgraph of $G$. 
	Thus to complete the proof of Case 2.2.2, it suffices to prove the following claim.
	
	\textit{Claim:} If $G$ contains $\mathcal{L}^\flat$ as a subgraph, then $n\leq 3\Delta - 1$.\\
	Consider the subgraph $\mathcal{L}^\flat$, and note that it is an induced subgraph of $G$ by Lemma \ref{lem_chordless_part_2}.
	There exist $x_2-u_3$ and $x_3-u_2$ geodesics of length at most 3 in $G$. 
	Since $\mathcal{L}^\flat$ is an induced subgraph of $G$, there are only two possible $x_2-u_3$ geodesics, both of which use some vertex $t_1$ in $G-\mathcal{L}^\flat$.
	These possible geodesics are $X_1: x_2, v_1, t_1, u_3$ and $X_2: x_2, t_1, v_3, u_3$. 
	Up to relabeling of the vertices, and making the face of $\mathcal{L}^\flat$ bounded by $v_1, x_2, x_1, x_3, v_3, v_2$ the outer face of the graph, the two plane graphs $\mathcal{L}^\flat \cup X_1$ and $\mathcal{L}^\flat \cup X_2$ are the same.
	Thus we assume without loss of generality that $X_1$ is a geodesic in $G$.
	Per Lemma \ref{lem_chordless_part_2}, the subgraph $\mathcal{L}^\flat \cup X_1$ is an induced subgraph of $G$.
	The only possible $x_3-u_2$ geodesic is $X_3: x_3, t_2, v_1, u_2$, where $t_2$ is not among the vertices mentioned thus far.
	Let $\mathcal{L}^* = \mathcal{L}^\flat \cup X_1 \cup X_2$, and observe that it is an induced subgraph of $G$ by Lemma \ref{lem_chordless_part_2}. 
	The interior of the 5-cycle on $v_1, t_2, x_3, v_3, v_2$ is dominated by $v_1$ and $v_3$, these being the only vertices of the cycle within distance 2 of $u_1$. 
	The interior of the 5-cycle on $v_1, x_2, x_1, x_3, t_2$ is dominated by $v_1$ and $x_3$, as these are the only vertices of the cycle within distance 2 of $w_2$.
	Similarly, the two regions bounded by 5-cycles that contain the vertex $t_1$ are also dominated by just two vertices.
	The interiors of the two 5-cycles on $v_1, y_2, y_1, x_1, x_2$ and $v_3, y_3, y_1, x_1, x_3$ are dominated by only $v_1$ and $v_3$ respectively, these being the only vertices of each cycle within distance 2 of $w_1$ and $w_1$, respectively. 
	Thus, all the regions mentioned above are empty by Lemma \ref{lem_5_cycle_empty}.
	As such, every vertex of $G-\mathcal{L}^*$ is in the interior of $C_1$, and hence adjacent to $v_1$ or to $v_3$. 
	Hence we prove the claim with the following contradiction:
	\begin{align*}
	n &= |V(\mathcal{L}^*)| + |V(G) - V(\mathcal{L}^*)|\\
	&\leq 17 + (d(v_1) - 8) + (d(v_3) - 6) \leq 2\Delta + 3 \leq 3\Delta - 1.
	\end{align*}
	
	\textit{Case 2.2.3:} The $y_1-u_1$ geodesic is the single edge $y_1u_1$, subdivided either once or twice into a path of length 2 or 3 respectively.
	Assume there exists some vertex $x_1$ in $G-\mathcal{L}$ on the path $Y_1: y_1, x_1, u_1$ in $G$, and note that $\mathcal{L}\cup Y_1$ is an induced subgraph of $G$ by Lemma \ref{lem_chordless_part_2}.
	Since the distance between $x_1$ and the vertices $w_1$ and $w_2$ is at most 3, there are paths $x_1, x_2, v_1$ and $x_1, x_3, v_3$ in $G$. 
	But now we see that $\mathcal{L}^\flat$ is a subgraph of $G$, and $n\leq 3\Delta - 1$ by the claim in Case 2.2.2.
	If we instead assume that there are vertices $x_1$ and $z_1$ on the path $Y_2: y_1, x_1, z_1, u_1$, we again see that $\mathcal{L}\cup Y_2$ is an induced subgraph of $G$, and that $G$ contains paths $x_1, x_2, v_1$ and $x_1, x_3, v_3$. 
	Similarly, the graph $G$ will also have paths $z_1, z_2, v_1$ and $z_1, z_3, v_3$, and we see that $G$ contains $\mathcal{L}^\flat$ as a subgraph. 
	Again invoke the claim in Case 2.2.2 to complete the proof.
\end{proof}   

\section{Bounding the order, part III: Not a 4-cycle in sight}
\label{sec:bounding_3}

In this section, we show that a pentagulation $G$ of diameter 3, order $n$ and maximum degree $\Delta \geq 8$ contains at least one 4-cycle.
The restriction $\Delta \geq 8$ is used heavily.
As demonstrated by Figure \ref{fig:delta_need_be_large}, the result is false if $\Delta \leq 3$. 
The author does not know whether or not $n \leq 3\Delta - 1$ when $\Delta$ is between 4 and 7.

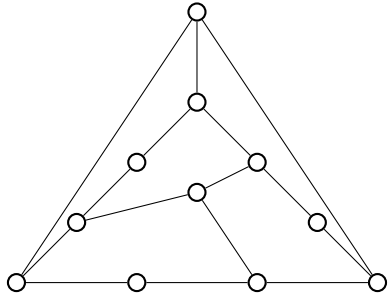
\begin{figure}[h]
\centering
\begin{tikzpicture}
[inner sep=0.8mm, scale=0.8, 
vertex/.style={circle,thick,draw},
dvertex/.style={rectangle,thick,draw, inner sep=1.3mm}, 
thickedge/.style={line width=1.5pt}] 

\node[vertex] (1) at (0,0) {};
\node[vertex] (2) at (2,0) {};
\node[vertex] (3) at (4,0) {};
\node[vertex] (4) at (6,0) {};
\node[vertex] (5) at (1,1) {};
\node[vertex] (6) at (2,2) {};
\node[vertex] (7) at (3,3) {};
\node[vertex] (8) at (4,2) {};
\node[vertex] (9) at (5,1) {};
\node[vertex] (10) at (3,1.5) {};
\node[vertex] (11) at (3,4.5) {};

\draw (1)--(5)--(6)--(7)--(8)--(9)--(4)--(3)--(2)--(1);
\draw (10)--(5) (10)--(3) (10)--(8);
\draw (11)--(1) (11)--(7) (11)--(4);

\end{tikzpicture}
\caption{A pentagulation of diameter 3, girth 5 and maximum degree less than 8.}
\label{fig:delta_need_be_large}
\end{figure}

\begin{lem}
	Let $G$ be a pentagulation with girth 5, and let $v$ be a vertex of $G$.
	Then $N(v)$ is an independent set, every vertex of $N_2(v)$ has a unique neighbor in $N(v)$, and every vertex of $N(v)$ has at least one neighbor in $N_2(v)$.
	\label{lem_nhood_rules}
\end{lem}

\begin{proof}
	Since $G$ contains no triangles, $N(v)$ is an independent set.
	Because $G$ contains no 4-cycles, any vertex of $N_2(v)$ has exactly one neighbor in $N(v)$.
	As $G$ is 2-connected and triangle-free, every vertex of $N(u)$ has a neighbor in $N_2(v)$.
\end{proof}

\begin{lem}
	If $G$ is a pentagulation of girth 5, then $G$ is either the cycle $C_5$, or $G$ does not contain two adjacent vertices of degree 2.
	\label{lem_deg_two_nbours}
\end{lem}

\begin{proof}
	Assume to the contrary that $G$ is a pentagulation of girth 5 other than $C_5$ that contains two adjacent vertices $x$ and $y$ of degree 2.
	Let $w$ be the single vertex of $N_1(x)-\{y\}$ and $z$ the vertex of $N_1(y) - \{x\}$.
	The path $P: w, x, y, z$ lies on the boundary of two distinct faces $f_1$ and $f_2$ of $G$, each bounded by 5-cycles.
	Thus there exist two distinct vertices $u$ and $v$ that are both adjacent to $w$ and $z$.
	Hence there is a 4-cycle $u, w, v, x$, contradicting the girth of $G$.
\end{proof}

Consider a vertex $v$ in a pentagulation $G$.
Let $\bm{\fv}$ be the subgraph of $G$ consisting of the edges and vertices that lie on the boundary of any face incident with $v$. 
Given two vertices $x$ and $y$ of $N_2(v)$, call an $x-y$ path $Q$ of length $k$ a $\bm{k}$\textbf{-chord} (with respect to $v$) if both $(Q-\{x,y\}) \cap N_2(v) = \emptyset$ and $E(Q) \cap E(\fv) = \emptyset$.

For example, consider the subgraph of a girth 5 pentagulation shown in Figure \ref{fig:delta_eight_nhood}.
The path $P: w_1, w_5$ is a 1-chord with respect to $v$, while $Q: w_5, z, w_8$ is a 2-chord.
The edge $w_1w_2$ is not a 1-chord, since it belongs to $\fv$.
Notice that $\fv \cup P$ contains a cycle $C_P: w_1, w_5, u_3, v, u_1$ formed by taking the union of the $w_1-w_5$ 1-chord $P$ and the two unique $v-w_1$ and $v-w_5$ geodesics. 
One can construct another cycle $C_Q: w_5, z, w_8, u_5, v, u_3$ in the same fashion.

\begin{figure}[h]
\centering
\begin{tikzpicture}
[scale = 0.9, inner sep=0.8mm, 
vertex/.style={circle,thick,draw},
dvertex/.style={rectangle,thick,draw, inner sep=1.3mm}, 
thickedge/.style={line width=1.5pt}] 

\node[vertex, fill=black] (v) at (0,0)  [label = 270:{$v$}] {};

\node[vertex] (u1) at (180:2) [label = 270:{$u_1$}] {};
\node[vertex] (u2) at (135:2) [label = 225:{$u_2$}] {};
\node[vertex] (u3) at (90:2) [label = 180:{$u_3$}] {};
\node[vertex] (u4) at (45:2) [label = 315:{$u_4$}] {};
\node[vertex] (u5) at (0:2) [label = 270:{$u_5$}] {};

\node[vertex] (w1) at (180:4) [label = 270:{$w_1$}] {};
\node[vertex] (w2) at (160:4) [label = 90:{$w_2$}] {};
\node[vertex] (w3) at (140:4) [label = 225:{$w_3$}] {};
\node[vertex] (w4) at (120:4) {};
\node[vertex] (w5) at (100:4) [label = 90:{$w_5$}] {};
\node[vertex] (w6) at (80:4) {};
\node[vertex] (w7) at (45:4) {};
\node[vertex] (w8) at (0:4) [label = 270:{$w_8$}] {};

\node[vertex] (z) at (50:5.3) [label = 220:{$z$}] {};

\draw (v)--(u1) (v)--(u2) (v)--(u3) (v)--(u4) (v)--(u5);
\draw (u1)--(w1) (u2)--(w2) (u2)--(w3) (u2)--(w4) (u3)--(w5) (u3)--(w6) (u4)--(w7) (u5)--(w8);
\draw (w1)--(w2) (w4)--(w5) (w6)--(w7) (w7)--(w8);
\draw (w1) .. controls (150:6) and (130:6) .. (w5);
\draw (w5) .. controls (75:5) .. (z) ..controls (25:5) .. (w8);

\draw[dashed] (v)--(210:2) (v)--(-30:2);

\node at (140:6) {$P: w_1, w_5$};
\node at (42:6.3) {$Q: w_5, z, w_8$};

\end{tikzpicture}
\caption{A vertex $v$ in a pentagulation of girth five, and some of the edges and vertices near it.
The dashed lines indicate some edges to parts of the graph not shown.}
\label{fig:delta_eight_nhood}
\end{figure}
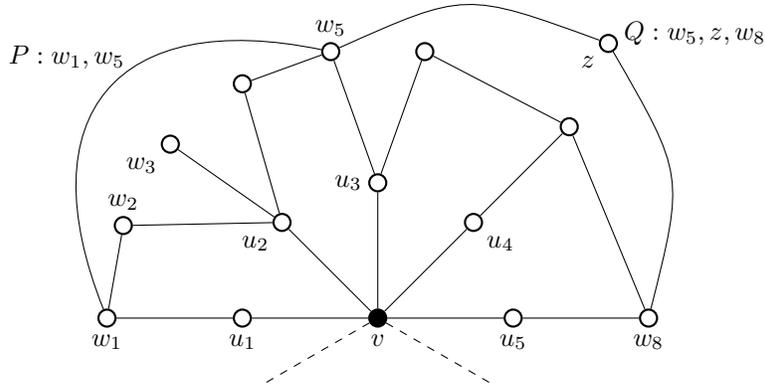

As the next lemma demonstrates, 1-chords and 2-chords with respect to some vertex will always induce cycles in the same manner that $P$ and $Q$ induce $C_P$ and $C_Q$. 

\begin{lem}
	Let $G$ be a pentagulation with girth 5, and let $v$ be a vertex of $G$ such that $d(v) \geq 8$.
	Given distinct vertices $x$ and $y$ of $N_2(v)$, let $P:x, \dots, y$ be a $k$-chord of $v$, and let $u_x$ and $u_y$ denote the unique vertices in $N(v) \cap N(x)$ and $N(v) \cap N(y)$ respectively.
	If $k \leq 2$, then $u_x$ and $u_y$ are distinct, and $P, u_y, v, u_x$ is a Jordan separating cycle.
	\label{lem_unique_nbours_distinct}
\end{lem}

\begin{proof}
	There are unique vertices $u_x$ and $u_y$ as described, per Lemma \ref{lem_nhood_rules}.
	Assume to the contrary that $k \leq 2$, but that $u_x = u_y$.
	The cycle $P, u_y$ has length $k+2 < 5$, which contradicts the fact that $g(G) =5$.
	Thus $u_x \neq u_y$, and so $C_P : P, u_y, v, u_x$ is a cycle.
	It remains to show that $C_P$ is Jordan separating. 
	Since $C_P$ is a cycle of length 5 or 6, and $E(P) \cap E(\fv) = \emptyset$, the cycle $C_P$ is neither a face-cycle ($P$ does not share an edge with a face incident to $v$), nor does it have any chords (as the girth of $G$ is 5).
	Thus $C_P$ is a Jordan separating cycle.
\end{proof}
 
Let $v$ be a vertex of a girth 5 pentagulation, and let the path $Q: x, \dots, y$ be a $k$-chord, for $k \in \{1,2\}$, with respect to $v$. 
If $u_x$ and $u_y$ are the unique vertices of $N(v)$ adjacent to $x$ and $y$ respectively, then the cycle $C_Q : Q, u_y, v, u_x$ is the \textbf{cycle under} $\bm{Q}$. 
The chord $Q$ is said to be \textbf{minimal} if $C_Q$ dominates its interior, and there does not exist any $k$-chord (of the same length) $Q'$ such that $\Int(C_{Q'}) \subset \Int(C_Q)$.

\begin{thm}
	Let $G$ be a diameter 3, girth 5 pentagulation of maximum degree $\Delta$, and let $v$ be a vertex of $G$ with maximum degree. 
	If $\Delta \geq 8$, then there do not exist any 1-chords with respect to $v$.
	\label{thm_len_1_chords}
\end{thm}

\begin{proof}
	We assume to the contrary that there exist vertices $w_0'$ and $w_j'$ in $N_2(v)$, and some 1-chord $Q': w_1', w_j'$ with respect to $v$.
	Label the vertices of $N(v) = \{u_0', u_1', \dots, u_{\Delta-1}' \}$ in clockwise order, so that $u_i'$ and $u_{i+1}'$ always lie on the boundary of the same face (subscripts taken modulo $\Delta$).
	Let $u_0'$ and $u_j'$ be the unique, distinct neighbors of $w_0'$ and $w_j'$ respectively (these exist by Lemmas \ref{lem_nhood_rules} and \ref{lem_unique_nbours_distinct}).
	Let $C_{Q'}$ denote the cycle under $Q'$ with respect to $v$.
	By Lemma \ref{lem_unique_nbours_distinct}, $C_{Q'}$ is a Jordan separating cycle 
	Since the diameter of $G$ is 3, the cycle $C_{Q'}$ dominates either its interior or its exterior. 
	Embed $G$ such that $C_{Q'}$ dominates its interior, and let $Q$ be a minimal 1-chord in $\Int[C_{Q'}]$ (it is possible that $Q = Q'$).
	Relabel the vertices of $N(v)$ and $N_2(v)$ so that the start and end vertices of $Q$ are labeled $w_0$ and $w_j$ respectively, the neighbors $u_i$ of $N(v)$ are still in clockwise order, and $w_0u_0$, $w_ju_j$ are edges of $E(G)$.
	Let $f_i$ be the face incident with $v$ that has vertices $u_i$ and $u_{i+1}$ on its boundary.
	
	\textit{Claim 1:} The inequality $j < 3$ holds (i.e., the interior of $C_Q$ contains at most two faces incident with $v$).\\
	We first assume to the contrary that $j \geq 4$ (See Figure \ref{fig_1_chord_j_geq_4}).
	Let $w_2$ be a vertex of $N_2(v)\cap N(u_2)$ (which exists by Lemma \ref{lem_nhood_rules}).
	Since $C_Q$ dominates its interior, $w_2$ is adjacent to some vertex of $C_Q$. 
	Because $G$ has girth 5, $w_2$ is not adjacent to any of $u_0$, $v$ or $u_j$.
	By the minimality of $Q$, $w_2$ is not adjacent to either $w_0$ or $w_j$, a contradiction.
	
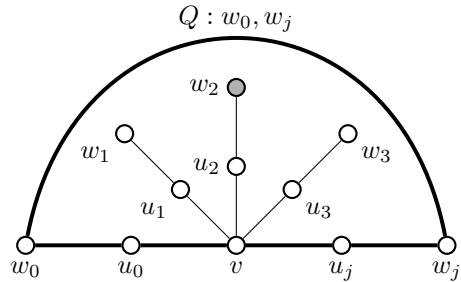
\begin{figure}[h]
\centering
\begin{tikzpicture}
[scale =0.7, inner sep=0.8mm, 
vertex/.style={circle,thick,draw},
dvertex/.style={rectangle,thick,draw, inner sep=1.3mm}, 
thickedge/.style={line width=1.5pt}] 

\node[vertex] (v) at (0,0)  [label = 270:{$v$}] {};

\node[vertex] (u1) at (180:2) [label = 270:{$u_0$}] {};
\node[vertex] (u2) at (135:1.5) [label = 225:{$u_1$}] {};
\node[vertex] (u3) at (90:1.5) [label = 180:{$u_2$}] {};
\node[vertex] (u4) at (45:1.5) [label = 315:{$u_3$}] {};
\node[vertex] (u5) at (0:2) [label = 270:{$u_j$}] {};

\node[vertex] (w1) at (180:4) [label = 270:{$w_0$}] {};
\node[vertex] (w2) at (135:3) [label = 225:{$w_1$}] {};
\node[vertex, fill=black!30] (w3) at (90:3) [label = 180:{$w_2$}] {};
\node[vertex] (w4) at (45:3) [label = 315:{$w_3$}] {};
\node[vertex] (w5) at (0:4) [label = 270:{$w_j$}] {};

\draw (v)--(u1) (v)--(u2) (v)--(u3) (v)--(u4) (v)--(u5);
\draw (w1) .. controls (120:6) and (60:6) .. (w5);
\draw (u1)--(w1) (u2)--(w2) (u3)--(w3) (u4)--(w4) (u5)--(w5);

\draw[thickedge] (w1) .. controls (120:6) and (60:6) .. (w5);
\draw[thickedge] (v)--(u1) (v)--(u5) (u1)--(w1) (u5)--(w5);

\node at (90:4.3) {$Q: w_0, w_j$};

\end{tikzpicture}
\caption{This figure shows Claim 1 of Theorem \ref{thm_len_1_chords}. The cycle $C_Q$ under the 1-chord $Q$ is bold, and the unique $N_2(v)$ neighbour $w_2$ of $u_2$ is grey.}
\label{fig_1_chord_j_geq_4}
\end{figure}
	
	Now suppose for the sake of contradiction that $j = 3$.
	Let $w_1$ be a vertex of $N(u_1) \cap N_2(v)$, and $w_2$ a vertex of $N(u_2) \cap N_2(v)$.
	By minimality of $Q$, $w_1$ is not adjacent to $w_j$.
	Since $G$ has girth 5, $w_1$ is not adjacent to $u_0$, $v$ or $u_j$.
	Because $C_Q$ dominates its interior, $w_1$ is adjacent to $w_0$.
	Similarly, $w_2$ is adjacent to $w_j$, but not to $w_0$.
	This leaves two cases to consider.
	
\begin{figure}[h]
\centering
\begin{tikzpicture}
[scale =0.7, inner sep=0.8mm, 
vertex/.style={circle,thick,draw},
dvertex/.style={rectangle,thick,draw, inner sep=1.3mm}, 
thickedge/.style={line width=1.5pt}] 

\node[vertex] (v) at (0,0)  [label = 270:{$v$}] {};

\node[vertex] (u1) at (180:2) [label = 270:{$u_0$}] {};
\node[vertex] (u2) at (120:1.5) [label = 225:{$u_1$}] {};
\node[vertex] (u3) at (60:1.5) [label = 180:{$u_2$}] {};
\node[vertex] (u5) at (0:2) [label = 270:{$u_j$}] {};

\node[vertex] (w1) at (180:4) [label = 270:{$w_0$}] {};
\node[vertex] (w2) at (120:3) [label = 90:{$w_1$}] {};
\node[vertex] (w3) at (60:3) [label = 90:{$w_2$}] {};
\node[vertex] (w5) at (0:4) [label = 270:{$w_j$}] {};

\draw (v)--(u1) (v)--(u2) (v)--(u3) (v)--(u5);
\draw (w1) .. controls (120:6) and (60:6) .. (w5);
\draw (u1)--(w1) (u2)--(w2) (u3)--(w3) (u5)--(w5);

\draw (w1)--(w2)--(w3)--(w5);

\draw[thickedge] (w1) .. controls (120:6) and (60:6) .. (w5);
\draw[thickedge] (v)--(u1) (v)--(u5) (u1)--(w1) (u5)--(w5);

\node at (90:4.3) {$Q: w_0, w_j$};
\node at (270: 1) {(1)};

\begin{scope}[shift={(10.3,0)}]
	\node[vertex] (v) at (0,0)  [label = 270:{$v$}] {};
	
	\node[vertex] (u1) at (180:2) [label = 270:{$u_0$}] {};
	\node[vertex] (u2) at (120:1.5) [label = 225:{$u_1$}] {};
	\node[vertex] (u3) at (60:1.5) [label = 180:{$u_2$}] {};
	\node[vertex] (u5) at (0:2) [label = 270:{$u_j$}] {};
	
	\node[vertex] (w1) at (180:4) [label = 270:{$w_0$}] {};
	\node[vertex] (w2) at (135:3) [label = 90:{$w_1$}] {};
	\node[vertex] (w2') at (105:3) [label = 90:{$w_1'$}] {};
	\node[vertex] (w3) at (60:3) [label = 90:{$w_2$}] {};
	\node[vertex] (w5) at (0:4) [label = 270:{$w_j$}] {};
	
	\draw (v)--(u1) (v)--(u2) (v)--(u3) (v)--(u5);
	\draw (w1) .. controls (120:6) and (60:6) .. (w5);
	\draw (u1)--(w1) (u2)--(w2) (u3)--(w3) (u5)--(w5);
	
	\draw (w1)--(w2) (u2)--(w2') (w3)--(w5);
	
	\draw[thickedge] (w1) .. controls (120:6) and (60:6) .. (w5);
	\draw[thickedge] (v)--(u1) (v)--(u5) (u1)--(w1) (u5)--(w5);
	
	\node at (90:4.3) {$Q: w_0, w_j$};
	\node at (270: 1) {(2)};
\end{scope}

\end{tikzpicture}
\caption{If $j=3$ in the proof of Claim 1, there are two possibilities.
Either both $u_1$ and $u_2$ have degree two (1), as in Claim 1 Case 1, or one of them has degree at least three (2), as in Claim 1 Case 2.}
\label{fig_1_chord_j_eq_3}
\end{figure}
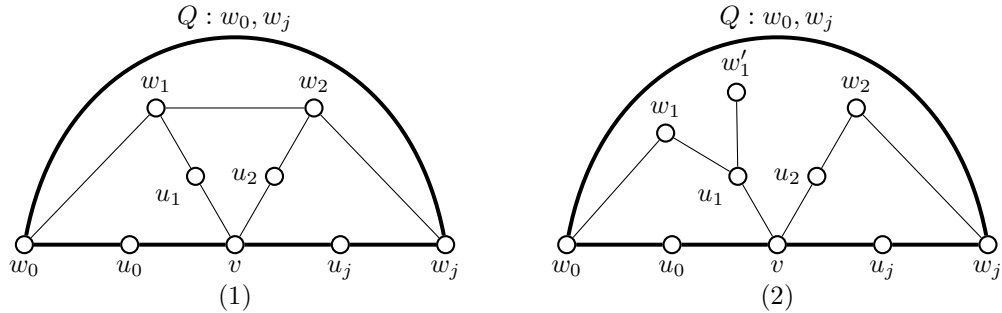
	
	\textit{Claim 1 Case 1:} The degrees of $u_1$ and $u_2$ satisfy $d(u_1) = d(u_2) = 2$.\\
	The path $w_1, u_1, v, u_2, w_2$ lies along the boundary of a face of $G$, so $w_1$ and $w_2$ are adjacent (See Figure \ref{fig_1_chord_j_eq_3} (1)). 
	Thus the vertices $w_0, w_1, w_2, w_j$ lie on a 4-cycle, contradicting the the girth of $G$.\\
	\textit{Claim 1 Case 2:} either $u_1$ or $u_2$ has degree at least three.\\
	Assume without loss of generality that $u_1$ has a vertex $w_1'$ of $N(u_1)\cap N_2(v)$ other than $w_1$ (See Figure \ref{fig_1_chord_j_eq_3} (2)).
	Since $C_Q$ dominates its interior and $G$ has no cycles of length 3 or 4, $w_1'$ is adjacent to either $w_0$ or $w_j$. 
	The cycle under the either the chord $w_0w_1'$ or the chord $w_jw_1'$ is contained strictly in $\Int[C_Q]$, contradicting the minimality of $Q$ and proving Claim 1.
	
	Since $j<3$, there are at least five neighbors $u_3$, $u_4$, \dots, $u_{\Delta - 1}$ of $v$ in $\ext(C_Q)$.
	We consider cases, according to whether or not $w_0$ and $w_j$ have neighbors in $\Int(C_Q)$.
	
	\textit{Case 1:} Neither $w_0$ nor $w_j$ have any neighbors in $\Int(C_Q)$.\\
	In $\Int[C_Q]$, the only neighbors of $w_0$ are $u_0$ and $w_j$, and the only neighbors of $w_j$ are $u_j$ and $w_0$.
	Thus the path $P: u_0, w_0, w_j, u_j$ lies on the boundary of a face contained in $\Int(C_Q)$, so there is a vertex $x$ such that the cycle $P, x$ bounds a face.
	By the assumption that $w_0w_j$ is a 1-chord with respect to $v$, we have $x\neq v$.
	Thus there is a 4-cycle on $v, u_0, x, u_j$, a contradiction (See Figure \ref{fig_1_chord_0_1_nbour}).
	
\begin{figure}[h]
\centering
\begin{tikzpicture}
[scale =0.7, inner sep=0.8mm, 
vertex/.style={circle,thick,draw},
dvertex/.style={rectangle,thick,draw, inner sep=1.3mm}, 
thickedge/.style={line width=1.5pt}] 

\node[vertex] (v) at (0,0)  [label = 270:{$v$}] {};

\node[vertex] (u1) at (135:2) [label = 270:{$u_0$}] {};
\node[vertex] (u5) at (45:2) [label = 270:{$u_j$}] {};

\node[vertex, fill=black] (w1) at (135:4) [label = 270:{$w_0$}] {};
\node[vertex, fill=black] (w5) at (45:4) [label = 270:{$w_j$}] {};

\node[vertex, black!40] (x) at (90:3) [label = 90:{$x$}] {};

\draw (v)--(u1) (v)--(u5);
\draw (w1) .. controls (110:6) and (70:6) .. (w5);
\draw (u1)--(w1) (u5)--(w5);

\draw[thickedge] (w1) .. controls (110:6) and (70:6) .. (w5);
\draw[thickedge] (v)--(u1) (v)--(u5) (u1)--(w1) (u5)--(w5);

\draw[dashed] (u1)--(x)--(u5);

\node at (90:4.4) {$Q$};
\node at (270: 1) {Case 1};

\begin{scope}[shift={(10.3,0)}]
	\node[vertex] (v) at (0,0)  [label = 270:{$v$}] {};
	
	\node[vertex] (u1) at (135:2) [label = 270:{$u_0$}] {};
	\node[vertex] (u5) at (45:2) [label = 270:{$u_j$}] {};
	
	\node[vertex] (w1) at (135:4) [label = 270:{$w_0$}] {};
	\node[vertex, fill=black] (w5) at (45:4) [label = 270:{$w_j$}] {};
	
	\node[vertex] (x) at (105:3) [label = 90:{$x$}] {};
	\node[vertex] (y) at (75:3) [label = 90:{$y$}] {};
	
	\draw (v)--(u1) (v)--(u5);
	\draw (w1) .. controls (110:6) and (70:6) .. (w5);
	\draw (u1)--(w1) (u5)--(w5);
	
	\draw[thickedge] (w1) .. controls (110:6) and (70:6) .. (w5);
	\draw[thickedge] (v)--(u1) (v)--(u5) (u1)--(w1) (u5)--(w5);
	
	\draw (w1)--(x)--(y)--(u5);
	
	\node at (90:4.4) {$Q$};
	\node at (270: 1) {Case 2};
\end{scope}

\end{tikzpicture}
\caption{In Case 1, we assume that neither $w_0$ nor $w_j$ has neighbours in $\Int(C_Q)$ (and colour these vertices black to indicate this).
In Case 2, we assume that $w_0$ has a neighbour in $\Int(C_Q)$, but $w_j$ does not.}
\label{fig_1_chord_0_1_nbour}
\end{figure}
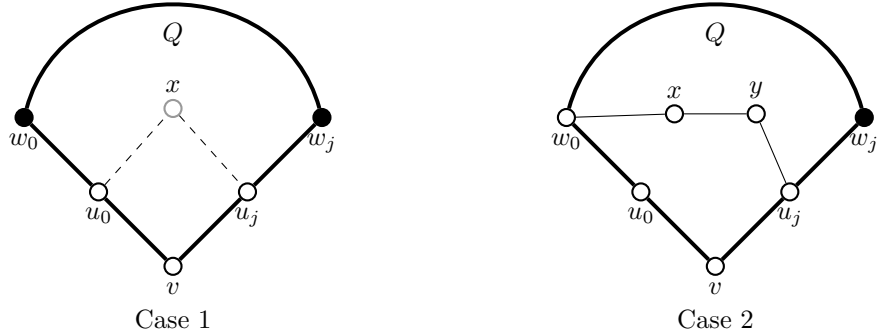
	
	\textit{Case 2:} Either $w_0$ or $w_j$ has a neighbor in $\Int(C_Q)$, but not both.\\
	Assume without loss of generality that there is a vertex $x$ in $\Int(C_Q)$ that is adjacent to $w_0$.
	If there are multiple vertices in $N_1(w_0) \cap \Int(C_Q)$, choose $x$ such that the edges $w_0w_j$ and $w_0x$ lie on the boundary of a common face.
	Because $w_j$ has no neighbor in $\Int(C_Q)$, the path $P: u_j, w_j, w_0, x$ lies on the boundary of some face $f$ in the interior of $C_Q$.
	Thus there is some vertex $y$ in $\Int[C_Q]$ such that the cycle $P, y$ bounds $f$.
	As $G$ has girth 5, the vertex $y$ is in $N_2(v)$ (See Figure \ref{fig_1_chord_0_1_nbour}).
	There are a number of cases to consider, based on the structure of the faces $f_j$ and $f_{j+1}$.	
	
	\textit{Case 2.1:} There is some vertex $s$ in $N_1(w_j) \cap N_1(u_{j+1})$, and $d(u_{j+1}) = 2$.\\
	Let $t$ be the neighbor of $s$ on the boundary of the face $f_{j+1}$, and observe that $t$ and $u_{j+2}$ are adjacent (see Figure \ref{fig_1_chord_21_22}).
	Since the girth of $G$ is 5, we observe the following:\\ 
	(1) the vertex $w_j$ has no neighbors in the cycle $v, u_j, w_j, s, u_{j+1}$ besides $v$ and $w_j$,\\
	(2) the vertex $t$ is not adjacent to either $w_0$ or $w_j$,\\ 
	(3) the vertex $y$ is not adjacent to $u_0$, $w_0$ or $w_j$.
	Thus there is no possible $y-t$ path of length 3 or less, a contradiction.
	
\begin{figure}[h]
\centering
\begin{tikzpicture}
[scale =0.7, inner sep=0.8mm, 
vertex/.style={circle,thick,draw},
dvertex/.style={rectangle,thick,draw, inner sep=1.3mm}, 
thickedge/.style={line width=1.5pt}] 

\node[vertex] (v) at (0,0)  [label = 270:{$v$}] {};

\node[vertex] (u1) at (165:2) [label = 270:{$u_0$}] {};
\node[vertex] (u5) at (75:2) [label = 180:{$u_j$}] {};

\node[vertex] (w1) at (165:4) [label = 270:{$w_0$}] {};
\node[vertex, fill=black] (w5) at (75:4) [label = 180:{$w_j$}] {};

\node[vertex] (x) at (135:3) [label = 90:{$x$}] {};
\node[vertex] (y) at (105:3) [label = 90:{$y$}] {};

\node[vertex] (u6) at (45:2) [label = 300:{$u_{j+1}$}] {};
\node[vertex] (u7) at (15:2) [label = 300:{$u_{j+2}$}] {};

\node[vertex] (w6) at (45:4) [label = 45:{$s$}] {};
\node[vertex] (w7) at (15:4) [label = 45:{$t$}] {};

\draw (v)--(u1) (v)--(u5);
\draw (w1) .. controls (140:6) and (100:6) .. (w5);
\draw (u1)--(w1) (u5)--(w5);

\draw[thickedge] (w1) .. controls (140:6) and (100:6) .. (w5);
\draw[thickedge] (v)--(u1) (v)--(u5) (u1)--(w1) (u5)--(w5);

\draw (w1)--(x)--(y)--(u5);

\draw (w5)--(w6)--(w7);
\draw (v)--(u6)--(w6) (v)--(u7)--(w7);

\node at (120:4.4) {$Q$};
\node at (270: 1) {Case 2.1};

\begin{scope}[shift={(10.3,0)}]
	\node[vertex] (v) at (0,0)  [label = 270:{$v$}] {};
	
	\node[vertex] (u1) at (165:2) [label = 270:{$u_0$}] {};
	\node[vertex] (u5) at (75:2) [label = 180:{$u_j$}] {};
	
	\node[vertex] (w1) at (165:4) [label = 270:{$w_0$}] {};
	\node[vertex, fill=black] (w5) at (75:4) [label = 180:{$w_j$}] {};
	
	\node[vertex] (x) at (135:3) [label = 90:{$x$}] {};
	\node[vertex] (y) at (105:3) [label = 90:{$y$}] {};
	
	\node[vertex] (u6) at (37.5:2) [label = 300:{$u_{j+1}$}] {};
	\node[vertex] (u7) at (0:2) [label = 300:{$u_{j+2}$}] {};
	
	\node[vertex] (w6) at (45:4) [label = 180:{$s$}] {};
	\node[vertex] (w6') at (30:4) [label = 0:{$t$}] {};
	\node[vertex] (w7) at (0:4) [label = 45:{$z$}] {};
	
	\draw (v)--(u1) (v)--(u5);
	\draw (w1) .. controls (140:5.5) and (100:5.5) .. (w5);
	\draw (u1)--(w1) (u5)--(w5);
	
	\draw[thickedge] (w1) .. controls (140:5.5) and (100:5.5) .. (w5);
	\draw[thickedge] (v)--(u1) (v)--(u5) (u1)--(w1) (u5)--(w5);
	
	\draw (w1)--(x)--(y)--(u5);
	
	\draw (w5)--(w6) (w6')--(w7);
	\draw (v)--(u6)--(w6) (v)--(u7)--(w7) (u6)--(w6');
	\draw[dashed] (w1) .. controls (135:8) and (75:8) .. (w6');
	
	\node at (120:4.2) {$Q$};
	\node at (270: 1) {Case 2.2};
\end{scope}

\end{tikzpicture}
\caption{The diagram on the left illustrates Case 2.1, in which $d(u_{j+1}) = 2$ and the vertex of $N_2(v)\cap N(u_{j+1})$ is adjacent to $w_j$.
On the right is Case 2.2, in which $d(u_{j+1}) > 2$, and some vertex of $N_2(v)\cap N(u_{j+1})$ is adjacent to $w_j$.}
\label{fig_1_chord_21_22}
\end{figure}
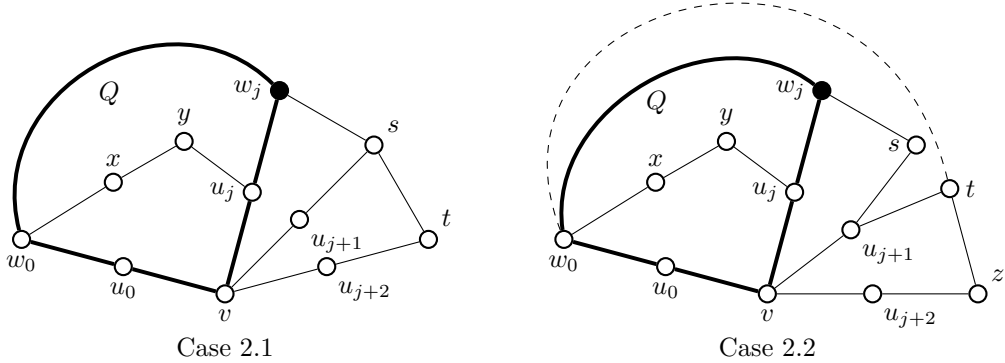
	
	\textit{Case 2.2:} There is a vertex $s$ in $N_1(w_j) \cap N_1(u_{j+1})$, and $d(u_{j+1}) \geq 3$.\\
	Since $u_{j+1}$ has at least two neighbors in $N_2(v)$, the neighbor $t$ of $u_{j+1}$ on the boundary of $f_{j+1}$ that is at distance 2 from $v$ is distinct from $s$.
	Let $z$ be the vertex $N_2(v)-\{t\}$ incident with $f_{j+1}$ (see Figure \ref{fig_1_chord_21_22}).
	Since $G$ has girth 5, $t$ is not adjacent to $w_j$.
	Since $d(t, y) \leq 3$, the vertices $t$ and $w_0$ are adjacent.
	
	Because the diameter of $G$ is 3, the vertices $t$ and $w_0$ are adjacent to ensure that $d(t, y) \leq 3$.
	The vertex $z$ is not adjacent to any vertex within distance 2 of $y$ by planarity, and the fact that $G$ has girth 5.
	Thus $d(z,y) > 3$, contradicting the diameter of $G$.
	
	\textit{Case 2.3:} There is no vertex in $N_1(w_j) \cap N_1(u_{j+1})$.\\
	Let $s$ and $t$ be the vertices of $N_2(v)$, incident with $f_j$, and adjacent to $u_j$ and $u_{j+1}$ respectively.
	Note that $s$ and $t$ are adjacent.
	If $t$ is incident with the face $f_{j+1}$, then $t$ has a neighbor $z$ in $N(u_{j+2})$ that is also incident with $f_{j+1}$ (see Figure \ref{fig_1_chord_23_setup} (1)).
	If $t$ is not incident with $f_{j+1}$, then there is a vertex $z'$ in $N(u_{j+1}) - \{t\}$ that is incident with $f_{j+1}$ (see Figure \ref{fig_1_chord_23_setup} (2)). 
	There are three ways to construct a $t-x$ geodesic of length at most 3.
	
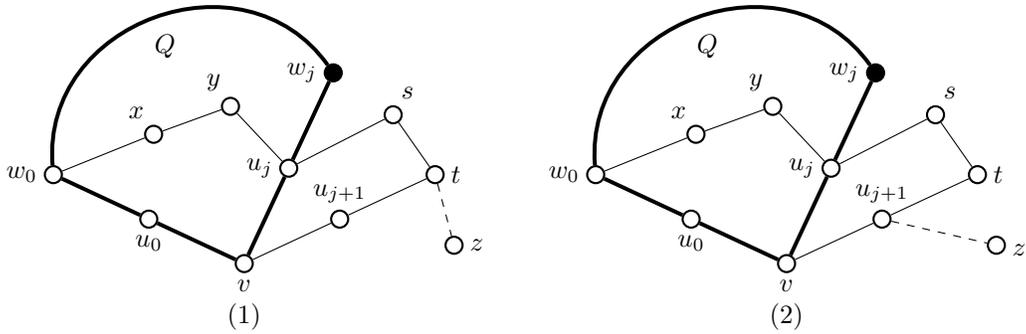
\begin{figure}[h]
\centering
\begin{tikzpicture}
[scale =0.7, inner sep=0.8mm, 
vertex/.style={circle,thick,draw},
dvertex/.style={rectangle,thick,draw, inner sep=1.3mm}, 
thickedge/.style={line width=1.5pt}] 

\node[vertex] (v) at (0,0)  [label = 270:{$v$}] {};

\node[vertex] (u1) at (155:2) [label = 270:{$u_0$}] {};
\node[vertex] (u5) at (65:2) [label = 180:{$u_j$}] {};

\node[vertex] (w1) at (155:4) [label = 180:{$w_0$}] {};
\node[vertex, fill=black] (w5) at (65:4) [label = 180:{$w_j$}] {};

\node[vertex] (x) at (125:3) [label = 110:{$x$}] {};
\node[vertex] (y) at (95:3) [label = 110:{$y$}] {};

\node[vertex] (u6) at (25:2) [label = 90:{$u_{j+1}$}] {};
\node[vertex] (w7) at (45:4) [label = 80:{$s$}] {};
\node[vertex] (w8) at (25:4) [label = 0:{$t$}] {};
\node[vertex] (w9) at (5:4) [label = 0:{$z$}] {};

\draw (v)--(u1) (v)--(u5);
\draw (w1) .. controls (130:6) and (90:6) .. (w5);
\draw (u1)--(w1) (u5)--(w5);

\draw[thickedge] (w1) .. controls (130:6) and (90:6) .. (w5);
\draw[thickedge] (v)--(u1) (v)--(u5) (u1)--(w1) (u5)--(w5);

\draw (w1)--(x)--(y)--(u5);

\draw (u5)--(w7)--(w8)--(u6)--(v);
\draw[dashed] (w8)--(w9);

\node at (110:4.4) {$Q$};
\node at (270: 1) {(1)};

\begin{scope}[shift={(10.3,0)}]
	\node[vertex] (v) at (0,0)  [label = 270:{$v$}] {};
	
	\node[vertex] (u1) at (155:2) [label = 270:{$u_0$}] {};
	\node[vertex] (u5) at (65:2) [label = 180:{$u_j$}] {};
	
	\node[vertex] (w1) at (155:4) [label = 180:{$w_0$}] {};
	\node[vertex, fill=black] (w5) at (65:4) [label = 180:{$w_j$}] {};
	
	\node[vertex] (x) at (125:3) [label = 110:{$x$}] {};
	\node[vertex] (y) at (95:3) [label = 110:{$y$}] {};
	
	\node[vertex] (u6) at (25:2) [label = 90:{$u_{j+1}$}] {};
	\node[vertex] (w7) at (45:4) [label = 80:{$s$}] {};
	\node[vertex] (w8) at (25:4) [label = 0:{$t$}] {};
	\node[vertex] (w9) at (5:4) [label = 0:{$z'$}] {};
	
	\draw (v)--(u1) (v)--(u5);
	\draw (w1) .. controls (130:6) and (90:6) .. (w5);
	\draw (u1)--(w1) (u5)--(w5);
	
	\draw[thickedge] (w1) .. controls (130:6) and (90:6) .. (w5);
	\draw[thickedge] (v)--(u1) (v)--(u5) (u1)--(w1) (u5)--(w5);
	
	\draw (w1)--(x)--(y)--(u5);
	
	\draw (u5)--(w7)--(w8)--(u6)--(v);
	\draw[dashed] (u6)--(w9);
	
	\node at (110:4.4) {$Q$};
	\node at (270: 1) {(2)};
\end{scope}

\end{tikzpicture}
\caption{In Case 2.3, either $d(u_{j+1}) = 2$, and $t$ has some neighbour $z$ incident with $f_{j+1}$ (1), or $d(u_{j+1}) > 2$, and $u_{j+1}$ has some neighbour $z'$ other than $t$ that is incident with $f_{j+1}$.}
\label{fig_1_chord_23_setup}
\end{figure}
	
	\textit{Case 2.3.1:} The vertices $t$ and $w_0$ are adjacent.\\
	In this case, $t, w_0, x$ a geodesic. 
	The graph $G$ contains one of the vertices $z$ or $z'$ described above, and has girth 5, and so either $d(z, y) > 3$ or $d(z', y) > 3$, respectively.
	
	\textit{Case 2.3.2:} There is a vertex $w_{\Delta - 1}$ that is adjacent to $t$, $w_0$ and $u_{\Delta - 1}$.\\
	The path $t, w_{\Delta - 1}, w_0, x$ is a $t-x$ geodesic (see Figure \ref{fig_1_chord_232_233}).
	One of $z$ or $z'$ is present in $G$, so by the planarity and girth constraints of $G$, either $d(z, y) > 3$ or $d(z', y) > 3$. 
	
\begin{figure}[h]
\centering
\begin{tikzpicture}
[scale =0.7, inner sep=0.8mm, 
vertex/.style={circle,thick,draw},
dvertex/.style={rectangle,thick,draw, inner sep=1.3mm}, 
thickedge/.style={line width=1.5pt}] 

\node[vertex] (v) at (0,0)  [label = 270:{$v$}] {};

\node[vertex] (u1) at (150:2) [label = 80:{$u_0$}] {};
\node[vertex] (u5) at (60:2) [label = 180:{$u_j$}] {};

\node[vertex] (w1) at (150:4) [label = 50:{$w_0$}] {};
\node[vertex, fill=black] (w5) at (60:4) [label = 180:{$w_j$}] {};

\node[vertex] (x) at (120:3) [label = 110:{$x$}] {};
\node[vertex] (y) at (90:3) [label = 110:{$y$}] {};

\node[vertex] (u6) at (20:2) [label = 90:{$u_{j+1}$}] {};
\node[vertex] (w7) at (40:4) [label = 120:{$s$}] {};
\node[vertex] (w8) at (20:4) [label = 0:{$t$}] {};
\node[vertex] (w9) at (0:4) [label = 0:{$z$ or $z'$}] {};

\node[vertex] (u0) at (180:2) [label = 90:{$u_{\Delta - 1}$}] {};
\node[vertex] (w0) at (180:4) [label = 180:{$w_{\Delta - 1}$}] {};

\draw (v)--(u1) (v)--(u5);
\draw (w1) .. controls (125:6) and (85:6) .. (w5);
\draw (u1)--(w1) (u5)--(w5);

\draw[thickedge] (w1) .. controls (125:6) and (85:6) .. (w5);
\draw[thickedge] (v)--(u1) (v)--(u5) (u1)--(w1) (u5)--(w5);

\draw (w1)--(x)--(y)--(u5);

\draw (u5)--(w7)--(w8)--(u6)--(v);
\draw[dashed] (w8)--(w9) (u6)--(w9);

\draw (v)--(u0)--(w0)--(w1);

\draw (w0) .. controls (130:8) and (70:8) .. (w8);

\node at (105:4.4) {$Q$};
\node at (270: 1) {Case 2.3.2};

\begin{scope}[shift={(10.3,0)}]
	\node[vertex] (v) at (0,0)  [label = 270:{$v$}] {};
	
	\node[vertex] (u1) at (165:2) [label = 270:{$u_0$}] {};
	\node[vertex] (u5) at (75:2) [label = 180:{$u_j$}] {};
	
	\node[vertex] (w1) at (165:4) [label = 180:{$w_0$}] {};
	\node[vertex, fill=black] (w5) at (75:4) [label = 180:{$w_j$}] {};
	
	\node[vertex] (x) at (135:3) [label = 110:{$x$}] {};
	\node[vertex] (y) at (105:3) [label = 110:{$y$}] {};
	
	\node[vertex] (u6) at (35:2) [label = 90:{$u_{j+1}$}] {};
	\node[vertex] (w7) at (55:4) [label = 80:{$s$}] {};
	\node[vertex] (w8) at (35:4) [label = 180:{$t$}] {};
	\node[vertex] (w9) at (20:4) [label = 0:{$z'$}] {};
	
	\node[vertex] (u7) at (0:2) [label = 90:{$u_{j+2}$}] {};
	\node[vertex] (wa) at (0:4) [label = 0:{$a$}] {};
	\node[vertex] (wb) at (100:6) [label = 270:{$b$}] {};
	
	\draw (v)--(u1) (v)--(u5);
	\draw (w1) .. controls (140:6) and (100:6) .. (w5);
	\draw (u1)--(w1) (u5)--(w5);
	
	\draw[thickedge] (w1) .. controls (140:6) and (100:6) .. (w5);
	\draw[thickedge] (v)--(u1) (v)--(u5) (u1)--(w1) (u5)--(w5);
	
	\draw (w1)--(x)--(y)--(u5);
	
	\draw (u5)--(w7)--(w8)--(u6)--(v);
	\draw (u6)--(w9)--(wa)--(u7)--(v);
	
	\draw (w1) .. controls (145:6) and (120:6) .. (wb) .. controls (80:6) and (55:6) .. (w8);
	
	\node at (120:4.4) {$Q$};
	\node at (270: 1) {Case 2.3.3};
\end{scope}

\end{tikzpicture}
\caption{The left figure illustrates Case 2.3.2 in which $t$ and $w_{\Delta - 1}$ are adjacent. 
The right figure shows Case 2.3.3, under the assumption that $G$ contains the vertex $z'$ that is not adjacent to $t$.}
\label{fig_1_chord_232_233}
\end{figure}
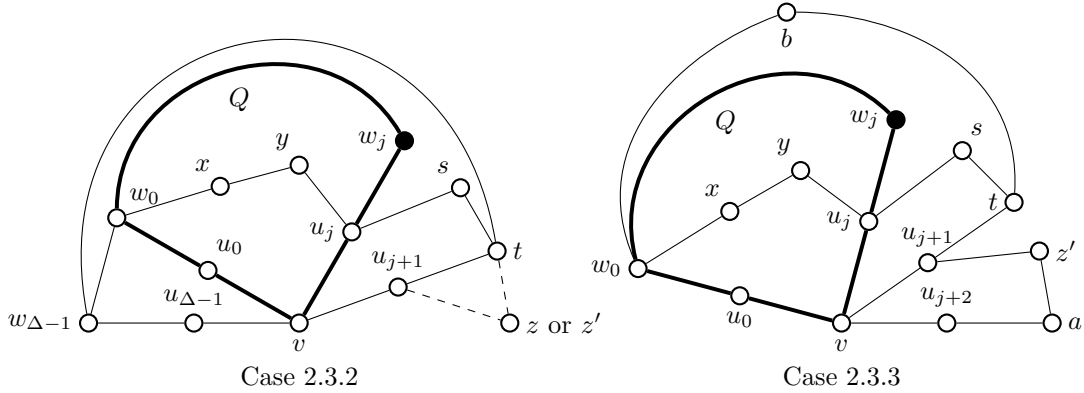
	
	\textit{Case 2.3.3:} There is some vertex $b$, that is not adjacent to $u_{\Delta - 1}$, but that is adjacent to both $t$ and $w_0$.
	Thus the $t-x$ geodesic is $t, b, w_0, x$.
	If $G$ contains $z$, which is adjacent to $t$, then $z$ is not adjacent to $w_0$ as this induces a 4-cycle on $z$, $w_0$, $b$ and $t$. 
	Thus, if $G$ contains $z$, we have the contradiction $d(z, y) > 3$.
	Therefore $z'$ is a vertex of $G$.
	Let $a$ be the vertex of $N_2(v) \cap N(z')$ that is incident with $f_{j+1}$ (see Figure \ref{fig_1_chord_232_233}, Case 2.3.3).
	The only possible $y-z'$ geodesic is $z', w_0, x, y$, so $z'$ and $w_0$ are adjacent.
	As $G$ is triangle-free, $a$ and $w_0$ are not adjacent.
	Therefore $d(a, y) > 3$, concluding Case 2.
	
	\textit{Case 3:} The vertices $w_0$ and $w_j$ each have a neighbor in $\Int(C_Q)$.\\
	Let $x$ and $y$ be vertices in $\Int(C_Q)$ that are adjacent to $w_0$ and $w_j$, respectively.
	Since $G$ has girth 5, $x$ is not adjacent to any vertex of $C_Q$ apart from $w_0$, and $y$ is not adjacent to any vertex of $C_Q$ besides $w_j$.
	There are two subcases to consider.
	
	\textit{Case 3.1:} At least one of vertices $u_0$ and $u_j$ have a neighbor in $\ext(C_Q)$.\\
	Assume without loss of generality that $u_0$ is adjacent to some vertex in $\ext(C_Q)$.
	Let $s$ be the neighbor of $u_0$ in $\ext(C_Q)$ that is incident with the face $f_{\Delta - 1}$, and let $t$ be the other neighbor of $s$ that is also incident with $f_{\Delta - 1}$.
	Note that $s$ is not adjacent to $w_j$, as this induces a 4-cycle on the vertices $s$, $w_j$, $w_0$ and $u_0$. 
	There are two ways that $G$ may contain an $s-y$ path of length at most 3, and we consider both as subcases.
	
\begin{figure}[h]
\centering
\begin{tikzpicture}
[scale =0.7, inner sep=0.8mm, 
vertex/.style={circle,thick,draw},
dvertex/.style={rectangle,thick,draw, inner sep=1.3mm}, 
thickedge/.style={line width=1.5pt}] 

\node[vertex] (v) at (0,0)  [label = 270:{$v$}] {};

\node[vertex] (u0) at (110:2) [label = 0:{$u_0$}] {};
\node[vertex] (uj) at (20:2) [label = 270:{$u_j$}] {};
\node[vertex] (ud) at (160:2) [label = 90:{$u_{\Delta - 1}$}] {};
\node[vertex] (x) at (85:3) [label = 90:{$x$}] {};
\node[vertex] (y) at (45:3) [label = 90:{$y$}] {};
\node[vertex] (s) at (135:4) [label = 180:{$s$}] {};
\node[vertex] (t) at (160:4) [label = 180:{$t$}] {};
\node[vertex] (a) at (80:6) [label = 270:{$a$}] {};

\node[vertex] (w0) at (110:4) [label = 240:{$w_0$}] {};
\node[vertex] (wj) at (20:4) [label = 340:{$w_j$}] {};

\draw (v)--(u0) (v)--(uj);
\draw (w0) .. controls (85:6) and (45:6) .. (wj);
\draw (u0)--(w0) (uj)--(wj);
\draw (w0)--(x) (wj)--(y);
\draw (u0)--(s)--(t)--(ud)--(v);

\draw[thickedge] (w0) .. controls (85:6) and (45:6) .. (wj);
\draw[thickedge] (v)--(u0) (v)--(uj) (u0)--(w0) (uj)--(wj);

\draw (s) .. controls (120:6) and (100:6) .. (a) .. controls (60:6) and (40:6) .. (wj);

\node at (65:4.4) {$Q$};
\node at (270: 1) {Case 3.1.1};

\begin{scope}[shift={(10.3,0)}]
	\node[vertex] (v) at (0,0)  [label = 270:{$v$}] {};
	
	\node[vertex] (u0) at (110:2) [label = 0:{$u_0$}] {};
	\node[vertex] (uj) at (20:2) [label = 270:{$u_j$}] {};
	\node[vertex] (ud) at (160:2) [label = 90:{$u_{\Delta - 1}$}] {};
	\node[vertex] (x) at (85:3) [label = 90:{$x$}] {};
	\node[vertex] (y) at (45:3) [label = 90:{$y$}] {};
	\node[vertex] (s) at (135:4) [label = 180:{$s$}] {};
	\node[vertex] (t) at (160:4) [label = 180:{$t$}] {};
	\node[vertex] (z) at (180:4) [label = 180:{$z$}] {};
	
	\node[vertex] (w0) at (110:4) [label = 240:{$w_0$}] {};
	\node[vertex] (wj) at (20:4) [label = 340:{$w_j$}] {};
	
	\draw (v)--(u0) (v)--(uj);
	\draw (w0) .. controls (85:6) and (45:6) .. (wj);
	\draw (u0)--(w0) (uj)--(wj);
	\draw (w0)--(x) (wj)--(y);
	\draw (u0)--(s)--(t)--(ud)--(v);
	
	\draw[thickedge] (w0) .. controls (85:6) and (45:6) .. (wj);
	\draw[thickedge] (v)--(u0) (v)--(uj) (u0)--(w0) (uj)--(wj);
	
	\draw[dashed] (ud)--(z)--(t);
	\draw (t) .. controls (125:8) and (50:9) .. (wj);
	
	\node at (65:4.4) {$Q$};
	\node at (270: 1) {Case 3.1.2};
\end{scope}

\end{tikzpicture}
\caption{In Case 3.1.1, there is an $s-y$ path $s, a, w_j, y$ containing some vertex $a$ in $N_2(v)\cup N_3(v) - \{t\}$.
In Case 3.1.2, the vertex $t$ is adjacent to $w_j$, and $s, t, w_j, y$ is an $s-y$ path of length 3.}
\label{fig_1_chord_31_311}
\end{figure}
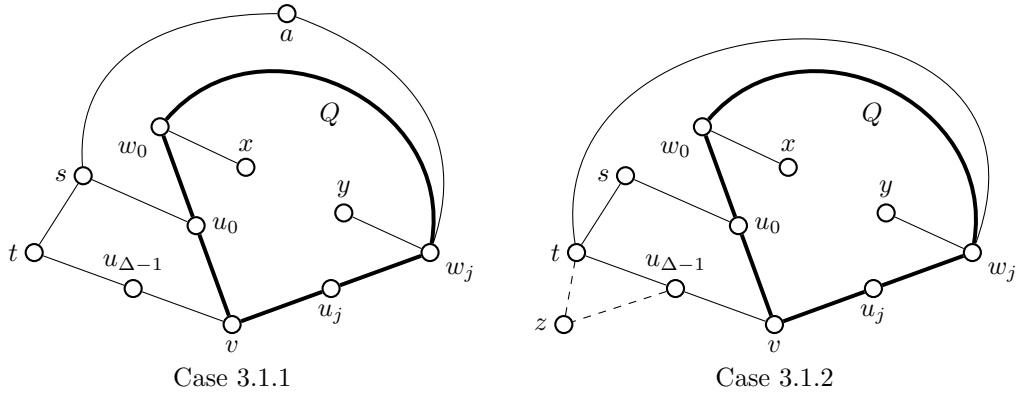
	
	\textit{Case 3.1.1:} There is some vertex $a\neq t$ that is adjacent to both $s$ and $w_j$.\\
	The path $s, a, w_j, y$ is the $s-y$ geodesic (see Figure \ref{fig_1_chord_31_311}). 
	Since $G$ has girth 5, $d(t,x) > 3$, a contradiction.
	
	\textit{Case 3.1.2:} The vertices $t$ and $w_j$ are adjacent.\\
	The $s-y$ geodesic is $s, t, w_j, y$ (see Figure \ref{fig_1_chord_31_311}).
	We consider the face $f_{\Delta - 2}$. 
	Either the vertex $t$ is incident with this face, and there is a vertex $z$ in $N_1(t)\cap N_1(u_{\Delta - 2})$, or $t$ is not incident with this face, and there is a vertex $z'$ in $N_1(u_{\Delta - 1}) \cap N_2(v)$.
	In both cases we derive a contradiction, as either $d(z,x) > 3$ or $d(z', x) > 3$.
	
	Since Case 3.1 yields a contradiction, we may assume that neither $u_0$ nor $u_j$ has a neighbor in $\ext(C_Q)$.
	Since $u_0$ has no neighbor in $\ext(C_Q)$, $w_0$ is incident with the face $f_{\Delta - 1}$.
	Similarly, $w_j$ is incident with $f_j$.
	Let $s$ be the vertex of $N_1(w_0) - \{u_0\}$ that is incident with $f_{\Delta - 1}$, and let $w_{j+1}$ be the vertex of $N_1(w_j) - \{u_j\}$ that is incident with $f_j$.
	
	\textit{Case 3.2:} The vertex $u_{\Delta - 1}$ has degree at least 3.\\
	In this case, $s$ is only incident with the face $f_{\Delta - 1}$, and not the face $f_{\Delta - 2}$.
	Let $t$ denote the neighbor of $u_{\Delta - 1}$ that is incident with $f_{\Delta - 2}$, and we let $z$ be the vertex of $N(t)-u_{\Delta - 1}$ that is incident with $f_{\Delta - 2}$ (See Figure \ref{fig_1_chord_32_321}).
	
\begin{figure}[h]
\centering
\begin{tikzpicture}
[scale =0.7, inner sep=0.8mm, 
vertex/.style={circle,thick,draw},
dvertex/.style={rectangle,thick,draw, inner sep=1.3mm}, 
thickedge/.style={line width=1.5pt}] 

\node[vertex] (v) at (0,0)  [label = 270:{$v$}] {};

\node[vertex] (u0) at (125:2) [label = 20:{$u_0$}] {};
\node[vertex] (uj) at (55:2) [label = 95:{$u_j$}] {};

\node[vertex] (w0) at (125:4) [label = 110:{$w_0$}] {};
\node[vertex] (wj) at (55:4) [label = 10:{$w_j$}] {};
\node[vertex] (x) at (107:3.4) [label = 90:{$x$}] {};
\node[vertex] (y) at (73:3.4) [label = 90:{$y$}] {};

\node[vertex] (ud1) at (155:2) [label = 270:{$u_{\Delta - 1}$}] {};
\node[vertex] (ud2) at (180:2) [label = 270:{$u_{\Delta - 2}$}] {};
\node[vertex] (s) at (145:4) [label = 140:{$s$}] {};
\node[vertex] (t) at (165:4) [label = 180:{$t$}] {};
\node[vertex] (z) at (180:4) [label = 180:{$z$}] {};

\node[vertex] (wj1) at (30:4) [label = 10:{$w_{j+1}$}] {};
\node[vertex] (uj1) at (30:2) [label = 350:{$u_{j+1}$}] {};

\draw (v)--(u0) (v)--(uj);
\draw (u0)--(w0) (uj)--(wj);

\draw[thickedge] (w0) .. controls (105:5.5) and (75:5.5) .. (wj);
\draw[thickedge] (v)--(u0) (v)--(uj) (u0)--(w0) (uj)--(wj);

\draw (v)--(ud1)--(s)--(w0);
\draw (ud1)--(t)--(z)--(ud2)--(v);
\draw (v)--(uj1)--(wj1)--(wj);

\draw (w0)--(x) (wj)--(y);

\node at (90:4.4) {$Q$};
\node at (270: 1) {Case 3.2};

\begin{scope}[shift={(10.3,0)}]
	\node[vertex] (v) at (0,0)  [label = 270:{$v$}] {};
	
	\node[vertex] (u0) at (125:2) [label = 20:{$u_0$}] {};
	\node[vertex] (uj) at (55:2) [label = 95:{$u_j$}] {};
	
	\node[vertex] (w0) at (125:4) [label = 110:{$w_0$}] {};
	\node[vertex] (wj) at (55:4) [label = 10:{$w_j$}] {};
	\node[vertex] (x) at (107:3.4) [label = 90:{$x$}] {};
	\node[vertex] (y) at (73:3.4) [label = 90:{$y$}] {};
	
	\node[vertex] (ud1) at (155:2) [label = 270:{$u_{\Delta - 1}$}] {};
	\node[vertex] (ud2) at (180:2) [label = 270:{$u_{\Delta - 2}$}] {};
	\node[vertex] (s) at (145:4) [label = 140:{$s$}] {};
	\node[vertex] (t) at (165:4) [label = 180:{$t$}] {};
	\node[vertex] (z) at (180:4) [label = 180:{$z$}] {};
	
	\node[vertex] (wj1) at (30:4) [label = 10:{$w_{j+1}$}] {};
	\node[vertex] (uj1) at (30:2) [label = 350:{$u_{j+1}$}] {};
	
	\draw (v)--(u0) (v)--(uj);
	\draw (u0)--(w0) (uj)--(wj);
	
	\draw[thickedge] (w0) .. controls (105:5.5) and (75:5.5) .. (wj);
	\draw[thickedge] (v)--(u0) (v)--(uj) (u0)--(w0) (uj)--(wj);
	
	\draw (v)--(ud1)--(s)--(w0);
	\draw (ud1)--(t)--(z)--(ud2)--(v);
	\draw (v)--(uj1)--(wj1)--(wj);
	
	\draw (t) .. controls (140:7) and (80:7) .. (wj);
	
	\draw (w0)--(x) (wj)--(y);
	
	\node at (90:4.4) {$Q$};
	\node at (270: 1) {Case 3.2.1};
\end{scope}

\end{tikzpicture}
\caption{If $d(u_{\Delta-1}) > 2$, then distinct neighbours $s$ and $t$ of $u_{\Delta -1}$ are incident with the faces $f_{\Delta - 1}$ and $f_{\Delta - 2}$, respectively (Case 3.2).
In Case 3.2.1, we consider the possibility that there is a $t-y$ path of the form $t, w_j, y$.}
\label{fig_1_chord_32_321}
\end{figure}
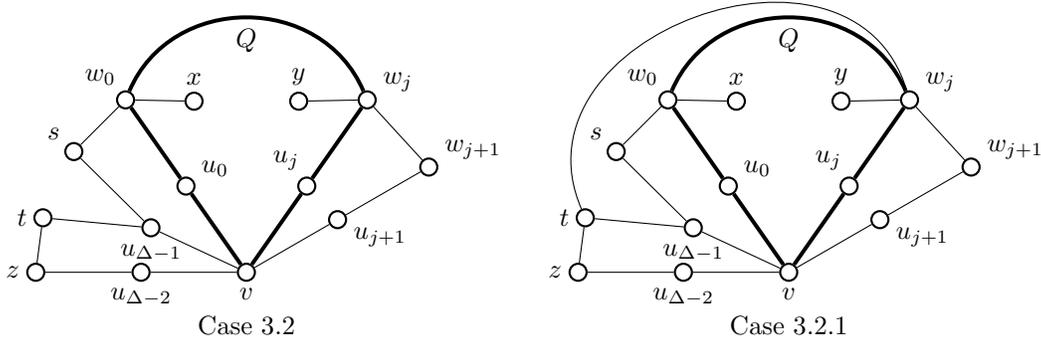
	
	Considering the girth and planarity of $G$, there are only three possibilities for a $y-t$ geodesic.
	
	\textit{Case 3.2.1:} The vertices $t$ and $w_j$ are adjacent.\\
	The $t-y$ geodesic is $t, w_j, y$ (See Figure \ref{fig_1_chord_32_321}).
	Since $G$ has girth 5, there is no $z-x$ path of length 3 or less, a contradiction.
	
	\textit{Case 3.2.2:} There is some vertex $a \neq z$ that is adjacent to both $t$ and $w_j$.\\
	It is possible that $a = w_{j+1}$, but this does not affect the argument.
	The $t-y$ geodesic is $t, a, w_j, y$. 
	Similar to Case 3.2.1, $d(z,x) > 3$.
	
	\textit{Case 3.2.3:} The vertices $z$ and $w_j$ are adjacent.\\
	The $t-y$ geodesic is $t, z, w_j, y$.
	Consider the vertex $u_{j+1}$.
	If it has degree 2, then there is a vertex $b \neq u_{j+1}$ that adjacent to $w_{j+1}$ and incident with the face $f_{j+1}$.
	If $d(u_{j+1}) \geq 3$, then there exists a vertex $b' \neq w_{j+1}$ that is adjacent to $u_{j+1}$ and incident with $f_{j+1}$.
	In either case, the vertex $b$ or $b'$ is not adjacent to $w_j$ since $G$ has girth 5.
	Whether $G$ contains $b$ or $b'$, we obtain a contradiction, since either $d(b, x) > 3$ or $d(b', x) > 3$.
	
	\textit{Case 3.3:} The vertex $u_{\Delta - 1}$ has degree 2.\\
	The vertex $s$ is the only neighbor of $u_{\Delta - 1}$ besides $v$. Denote by $t$ the vertex of $N_1(s) - \{u_{\Delta - 1} \}$ that is incident with the face $f_{\Delta - 2}$.
	Since $G$ is a plane graph of girth 5, $t$ is not adjacent to either $w_0$ or $w_j$.
	There are two subcases to consider: one for each way that $G$ can exhibit a $t-y$ geodesic.
	
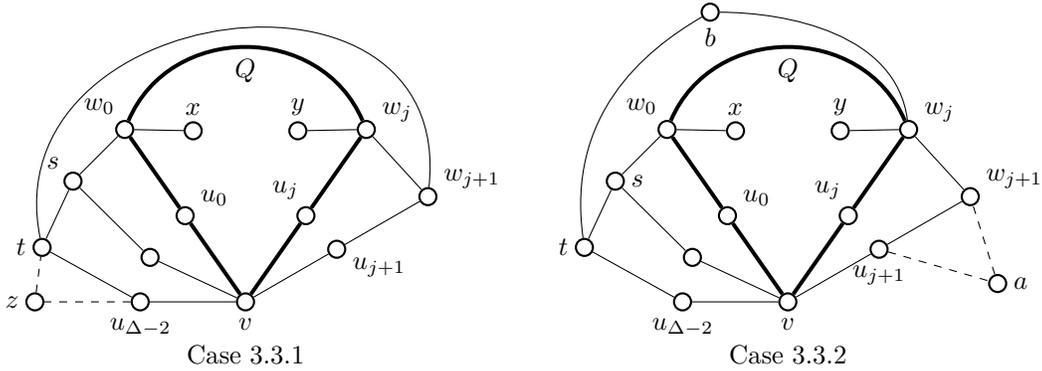
\begin{figure}[h]
\centering
\begin{tikzpicture}
[scale =0.7, inner sep=0.8mm, 
vertex/.style={circle,thick,draw},
dvertex/.style={rectangle,thick,draw, inner sep=1.3mm}, 
thickedge/.style={line width=1.5pt}] 

\node[vertex] (v) at (0,0)  [label = 270:{$v$}] {};

\node[vertex] (u0) at (125:2) [label = 20:{$u_0$}] {};
\node[vertex] (uj) at (55:2) [label = 95:{$u_j$}] {};

\node[vertex] (w0) at (125:4) [label = 110:{$w_0$}] {};
\node[vertex] (wj) at (55:4) [label = 10:{$w_j$}] {};
\node[vertex] (x) at (107:3.4) [label = 90:{$x$}] {};
\node[vertex] (y) at (73:3.4) [label = 90:{$y$}] {};

\node[vertex] (ud1) at (155:2) {}; 
\node[vertex] (ud2) at (180:2) [label = 270:{$u_{\Delta - 2}$}] {};
\node[vertex] (s) at (145:4) [label = 140:{$s$}] {};
\node[vertex] (t) at (165:4) [label = 180:{$t$}] {};
\node[vertex] (z) at (180:4) [label = 180:{$z$}] {};

\node[vertex] (wj1) at (30:4) [label = 10:{$w_{j+1}$}] {};
\node[vertex] (uj1) at (30:2) [label = 350:{$u_{j+1}$}] {};

\draw (v)--(u0) (v)--(uj);
\draw (u0)--(w0) (uj)--(wj);

\draw[thickedge] (w0) .. controls (105:5.5) and (75:5.5) .. (wj);
\draw[thickedge] (v)--(u0) (v)--(uj) (u0)--(w0) (uj)--(wj);

\draw (v)--(ud1)--(s)--(w0);
\draw (v)--(uj1)--(wj1)--(wj);
\draw (s)--(t)--(ud2)--(v);

\draw (t) .. controls (130:7.5) and (60:8) .. (wj1);
\draw[dashed] (ud2)--(z)--(t);

\draw (w0)--(x) (wj)--(y);

\node at (90:4.4) {$Q$};
\node at (270: 1) {Case 3.3.1};

\begin{scope}[shift={(10.3,0)}]
	\node[vertex] (v) at (0,0)  [label = 270:{$v$}] {};
	
	\node[vertex] (u0) at (125:2) [label = 20:{$u_0$}] {};
	\node[vertex] (uj) at (55:2) [label = 95:{$u_j$}] {};
	
	\node[vertex] (w0) at (125:4) [label = 110:{$w_0$}] {};
	\node[vertex] (wj) at (55:4) [label = 10:{$w_j$}] {};
	\node[vertex] (x) at (107:3.4) [label = 90:{$x$}] {};
	\node[vertex] (y) at (73:3.4) [label = 90:{$y$}] {};
	
	\node[vertex] (ud1) at (155:2) {}; 
	\node[vertex] (ud2) at (180:2) [label = 270:{$u_{\Delta - 2}$}] {};
	\node[vertex] (s) at (145:4) [label = 0:{$s$}] {};
	\node[vertex] (t) at (165:4) [label = 180:{$t$}] {};
	\node[vertex] (a) at (5:4) [label = 0:{$a$}] {};
	
	\node[vertex] (wj1) at (30:4) [label = 10:{$w_{j+1}$}] {};
	\node[vertex] (uj1) at (30:2) [label = 270:{$u_{j+1}$}] {};
	\node[vertex] (b) at (105:5.7) [label = 270:{$b$}] {};
	
	\draw (v)--(u0) (v)--(uj);
	\draw (u0)--(w0) (uj)--(wj);
	
	\draw[thickedge] (w0) .. controls (105:5.5) and (75:5.5) .. (wj);
	\draw[thickedge] (v)--(u0) (v)--(uj) (u0)--(w0) (uj)--(wj);
	
	\draw (v)--(ud1)--(s)--(w0);
	\draw (v)--(uj1)--(wj1)--(wj);
	\draw (s)--(t)--(ud2)--(v);
	
	\draw (t) .. controls (140:5.5) and (120:5.5) .. (b) .. controls (90:5.5) and (68:5.5) .. (wj);
	
	\draw (w0)--(x) (wj)--(y);
	\draw[dashed] (uj1)--(a)--(wj1);
	
	\node at (90:4.4) {$Q$};
	\node at (270: 1) {Case 3.3.2};
\end{scope}

\end{tikzpicture}
\caption{In Case 3.3, we assume that $u_{\Delta - 1}$ has only two neighbours. 
In subcase 3.3.1, we consider what happens when the vertices $t$ and $w_{j+1}$ are adjacent.}
\label{fig_1_chord_33_331}
\end{figure}
	
	\textit{Case 3.3.1:} The vertices $t$ and $w_{j+1}$ are adjacent.\\
	The $t-y$ path is $t, w_{j+1}, w_j, y$.
	Either $t$ or $u_{\Delta - 2}$ has some neighbor $z$ in $N_2(v)$ that has not yet been mentioned. 
	We obtain a contradiction as $d(z,x) > 3$ (see Figure \ref{fig_1_chord_33_331}).
	
	\textit{Case 3.3.2:} There is some vertex $b\neq w_{j+1}$ that is adjacent to both $t$ and $w_j$.\\
	We have the $y-t$ geodesic $t, b, w_j, y$.
	Either $d(u_{j+1}) = 2$, and so $w_{j+1}$ has a neighbor in $N_2(v)$ incident with $f_{j+1}$, or $d(u_{j+1}) \geq 3$ and $u_{j+1}$ has a neighbor in $N_2(v)-\{w_{j+1}\}$ incident with $f_{j+1}$.
	In either case, call this neighbor $a$, and note that $d(a,x) > 3$.
	
	In all cases, we derive a contradiction, completing the proof.
\end{proof}

\begin{thm}
	Let $G$ be a pentagulation of diameter 3, girth 5 and maximum degree $\Delta$, and let $v$ be a vertex of $G$ with maximum degree. 
	If $\Delta \geq 8$, then $G$ does not have any 2-chords with respect to $v$.
	\label{thm_len_2_chords}
\end{thm}

\begin{proof}
	Assume for the sake of contradiction that there does exist some 2-chord with respect to $v$.
	Repeat the argument used at the start of the proof of Theorem \ref{thm_len_1_chords}, and adopt the same labeling convention for the vertices of $N(v)$ and $N_2(v)$, and for the faces incident with the vertex $v$. 
	There is a minimal 2-chord $Q: w_0, a, w_j$, where $w_0$ and $w_j$ are vertices of $N_2(v)$, the vertex $a$ lies in $N_3(v)$, and the cycle $C_Q$ under $Q$ dominates its interior.
	The vertices $u_0$ and $u_j$ are the unique vertices of $N(v) \cap N(w_0)$ and $N(v)\cap N(w_j)$, respectively.
	
	\textit{Claim 1:} The index $j$ satisfies $j<4$.\\
	Assume to the contrary that $j\geq 4$, and observe per Lemma \ref{lem_nhood_rules} that $u_2$ has some neighbor $w_2$ in $N_2(v)$.
	By Theorem \ref{thm_len_1_chords}, the vertex $w_2$ is adjacent to neither $w_0$ nor $w_j$.
	Since $G$ has girth 5, $w_2$ is not adjacent to either $u_0$ or $u_j$.
	Since $C_Q$ dominates its interior, $w_2$ is adjacent to $a$.
	Thus $w_2, a, w_0$ is a 2-chord, which contradicts the minimality of $Q$ and proves Claim 1.
	
	\textit{Claim 2:} It is not possible that both $w_0$ and $w_j$ have neighbors in $\Int(C_Q)$.\\
	Assume to the contrary that $w_0$ has some neighbor $x$ in $\Int(C_Q)$ and $w_j$ has a neighbor $y$ in $\Int(C_Q)$.
	We have $x \neq y$: were $x=y$, there would be a 4-cycle on the vertices $w_0, x, w_j, a$.
	Since $G$ has girth 5, $x$ is not adjacent to $a$ or $w_j$, and $y$ is not adjacent to $a$ or $w_0$.
	The face $f_{j+2}$ is bounded by the 5-cycle $v, u_{j+2}, s, t, u_{j+3}$, where $s$ and $t$ are vertices of $N_2(v)$.
	Note that $d(t,y) \leq 3$.
	Per Theorem \ref{thm_len_1_chords}, the vertex $t$ is not adjacent to any vertices of $N_2(v)$ apart from $s$, and possibly one other vertex that is incident with the face $f_{j+3}$.
	Hence $G$ can only exhibit a $t-y$ path in one of two ways (see Figure \ref{fig_2_chord_twovert}):\\
	(1) the vertices $a$ and $t$ are adjacent, and the geodesic is $t, a, w_j, y$, or\\
	(2) there is some vertex $b$ in $N_3(v)$ that is adjacent to both $t$ and $w_j$, yielding a geodesic $t, b, w_j, y$.
	
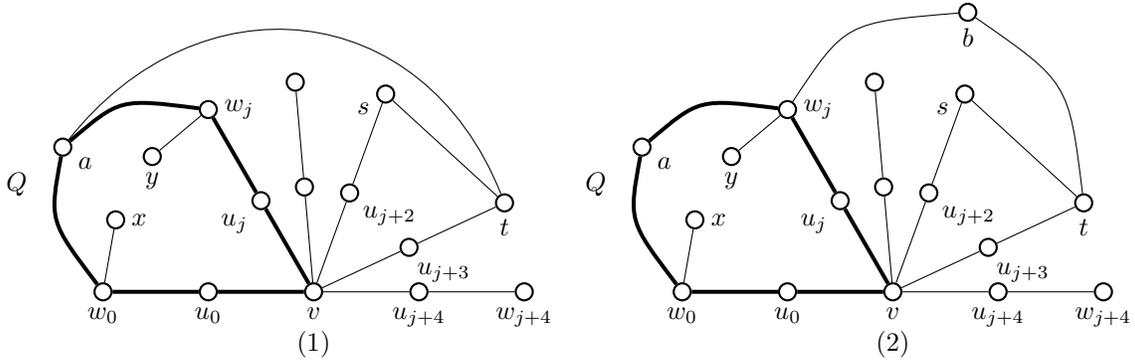
\begin{figure}[h]
\centering
\begin{tikzpicture}
[scale =0.7, inner sep=0.8mm, 
vertex/.style={circle,thick,draw},
dvertex/.style={rectangle,thick,draw, inner sep=1.3mm}, 
thickedge/.style={line width=1.5pt}] 

\node[vertex] (v) at (0,0)  [label = 270:{$v$}] {};

\node[vertex] (u0) at (180:2) [label = 270:{$u_0$}] {};
\node[vertex] (uj) at (120:2) [label = 225:{$u_j$}] {};
\node[vertex] (uj1) at (95:2) {};
\node[vertex] (uj2) at (70:2) [label = 315:{$u_{j+2}$}] {};
\node[vertex] (uj3) at (25:2) [label = 280:{$u_{j+3}$}] {};
\node[vertex] (uj4) at (0:2) [label = 270:{$u_{j+4}$}] {};

\node[vertex] (w0) at (180:4) [label = 270:{$w_0$}] {};
\node[vertex] (wj) at (120:4) [label = 0:{$w_j$}] {};
\node[vertex] (wj1) at (95:4) {};
\node[vertex] (wj2) at (70:4) [label = 190:{$s$}] {};
\node[vertex] (wj3) at (25:4) [label = 270:{$t$}] {};
\node[vertex] (wj4) at (0:4) [label = 270:{$w_{j+4}$}] {};

\node[vertex] (x) at (160:4) [label = 0:{$x$}] {};
\node[vertex] (y) at (140:4) [label = 270:{$y$}] {};

\node[vertex] (a) at (150:5.5) [label = 340:{$a$}] {};

\node at (160:6) {$Q$};

\draw[thickedge] (w0)--(u0)--(v)--(uj)--(wj);
\draw[thickedge] (w0) .. controls (165:5.2) .. (a) .. controls (135:5.2) .. (wj);

\draw (w0)--(x) (wj)--(y);
\draw (v)--(uj1) (v)--(uj2) (v)--(uj3) (v)--(uj4);
\draw (uj1)--(wj1) (uj2)--(wj2) (uj3)--(wj3) (uj4)--(wj4);
\draw (wj2)--(wj3);

\draw (a) .. controls (110:6.4) and (70:6) .. (wj3);

\node at (0,-1) {(1)};

\begin{scope}[shift={(11,0)}]
	\node[vertex] (v) at (0,0)  [label = 270:{$v$}] {};
	
	\node[vertex] (u0) at (180:2) [label = 270:{$u_0$}] {};
	\node[vertex] (uj) at (120:2) [label = 225:{$u_j$}] {};
	\node[vertex] (uj1) at (95:2) {};
	\node[vertex] (uj2) at (70:2) [label = 315:{$u_{j+2}$}] {};
	\node[vertex] (uj3) at (25:2) [label = 280:{$u_{j+3}$}] {};
	\node[vertex] (uj4) at (0:2) [label = 270:{$u_{j+4}$}] {};
	
	\node[vertex] (w0) at (180:4) [label = 270:{$w_0$}] {};
	\node[vertex] (wj) at (120:4) [label = 0:{$w_j$}] {};
	\node[vertex] (wj1) at (95:4) {};
	\node[vertex] (wj2) at (70:4) [label = 190:{$s$}] {};
	\node[vertex] (wj3) at (25:4) [label = 270:{$t$}] {};
	\node[vertex] (wj4) at (0:4) [label = 270:{$w_{j+4}$}] {};
	
	\node[vertex] (x) at (160:4) [label = 0:{$x$}] {};
	\node[vertex] (y) at (140:4) [label = 270:{$y$}] {};
	
	\node[vertex] (a) at (150:5.5) [label = 340:{$a$}] {};
	\node[vertex] (b) at (75:5.5) [label = 270:{$b$}] {};
	
	\node at (160:6) {$Q$};
	
	\draw[thickedge] (w0)--(u0)--(v)--(uj)--(wj);
	\draw[thickedge] (w0) .. controls (165:5.2) .. (a) .. controls (135:5.2) .. (wj);
	
	\draw (w0)--(x) (wj)--(y);
	\draw (v)--(uj1) (v)--(uj2) (v)--(uj3) (v)--(uj4);
	\draw (uj1)--(wj1) (uj2)--(wj2) (uj3)--(wj3) (uj4)--(wj4);
	\draw (wj2)--(wj3);
	
	\draw (wj) .. controls (100:5.2) .. (b) .. controls (50:5.2) .. (wj3);
	
	\node at (0,-1) {(2)};
\end{scope}

\end{tikzpicture}
\caption{In Claim 2, since $N_2(v)$ has no 1-chords but $G$ has diameter 3, either $t, a, w_j, y$ is a $t-y$ path (1), or $t, b, w_j, y$ is (2).}
\label{fig_2_chord_twovert}
\end{figure}
	
	Since $G$ has girth 5, and by Theorem \ref{thm_len_1_chords}, $G$ there are no 1-chords with respect to $v$. 
	Thus in both case (1) and (2), $d(s,x) > 3$, proving Claim 2.
	
	\textit{Claim 3:} $j < 3$.\\
	By Claim 1, we need only show that $j\neq 3$. 
	Suppose that $j=3$.
	By Lemma \ref{lem_nhood_rules}, $u_1$ and $u_2$ each have some neighbor, say $w_1$ and $w_2$ respectively, in $\Int(C_Q)$.
	Per Theorem \ref{thm_len_1_chords}, there are no 1-chords across $v$, so $w_1$ is not adjacent to $w_j$.
	By the minimality of $Q$, $w_1$ and $a$ are not adjacent, and since $G$ has girth 5, $w_1$ is not adjacent to $v$, $u_0$ or $u_j$.
	Similarly, $w_2$ is not adjacent to any of $w_0$, $a$, $v$, $u_0$ or $u_j$.
	Since $C_Q$ dominates its interior, $w_1$ is adjacent to $w_0$ and $w_2$ is adjacent to $w_3$.
	By Claim 2, this is not possible, proving Claim 3.
	
	There remain two cases to consider.
	
	\textit{Case 1:} Exactly one of $w_0$ or $w_j$ has a neighbor in $\Int(C_Q)$.\\
	Assume without loss of generality that $w_0$ has some neighbor, call it $x$, in $\Int(C_Q)$.
	The vertex $v$ has $d(v) \geq 8$, and by Claim 3, at most one neighbor of $v$ is contained in $\Int(C_Q)$.
	Thus $v$ has at least five neighbors in the exterior of $C_Q$.
	The face $f_{j+2}$ is bounded by a 5-cycle $v, u_{j+2}, s, t, u_{j+3}$, where $s$ and $t$ are vertices of $N_2(v)$. 
	Both $s$ and $t$ are within distance 3 of $x$.
	It is possible that $x$ is adjacent to $u_j$.
	However, $x$ is not adjacent to any other vertex of $V(C_Q) - \{w_0\}$, since $G$ has girth 5.
	As there are no 1-chords across $v$ per Theorem \ref{thm_len_1_chords}, there are two ways that $G$ may exhibit a $t-x$ geodesic.
	
\begin{figure}[h]
\centering
\begin{tikzpicture}
[scale =0.7, inner sep=0.8mm, 
vertex/.style={circle,thick,draw},
dvertex/.style={rectangle,thick,draw, inner sep=1.3mm}, 
thickedge/.style={line width=1.5pt}] 

\node[vertex] (v) at (0,0)  [label = 270:{$v$}] {};

\node[vertex] (u0) at (180:2) [label = 270:{$u_0$}] {};
\node[vertex] (uj) at (140:2) [label = 225:{$u_j$}] {};
\node[vertex] (uj1) at (120:2) {};
\node[vertex] (uj2) at (100:2) [label = 315:{$u_{j+2}$}] {};
\node[vertex] (uj3) at (50:2) {};
\node[vertex] (uj4) at (25:2) {};
\node[vertex] (uj5) at (0:2) [label = 270:{$u_{j+5}$}] {};

\node[vertex] (w0) at (180:4) [label = 270:{$w_0$}] {};
\node[vertex] (wj) at (140:4) [label = 0:{$w_j$}] {};
\node[vertex] (wj1) at (120:4) {};
\node[vertex] (wj2) at (100:4) [label = 190:{$s$}] {};
\node[vertex] (wj3) at (50:4) [label = 270:{$t$}] {};
\node[vertex] (wj4) at (25:4) {};
\node[vertex] (wj5) at (0:4) [label = 270:{$w_{j+5}$}] {};

\node[vertex] (x) at (160:4) [label = 0:{$x$}] {};

\node[vertex] (a) at (160:5.5) [label = 340:{$a$}] {};

\node at (160:6) {$Q$};

\draw[thickedge] (w0)--(u0)--(v)--(uj)--(wj);
\draw[thickedge] (w0) .. controls (170:5.2) .. (a) .. controls (150:5.2) .. (wj);

\draw (w0)--(x);
\draw (v)--(uj1) (v)--(uj2) (v)--(uj3) (v)--(uj4) (v)--(uj5);
\draw (uj1)--(wj1) (uj2)--(wj2) (uj3)--(wj3) (uj4)--(wj4) (uj5)--(wj5);
\draw (wj2)--(wj3);

\draw (a) .. controls (130:6.4) and (85:6) .. (wj3);

\node at (0,-1) {Case 1.1};

\begin{scope}[shift={(11,0)}]
	\node[vertex] (v) at (0,0)  [label = 270:{$v$}] {};
	
	\node[vertex] (u0) at (180:2) [label = 270:{$u_0$}] {};
	\node[vertex] (uj) at (140:2) [label = 225:{$u_j$}] {};
	\node[vertex] (uj1) at (120:2) {};
	\node[vertex] (uj2) at (100:2) [label = 315:{$u_{j+2}$}] {};
	\node[vertex] (uj3) at (50:2) {};
	\node[vertex] (uj4) at (25:2) {};
	\node[vertex] (uj5) at (0:2) [label = 270:{$u_{j+5}$}] {};
	
	\node[vertex] (w0) at (180:4) [label = 270:{$w_0$}] {};
	\node[vertex] (wj) at (140:4) [label = 90:{$w_j$}] {};
	\node[vertex] (wj1) at (120:4) [label = 0:{$y$}] {};
	\node[vertex] (wj2) at (100:4) [label = 190:{$s$}] {};
	\node[vertex] (wj3) at (50:4) [label = 270:{$t$}] {};
	\node[vertex] (wj4) at (25:4) {};
	\node[vertex] (wj5) at (0:4) [label = 270:{$z$}] {};
	
	\node[vertex] (x) at (160:4) [label = 0:{$x$}] {};
	
	\node[vertex] (a) at (160:5.0) [label = 340:{$a$}] {};
	\node[vertex] (b) at (130:6) [label = 270:{$b$}] {};
	
	
	\draw[thickedge] (w0)--(u0)--(v)--(uj)--(wj);
	\draw[thickedge] (w0) .. controls (170:4.8) .. (a) .. controls (150:4.8) .. (wj);
	
	\draw (w0)--(x);
	\draw (v)--(uj1) (v)--(uj2) (v)--(uj3) (v)--(uj4) (v)--(uj5);
	\draw (uj1)--(wj1) (uj2)--(wj2) (uj3)--(wj3) (uj4)--(wj4) (uj5)--(wj5);
	\draw (wj2)--(wj3);
	
	\draw (w0) .. controls (170:6.5) and (150:6) .. (b) .. controls (95:5.4) and (75:5.4) .. (wj3);
	
	\node at (0,-1) {Case 1.2};
\end{scope}

\end{tikzpicture}
\caption{There are two possibilities in Case 1, either $t, a, w_0, x$ is a $t-x$ path, as in subcase 1.1, or $t, b, w_0, x$ is, as in subcase 1.2.}
\label{fig_2_chord_onevert}
\end{figure}
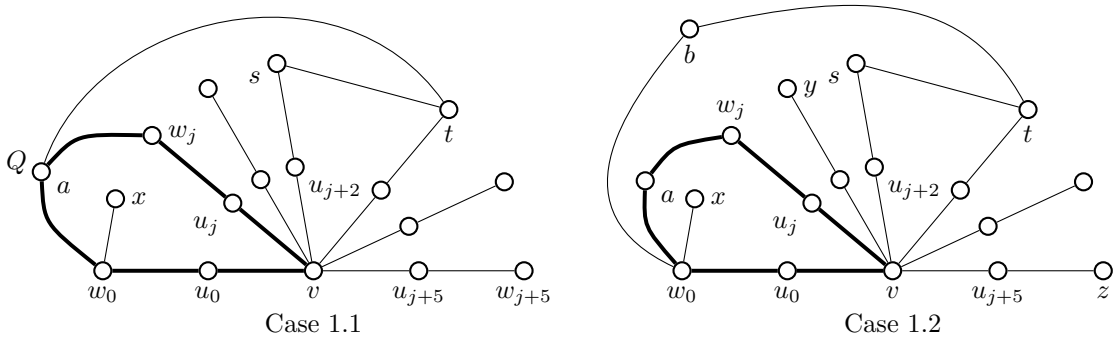
	
	\textit{Case 1.1:} The vertices $t$ and $a$ are adjacent.\\
	This case yields the path $t, a, w_0, x$ (see Figure \ref{fig_2_chord_onevert}).
	Since $G$ has girth 5, $s$ is not adjacent to $a$, $b$ or $w_j$. 
	Because there are no 1-chords across $v$ by Theorem \ref{thm_len_1_chords}, $t$ is not adjacent to $w_0$ (no neighbor of $s$ is adjacent to $w_0$).
	Thus $d(s,x) > 3$.
	
	\textit{Case 1.2:} There is a vertex $b$ in $N_3(v)$ that is adjacent to both $w_0$ and $t$.\\
	We have the $t-x$ geodesic $t, b, w_0, x$ (see Figure \ref{fig_2_chord_onevert}).
	There are two possibilities for an $s-x$ path of length at most 3.\\
	(1) either $s$ and $a$ are adjacent, or\\
	(2) there is some vertex $c$ in $N_3(v)$ that is adjacent to both $s$ and $w_0$.\\
	In either case, letting $y$ and $z$ be vertices of $N_1(u_{j+1}) \cap N_2(v)$ and $N_1(u_{j+5}) \cap N_2(v)$, respectively.
	Observe that $d(y,z) > 3$, completing Case 1.
	
	\textit{Case 2:} Neither $w_0$ nor $w_j$ has a neighbor in $\Int(C_Q)$.\\
	We claim that both $u_0$ and $u_j$ have neighbors in $\Int(C_Q)$.
	Assume to the contrary and without loss of generality that $u_0$ has no neighbor in $\Int(C_Q)$.
	Since $w_0$ has no neighbor in $\Int(C_Q)$, the path $a, w_0, u_0, v$ lies on the boundary of some face $f$ in $\Int(C_Q)$.
	Since $f$ is bounded by a 5-cycle, there is some vertex $z$ that is adjacent to both $a$ and $v$.
	Thus $v, z, a$ is a $v-a$ path of length 2, which contradicts the fact that $Q: w_0, a, w_j$ is a 2-chord (i.e., $a$ is in $N_3(v)$). 
	Hence there exist vertices $x$ and $y$ in $\Int(C_Q)$ that are adjacent to $u_0$ and $u_j$ respectively.
	Since $G$ contains no 4-cycles, $x \neq y$, and neither $x$ nor $y$ is adjacent to $a$.
	The face $f_{j+2}$ is bounded by a 5-cycle $v, u_{j+2}, s, t, u_{j+3}$, where $s$ and $t$ are vertices of $N_2(v)$.
	Because there are no 1-chords across $v$ (per Theorem \ref{thm_len_1_chords}), $s$ is not adjacent to any vertex of $N_2(v) \cap N_1(u_0)$ or $N_2(v) \cap N_1(u_j)$.
	As $d(s,x) \leq 3$, $s$ is adjacent to $a$ (and there is some vertex adjacent to both $a$ and $x$).
	Similarly, since $d(t,x) \leq 3$, $t$ is adjacent to $a$.
	However, we have a triangle on $a$, $s$ and $t$, a contradiction that completes the proof.
\end{proof}

\begin{thm}
	There does not exist a pentagulation with diameter 3, girth 5 and maximum degree greater than or equal to 8.
	\label{thm_girth5_DNE}
\end{thm}

\begin{proof}
	Assume to the contrary that $G$ is a pentagulation of girth 5, diameter 3 and maximum degree $\Delta \geq 8$.
	Let $v$ be a vertex of $G$ with maximum degree, and label the neighbors $u_1$, $u_2$, \dots, $u_{\Delta}$ of $v$ such that each path $u_i, v, u_{i+1}$ lies on the boundary of a face (subscripts taken mod $\Delta$).
	By Lemma \ref{lem_nhood_rules}, for each $i$ in $\{1,2,\dots, \Delta \}$, there is a vertex $w_i$ in $N(u_i) \cap N_2(v)$. 
	Note that each vertex $w_i$ is not adjacent to $u_j$ for any $j\neq i$, and $d(w_0, w_4) \leq 3$.
	
	We claim that any $w_0-w_4$ geodesic $Q$ is a 3-chord across $v$, i.e., the path $Q$ is of the form $w_0, a, b, w_4$, where $a$ and $b$ are vertices of $N_3(v)$.
	By Theorem \ref{thm_len_1_chords}, there are no 1-chords across $v$, so $w_0$ and $w_4$ are not adjacent.
	Similarly, there are no 2-chords across $v$ per Theorem \ref{thm_len_2_chords}, so $Q$ is not of the form $w_0, c, w_4$, where $c$ is some vertex of $N_3(v)$.
	The vertex $v$ is not in $Q$, since $Q$ has length at most 3 and $d(v, x_0) = d(v, x_4) = 2$.
	The path $Q$ does not contain any vertex of $N(v)$: 
	If $Q$ contains a vertex $u_j$ of $N(v)$, and $Q$ had length 2, then $Q$ is of the form $Q: w_0, u_j, w_4$, which is impossible.
	If $Q$ contains $u_j$ and has length 3, it is either of the form $w_0, u_j, x, w_4$ or $w_0, x, u_j, w_4$, where $x$ is some vertex of $N_2(v)$.
	But then either $xw_4$ or $w_0x$ is a 1-chord across $v$, which is impossible, so $V(Q) \cap N(v) = \emptyset$.
	To complete the proof of the claim, it suffices to show that $V(Q) \cap N_2(v) = \{w_0, w_4\}$.
	Assume to the contrary that there is a vertex $x$ of $Q$, that is not $w_0$ or $w_4$, in $N_2(v)$.
	If $Q$ has length 2, then it is of the form $w_0, x, w_4$.
	Since there are no 1-chords across $v$, $x$ is adjacent to $u_1$ or $u_{\Delta - 1}$, so $xw_4$ is a 1-chord across $v$, a contradiction.
	If $Q$ has length 3, then it is either $w_0, x, y, w_4$ or $w_0, y, x, w_4$, where $y$ is a vertex of $N_2(v)$ ($y$ is not in $N_3(v)$, since there are no 2-chords across $v$).
	By symmetry, we may assume without loss of generality that $Q: w_0, x, y, w_4$.
	Since there are no 1-chords across $v$, $x$ is a neighbor of $u_1$ or $u_{\Delta - 1}$, and $y$ is a neighbor of $u_3$ or $u_5$. 
	In all possible cases, $xy$ is a 1-chord across $v$, which proves the claim.
	
	The cycle $C_Q: w_0, a, b, w_4, u_4, v, u_0$ under $Q: w_0, a, b, w_4$ is a separating cycle that dominates either its interior or exterior. 
	Thus either $w_2$ or $w_6$ is adjacent to a vertex of $C_Q$.
	Suppose $w_2$ is adjacent to a vertex of $C_Q$ (the proof for $w_6$ is identical).
	As $G$ has girth 5, $w_2$ is not adjacent to any of $u_0$, $v$ or $u_4$.
	Because $G$ contains no 1-chords across $v$, $w_2$ is not adjacent to either $w_0$ or $w_4$.
	Thus $w_2$ is adjacent to $a$ or $b$.
	If $w_2$ is adjacent to $a$, then $w_2, a, w_0$ is a 2-chord across $v$, and if $w_2$ is adjacent to $b$, then $w_2, b, w_4$ is a 2-chord.
	In either case we obtain a contradiction, completing the proof.
\end{proof}

The main result follows immediately from Corollary \ref{cor:5_3_notriangle}, Theorem \ref{thm_girth5_DNE} and Theorem \ref{thm:5_one4cycle}.

\begin{thm}
	Let $G$ be a pentagulation of diameter 3, order $n$ and maximum degree $\Delta \geq 8$.
	The order of $G$ satisfies $n \leq 3\Delta - 1$.
	\label{thm_final_ddp_pentag}
\end{thm}

The bound in Theorem \ref{thm_final_ddp_pentag} is sharp for odd $\Delta$.
Consider the graph $\mathcal{H}$ in Figure \ref{fig:5_HI}.
We create a graph $G(\Delta)$ of maximum degree $\Delta = 2k+1$ from $\mathcal{H}$ as follows:
replace each white-vertex path of length 3 by a collection of internally disjoint paths: $k$ paths of length 3 and $k-1$ paths of length 2 (so $\mathcal{H}$ itself is $G(3)$).
It's easy to check that $G(\Delta)$ can be embedded such that each face is bounded by a 5-cycle, and that it has diameter 3, maximum degree $\Delta$ and $n = 3\Delta - 1$ vertices.

\section{Conclusion}
\label{sec:conclusion}

Theorem \ref{thm_final_ddp_pentag} and the sharpness example below it largely solve the degree-diameter problem for diameter 3 pentagulations. 
Between Theorem \ref{thm_final_ddp_pentag} and the results of \cite{ddpmpbg_dalfo_16, pglf_dupreez_21, mpgd2_seyffarth_89}, the degree-diameter problem has been solved exactly for all plane graphs of diameter 3 in which all faces are bounded by cycles of the same length.
A rough summary of the upper bounds is given in Table \ref{tab:summary}.
\renewcommand{\arraystretch}{1.3}

\begin{table}[ht]
	\centering
	\begin{tabular}{|c|c|c|c|c|c|}
		\hline
		& $\rho = 3$                  & $\rho = 4$              & $\rho = 5$                  & $\rho = 6$    & $\rho = 7$ \\ \hline
		$d=2$ & $\frac{3}{2}\Delta + 1$$^*$ & $\Delta + 2$            & 5                           & ---           & ---        \\ \hline
		$d=3$ & unknown                     & $3\Delta - 1$$^\dagger$ & $3\Delta - 1$$^*$$^\dagger$ & $2\Delta + 2$ & 7          \\ \hline
	\end{tabular}
	\caption{Table of maximum orders $n(\Delta, d)$ among plane graphs in which each face is bounded by a cycle of length $\rho$. Bounds with an asterisk $*$ are sharp for $\Delta$ odd, others are always sharp. Bounds with a dagger $\dagger$ are sharp only for $\Delta \geq 8$.}
	\label{tab:summary}
\end{table}

For large diameter, getting exact bounds is both difficult and tedious.
The last likely tractable exact bound still unknown is the bound for diameter 3 triangulations ($d = \rho = 3$). 
We end with some further problems:
\begin{itemize}[topsep=-\parskip]
	\item What is the maximum order of a diameter 3 triangulation?
	\item Let $\mu$ denote the size of the smallest face of a plane graph. What is the smallest function $\mu(d)$ such that every plane graph of diameter $d$ and smallest face size $\mu(d)$ has order $\mathcal{O}(\Delta)$?
	\item Find bounds on $n(\Delta, d)$ in plane graphs where every face has the same size $\rho$, or where every face has at least minimum size $\mu$.
\end{itemize}

\section*{Acknowledgments}
This research was supported by the South African NRF, Grant number 120104.
Thank-you to David Erwin for valuable discussion.

\bibliographystyle{plain}
\bibliography{ddp5}{}

\end{document}